\documentclass[10pt]{amsart}
  \usepackage{amsmath,amssymb,amsthm,amsfonts,graphicx}
  \graphicspath{{../../drawings/}}
  \usepackage[dvipsnames]{xcolor}

  \definecolor{limegreen}{rgb}{0.196,0.804,0.196}
  \definecolor{darkgreen}{rgb}{0.0,0.5,0.0}
  \definecolor{darkbluegreen}{rgb}{0,0.3,0.6}
  \definecolor{badgerred}{rgb}{0.715,0.004,0.004}
  \usepackage[font=small,labelfont={bf,sf}, textfont={sf},margin=0em]{caption}
  \usepackage{url,hyperref}
  \hypersetup{colorlinks,
  citecolor=badgerred,
  filecolor=black,
  linkcolor=blue,
  urlcolor=darkgreen,
  bookmarksopen=false,
  pdftex}
  \usepackage[notref,notcite]{showkeys}  
  \newcommand{\bsp}{\begin{split}}
  \newcommand{\esp}{{\end{split}}}
  \newcommand{\ds}{\displaystyle}
    \newcommand{\hilb}{\mathfrak{H}}
      \newcommand{\cH}{\mathfrak{H}}
   \newcommand{\be}{\begin{equation}}
     
       \newcommand{\tit}{{\tilde{\tau}}}
  \newcommand{\ee}{\end{equation}}
      \newcommand{\bsi}{{\bar \sigma}}
  \newcommand{\bee}{\begin{equation*}}
  \newcommand{\eee}{\end{equation*}}
  
  \newcommand{\cE}{\mathcal{E}}
  
   \newcommand{\bpsi}{{\Psi_2}}

 \newcommand{\Rm}{\mathrm{Rm}}

  \newcommand{\cD}{\mathfrak{D}}
  \newcommand{\cB}{\mathcal{B}}
 \newcommand{\pdd}[1]{\frac\partial{\partial #1}}
  \newcommand{\cC}{\mathcal{C}}
  
  \newcommand{\cL}{\mathcal{L}}

  \newcommand{\vft}{\varphi_T}
    \newcommand{\vf}{\varphi}

     \newcommand{\hZ}{Z}  

  \newcommand{\hv}{\mathfrak{D}}
  
  \newcommand{\pr}{\mathcal{P}}      
  \newcommand{\cyl}{\mathcal{C}}     
  \newcommand{\collar}{\mathcal{K}}  
\newcommand{\tip}{\mathcal{T}}     
\newcommand{\sk}{\smallskip} 
  \newcommand{\pd}{\partial}

  \newcommand{\dist}{{\mathrm{dist}}}
    \newcommand{\cQ}{\mathcal{Q}}

  \newcommand{\ud}{z}

  \newcommand{\emu}{{e^\mu du}}
  \newcommand{\R}{{\mathbb R}}
  \newcommand{\supp}{\mathop{\mathrm {supp}}}
  
  \newtheorem{theorem}{Theorem}[section]
  \newtheorem{proposition}[theorem]{Proposition}
  \newtheorem{lemma}[theorem]{Lemma}
  \newtheorem{definition}[theorem]{Definition}
  \newtheorem{prop}[theorem]{Proposition}
  \newtheorem{corollary}[theorem]{Corollary}
  \newtheorem{conjecture}[theorem]{Conjecture}
  \theoremstyle{remark}
  
  \newtheorem{remark}[theorem]{Remark}
  
  \newtheorem{claim}[theorem]{Claim}

  \newtheorem{step}{Step}

  \numberwithin{equation}{section}
  \numberwithin{theorem}{section}

\begin{document}
\title[Uniqueness (\today)]
{Uniqueness   of ancient compact non-collapsed  solutions to the 3-dimensional Ricci flow} 
\date{\today}
\author[Daskalopoulos]{Panagiota Daskalopoulos}
\address{Department of Mathematics, Columbia University, New York}
\author[Sesum]{Natasa Sesum}
\address{Department of Mathematics, Rutgers University, New Jersey}

\thanks{
P.
Daskalopoulos thanks the NSF for support in DMS-1266172.
N.
Sesum thanks the NSF for support in DMS-1811833.
}

\begin{abstract}
  In this paper we study the classification of compact $\kappa$-noncollapsed ancient solutions to the 3-dimensional Ricci flow 
which are rotationally and   reflection symmetric.   We prove that any such solution is isometric to  the  sphere 
    or   the type II 
  ancient solution    constructed by G. Perelman in \cite{Pe1}. 
\end{abstract}
\maketitle
\tableofcontents

\section{Introduction} 
Consider an ancient compact 3-dimensional solution to the Ricci flow
\begin{equation}
  \label{eq-rf}
  \frac{\partial}{\partial}g_{ij} = -2R_{ij}
\end{equation}
existing for $t\in (-\infty,0)$ so that it shrinks to a round point at $T$. The goal in this work is to provide the classification of such
solutions under natural geometric assumptions. 

\sk

Ancient  compact solutions to the 2-dimensional Ricci flow were classified by Daskalopoulos, Hamilton and Sesum in \cite{DHS}. 
 It turns out that in this case,  the complete list contains (up to conformal invariance) only the shrinking sphere solitons and the King  solution.  The latter 
is a  well known example  of ancient {\em collapsed} Ricci flow solution and can be written in closed form.  It was first discovered by J. King \cite{K1} in the context of
the logarithmic fast-diffusion equation on $\R^2$   and later independently by Rosenau \cite{R}  in the same context. 
It also appears as the {\em sausage model}   in the context of quantum field theory,  in the independent work of   Fateev-Onofri-Zamolodchikov \cite{FOZ}.
Although the King ancient solution is not   a soliton,  it may be  visualized as two steady solitons, called ``cigars'', coming from opposite spatial infinities glued together. 
Let us remark that the classification work in \cite{DHS}   classifies both collapsed and non-collapsed solutions. 

\sk
\sk 
In \cite{Ni}, Lei Ni  showed that any $\kappa$-noncollapsed ancient
solution to the Ricci flow which is of Type I and has positive
curvature operator has constant sectional curvature. In \cite{BHS}, Brendle, Huisken and Sinestrari  proved  that any ancient solution to the Ricci flow in dimension $n \ge 3$ which satisfies a suitable curvature pinching condition must have
constant sectional curvature. Fateev's examples (in \cite{Fa2}) which are collapsed show that  the pinching curvature condition in \cite{BHS} can not be removed. They also show the classification of closed ancient solutions even in dimension three, if we do not assume noncollapsedness may be very difficult, if not impossible. In \cite{BKN}, Bakas, Kong and Ni   construct several higher dimensional examples of type I ancient closed solutions to the Ricci flow which are non-collapsed and with positive sectional curvature. Observe that Perelman's solution is of positive curvature
operator, non-collapsed, but of type II.

\sk
Regarding the classification of ancient solutions to other geometric flows, let us mention  related work in the mean curvature flow setting. In \cite{ADS2} the authors showed   that every closed, {\em uniformly 2-convex}  and {\em non-collapsed}  ancient solution to the mean curvature flow must be either the family of contracting spheres or the unique, up to isometries, ancient oval constructed by White in \cite{Wh} and later by Haslhofer and Hershkovits in \cite{HO}. On the other hand, ancient noncompact non-collapsed uniformly 2-convex solutions were considered by Brendle and Choi in \cite{BC} and \cite{BC1}, where the authors showed  that  any noncollapsed, uniformly 2-convex noncompact ancient solution   to the mean curvature flow must be the   rotationally symmetric translating soliton, and hence the Bowl soliton, up to scaling and isometries.  Ancient compact {\em collapsed}  mean curvature flow solutions were studied in a recent interesting work  by Bourni, Langford and Tignalia in \cite{BLT}.  

\sk 
\sk 

Let us now turn our attention to the  3-dimensional Ricci flow. In \cite{Pe1}, G.  Perelman established the existence of 
  a rotationally symmetric ancient $\kappa$-noncollapsed solution on $S^3$ which is not a soliton. This is a   type II ancient solution backward in time, namely  its
  scalar curvature satisfies   $\sup_{M\times (-\infty,0)} |t| |R(x,t)| = \infty$ and forms a type I singularity forward in time, since it shrinks to a round point.  Perelman's ancient solution has backward in time limits  which are the Bryant soliton and the round cylinder $S^2\times \mathbb{R}$, depending on how the sequence of points and times about which one rescales are  chosen.  These are the only backward in time limits of the Perelman  ancient solution.
Let us remark that in contrast to the {\em collapsed } King ancient solution of the 2-dimensional Ricci flow, the Parelman ancient  solution is
{\em noncollapsed}. In fact there exist other ancient compact solutions to the 3-dimensional Ricci flow which are {\em  collapsed}  and are the analogues of the King solution  (see in  \cite{Fa2}, \cite{BKN}). 

\sk
\sk 

In \cite{Pe1}, Perelman introduced the following notion of {\em $\kappa$-noncollapsed} metrics.  

\begin{definition}[$\kappa$-noncollapsed property] 
The  metric $g$ is called  $\kappa$-noncollapsed  on the scale $\rho$, if every metric ball $B_r$  of radius $ r < \rho$  which satisfies 
$|Rm | \leq r^{-2}$  on $B_r$ has volume at least $\kappa \, r^n$. An ancient Ricci flow solution is called  $\kappa$-noncollapsed, if it is
 $\kappa$-noncollapsed on all scales $\rho$, for some $\kappa >0$.  
\end{definition} 

It turns out that this is an important notion in the context of ancient solutions and singularities.  In fact,  in \cite{Pe1}  Perelman proved that every ancient solution arising as a blow-up limit at a singularity of the Ricci flow on compact manifolds is $\kappa$-noncollapsed on all scales for some 
 $\kappa >0$. We have the following conjecture made by Perelman.

\begin{conjecture}[Perelman]
  \label{con-perelman}
  Let $(S^3,g(t))$ be a compact, ancient $\kappa$-noncollapsed solution to the Ricci flow \eqref{eq-rf} on $S^3$.  Then $g(t)$ is either a family of contracting spheres or Perelman's solution.
\end{conjecture}

The well known Hamilton-Ivey pinching estimate tells us that any two or three dimensional Ricci flow ancient solution,  with bounded curvature at each time slice,  has nonnegative sectional curvature.  Since our solution $(S^3,g(t))$ is closed, the strong maximum principle implies that  the sectional curvatures,  and hence the entire curvature operator,  are strictly positive.  It follows by Hamilton's Harnack estimate (see in \cite{Ha1})   that $R_t \ge 0$, yielding the  existence of
 a uniform constant $C > 0$ so that $R(\cdot,t) \le C$, for all $t\in (-\infty,t_0]$.  Since the curvature is positive, one concludes that 
\begin{equation}
  \label{eq-curv-bound}
  \|\Rm\|_{g(t)} \le C, \qquad \mbox{for all} \,\,\,  -\infty < t \le t_0,
\end{equation}
for a  uniform constant $C$. The above  discussion yields  that any closed 3-dimensional $\kappa$-noncollapsed ancient solution is actually a {\em $\kappa$-solution},
 in the sense that was defined by Perelman in \cite{Pe1}.

\sk 
In a recent important paper by S. Brendle (\cite{Br}), the author proved that a 3-dimensional non-compact ancient $\kappa$-solution is isometric to either a family of shrinking cylinders or their quotients, or to the Bryant soliton. The author first shows that all 3-dimensional ancient $\kappa$-solutions which are non-compact have to be rotationally symmetric. After that he shows that such a rotationally symmetric solution, if not a cylinder or its quotient, must be a steady Ricci soliton and hence the Bryant soliton by one of his earlier works (\cite{Br1}) about classification of steady Ricci solitons. 
\smallskip

The techniques of Brendle in \cite{Br} can be also applied to show the rotation symmetry of ancient compact and $\kappa$-noncollapsed solution to the 
Ricci flow \eqref{eq-rf} on $S^3$. Brendle has recently shown this in  \cite{Br2}. Bamler and Kleiner in \cite{BK},  obtained the same result as Brendle in the compact case, using different methods. However, since the rotationally symmetric solutions discovered by Perelman are not solitons, the classification
of rotationally symmetric ancient compact and $\kappa$-noncollapsed solutions is a difficult problem. 
Our goal in this work is to establish this classification under the additional assumption of reflection symmetry. In an upcoming  work we plan to remove this technical assumption.  Our main result states as follows. 

\begin{theorem}\label{thm-main-main}   Let $(S^3, g(t))$ be a  compact, $\kappa$-noncollased ancient solution to the Ricci flow on $S^3$
which is  symmetric with respect to rotation and reflection. Then $g(t)$ is either a family of contracting spheres or Perelman's solution.
\end{theorem} 

Combining Theorem \ref{thm-main-main} and recent results in \cite{Br2} immediately yield the following result.

\begin{theorem}
Let $(S^3, g(t))$ be a  compact, $\kappa$-noncollased ancient solution to the Ricci flow on $S^3$, which is symmetric with respect to reflection. Then $g(t)$ is either a family of contracting spheres or Perelman's solution.
\end{theorem}

Assume from now on that  $(S^3,g(t))$ is a Ricci flow solution which satisfies the assumptions of Theorem \ref{thm-main-main}. We will next see how one can express the
Ricci flow under rotational symmetry as a single  equation.  Let us  first remark  that by the work of Perelman we know that  the asymptotic soliton of $(S^3,g(t))$  is either a round cylinder or a sphere. We can understand this that every $\kappa$-solution has a gradient shrinking soliton buried inside of it, in an asymptotic sense as time approaches $-\infty$ (for more details on asymptotic solitons see \cite{Pe1}). In our recent work \cite{ADS3} we show that if the asymptotic soliton is the sphere, then the solution $(S^3,g(t))$ must be the  round sphere itself. Hence, from now on {\em we may assume  
that  the asymptotic soliton of our closed $\kappa$-noncollapsed solution is the  round cylinder $S^2\times\R$}.  

\sk

Since at each time slice, the metric is $SO(3)$-invariant, it can can be
written as
\[
g = \phi^2\, dx^2 + \psi^2\, g_{can}, \qquad \mbox{on} \,\, (-1,1)\times S^2
\]
where $(-1,1)\times S^2$ may be naturally identified with the sphere $S^3$
with its north and south poles removed. The  function $\psi(x,t) > 0$ may be
regarded as the radius of the hypersurface $\{x\}\times S^2$ at time $t$.  
The
distance function from the equator is given by
\[
s(x, t) = \int_0^x \phi(x', t)\, dx'.
\]
and abbreviating    $ds = \phi(x, t)\,  dx $, we  write our metric as
$
  g = ds^2 + \psi^2\, g_{can}. 
$
As it was remarked in  \cite{AK1}, for our metric \eqref{eq-metric} to define a smooth metric on $S^3$ we need  to have
$
  \psi_s(s_-(t)) = 1, \,  \psi_s(s_+(t)) = -1$ holding at the two tips  of our solution.  
Under the Ricci flow, the profile  function $\psi : (s_-(t), s_+(t)) \times(-\infty, 0)\to\R$ evolves by 
\begin{equation}
\label{eq-psi-in}
\psi_t = \psi_{ss} -  \frac{1-\psi_s^2}{\psi}.
\end{equation}

\sk

Consider next  a type I scaling of our metric, which leads to  the rescaled profile $u(\sigma,\tau)$ 
defined by 
\begin{equation}\label{eqn-defnu}
  u(\sigma, \tau) := \frac{\psi(s,t)}{\sqrt{-t}}, \qquad  \mbox{with}\,\, \sigma := \frac{s}{\sqrt{-t}}, \,\, \tau = -\log (-t).  
\end{equation}
A direct calculation  shows that  $u:(\sigma_-(\tau),\sigma_+(\tau))\times(-\infty, 0)\to\R$ satisfies the equation
\begin{equation}  \label{eq-u0}
  u_{\tau} = u_{\sigma\sigma} +  \frac{u_{\sigma}^2}{u}  -  \frac 1u + \frac u2
\end{equation}
with boundary conditions at the tips 
$u_\sigma(\sigma_-(\tau),\tau) = 1, \,  u_\sigma(\sigma_+(\tau),\tau) = -1$. 

\sk
\sk 

It follows from  the discussion above, since we know our solution is rotationally symmetric (\cite{BK}, \cite{Br2}), our main result, Theorem \ref{thm-main-main}, is equivalent to the  following uniqueness result.

\begin{theorem}
  \label{thm-main}
  Let $(S^3, g_1(t))$  and $(S^3, g_2(t))$, $-\infty < t < T$,   be two   compact  non-spherical  rotationally and reflection symmetric, $\kappa$-noncollased ancient solutions  to the 
  3-dimensional Ricci flow  which have  the same    axis  of symmetry and  whose profile functions 
  $\psi_1(s,t)$ and $\psi_2(s,t)$ satisfy equation \eqref{eq-psi-in}. 
  Then,  they are the same up to   translations in time and parabolic rescaling.  
  In particular, they coincide with the Perelman solution. 
\end{theorem}

A crucial first step  in showing Theorem \ref{thm-main} is to establish the (unique up to scaling) asymptotic behavior of any   compact  rotationally and reflection symmetric $\kappa$-noncollased ancient solution to the Ricci flow on $S^3$ which is not isometric to a sphere.   This was recently established by the authors in \cite{ADS3} and is summarized, for the reader's convenience,  in the next theorem. 

\begin{theorem}[Angenent, Daskalopoulos, Sesum in \cite{ADS3}]\label{thm-asym}
  Let $(S^3, g(t))$ be any reflection and rotationally symmetric compact  $\kappa$-noncollapsed ancient solution to the Ricci flow on $S^3$ which is not isometric to a round sphere.
  Then the rescaled profile $u(\sigma,\tau)$ solution to   \eqref{eq-u0}  has the following asymptotic expansions:
  \begin{enumerate}
    \item[(i)] For every $L > 0$,
    \[
    u(\sigma,\tau) = \sqrt{2} \Bigl(1 - \frac{\sigma^2 - 2}{8|\tau|}\Bigr) + o(|\tau|^{-1}), \qquad {\mbox on} \,\, |\sigma| \le L
    \]
    as $\tau \to -\infty$.
    \sk
    \item[(ii)] Define $z := {\sigma}/{\sqrt{|\tau|}}$ and $\bar{u}(\sigma,\tau) := u(z\sqrt{|\tau|}, \tau)$.
    Then,
    \[
    \lim_{\tau \to -\infty} \bar{u}(z,\tau) = \sqrt{2 - \frac{z^2}2}
    \]
    uniformly on compact subsets of $|z| < 2$.
    \sk
    
    \item[(iii)] Let $k(t) := R(p_t,t)$ be the maximal scalar curvature which is attained at each one of the two tips $p_t$,  for $t \ll -1$.  Then the rescaled Ricci flow solutions $(S^3,\bar g_{t}(s), p_{t})$, with ${\ds \bar {g}_{t}(\cdot, s) = k(t) \, g(\cdot,t+k(t)^{-1}\, s)}$, converge to the unique Bryant translating soliton with maximal scalar curvature one.
    Furthermore, $k(t)$ and the diameter $d(t)$ satisfy the asymptotics
    \[
    \qquad k(t) = \frac{\log|t|}{|t|} (1 + o(1))\, \quad \mbox{and} \quad d(t) = 4\sqrt{|t|\log |t|}\, (1 +o(1))
    \]
    as $t \to -\infty.$
  \end{enumerate}
\end{theorem} \noindent

\smallskip 

The outline of the paper is as follows. In section \ref{sec-outline} we give a detailed  outline of our proof, including all equations and norms that we consider in 
the two different regions, which are the cylindrical and the tip regions. In section \ref{sec-cylindrical} we study the linearized equation around the cylinder. We show that the norm of the projections of the  difference of our solutions onto unstable  and stable modes of the  linearized operator around the  cylinder is controlled by a tiny multiple of the norm of the projection of a difference of our solutions onto a neutral mode and a tiny multiple of the norm of a difference of our solutions outside the cylindrical region.  Section \ref{sec-tip} is devoted in the analysis of the tip region.  We  show how to define  a suitable {\em  weighted norm}  in the tip so that the norm of the difference of our solutions in the tip region is controlled by a tiny multiple of the norm of the  difference of our solutions inside the cylindrical region. In Section \ref{sec-conclusion}
we give the proof of Theorem \ref{thm-main}  using the results from previous two sections and carefully analyzing the error terms that appear when we approximate 
the Ricci flow equation by its linearization at the cylinder (in the cylindrical region)  or the  Bryant soliton (in the soliton region).  In our final Section appendix \ref{sec-appendix},  we include needed a'priori estimates whose proofs are somewhat similar to the proofs of analogous a'priori estimates in \cite{ADS2} but still different enough to have their proofs included for the sake of completeness of our arguments.

\section{Outline of the proof of Theorem \ref{thm-main}} \label{sec-outline}

Since the proof of Theorem \ref{thm-main} is quite involved,
in this preliminary section we give an outline of its  main steps.
Our method is based on a'priori estimates for the ``distance''   between two given ancient solutions which is 
measured in appropriate weighted $L^2$-norms.
We need to consider two different regions: the {\em cylindrical}  region and the {\em tip} region. In each of these regions, 
our distance norms are dictated by the  behavior of our solutions. 
In what follows, we define these regions, review the equations in each region and define the  weighted $L^2$-norms with respect to which we will prove coercive type estimates in the subsequent sections.
At the end of the section we will give an outline of the proof of Theorem \ref{thm-main}.

\smallskip 
\subsection{Equations under rotational symmetry} We have seen in the introduction that  a solution $g$ of \eqref{eq-rf}
on $S^{3}$ which is $SO(3)$-invariant, can be written as
\[
g = \phi^2\, dx^2 + \psi^2\, g_{can}, \qquad \mbox{on} \,\, (-1,1)\times S^2
\]
where $(-1,1)\times S^n$ may be naturally identified with the sphere $S^{3}$
with its North and South poles removed.  The  function $\psi(x,t) > 0$ may be
regarded as the radius of the hypersurface $\{x\}\times S^2$ at time $t$.  
By the reflection symmetry assumption we may assume that $\phi(x,t) =\phi(-x,t)$  and $\psi(x,t)=\psi(-x,t)$, for all $x\in (-1,1)$, that is $x =0$ is the point of reflection symmetry for our profile functions.
Then the distance function to the center of symmetry (we will refer to it as to the equator)  is given by
\begin{equation}
\label{eq-x0}
s(x, t) = \int_{0}^x \phi(x', t)\, dx'.
\end{equation}
We will write
\[
s_\pm(t) := \lim_{x\to \pm 1} s(x, t), 
\]
or shortly $s_{\pm}$, for the distance from the equator to the South and the North poles,
respectively, which depend on time, along the Ricci flow.  If we abbreviate
\[
ds = \phi(x, t) \, dx \qquad  \text{ and } \qquad 
\frac{\pd}{\pd s} = \frac 1{\phi(x, t)} \frac{\pd}{\pd x}
\]
then we can write our metric as
\begin{equation}\label{eq-metric}
  g = ds^2 + \psi^2\, g_{can}. 
\end{equation}

\sk

Let us next  review how you derive the evolution equation of the profile function $\psi(s,t)$
from the Ricci flow equation. 
 The time derivative does not commute with the $s$-derivative, and in general we
must use
\[
\frac{\pd}{\pd t}ds = \phi_t \, dx = \frac{\phi_t}{\phi} \, ds \qquad 
\text{ and } \qquad 
\left[ \frac{\pd}{\pd t}, \frac{\pd}{\pd s}\right] =
-\frac{\phi_t }{\phi} \frac{\pd}{\pd s}.
\]
The Ricci tensor is given by
\begin{align*}
  {\mathrm{Rc}}
  & = 2K_0 \, ds^2 + \left[K_0 +  K_1\right] \psi^2 g_{\rm can} \\
  & = -2\frac{\psi_{ss}}{\psi }\, ds^2
  + \left\{
  -\psi \psi_{ss}-1)\psi_s ^2+1
  \right\} g_{\rm can}
\end{align*}
where $K_0$ and $K_1$ are the two distinguished sectional curvatures that any
metric of the form \eqref{eq-metric} has.  They are the curvature of a plane tangent to $\{s\}\times S^n$, given by
\begin{equation}
  \label{eq-K1}
  K_1:= \frac{1-\psi_s^2}{\psi^2},
\end{equation}
and the curvature of an orthogonal plane given by
\begin{equation}
  \label{eq-K0}
  K_0 := -\frac{\psi_{ss}}{\psi}.
\end{equation}
Moreover, the scalar curvature is given by
\[
R = g^{jk} R_{jk}
= 4K_0 + 2K_1.
\]
The time derivative of the metric is
\[
\frac{\pd g}{\pd t} =
2\frac{\phi_t}{\phi}\, ds^2 + 2\psi\psi_t\, g_{\rm can}.
\]
Therefore, if the metrics $g(t)$ evolve by Ricci flow $\pd_t g = -2\mathrm{Rc}$,
then
\[
\frac{\phi_t}{\phi} = 2\frac{\psi_{ss}}{\psi},
\]
so that
\[
\pdd t \, ds = 2\frac{\psi_{ss}}{\psi}\, ds \text{ and }
\left[\pdd t, \pdd s\right]
= -2\frac{\psi_{ss}}{\psi}\, \pdd s.
\]
Under Ricci flow the radius $\psi(s,t)$ satisfies the equation  
\begin{equation}
  \label{eq-RF}
  \psi_t = \psi_{ss} -  \frac{1-\psi_s^2}{\psi}.
\end{equation}
As in \cite{AK1}, for our metric \eqref{eq-metric} to define a smooth metric on $S^3$ we need to have
\begin{equation}
  \label{closing-psi}
  \psi_s(s_-) = 1, \,\,\,\, \psi^{(2k)}(s_-) = 0 \qquad \mbox{and} \qquad \psi_s(s_+) = -1, \,\,\,\, \psi^{(2k)}(s_+) = 0,
\end{equation}
for $k\in \mathbb{N}\cup\{0\}$.

\subsection{Ancient Ricci ovals and their invariances} 
In \cite{ADS3} we have discussed results that lead to the conclusion that the asymptotic soliton of any $SO(3)$-invariant closed $\kappa$-solution is either a sphere or a cylinder. In the case it is a sphere,  we have proved in \cite{ADS3} that  our solution has to be the family of shrinking round spheres itself. Hence, for the rest of the paper we may  assume that any closed $\kappa$-solution in consideration below has a  round cylinder as its asymptotic  cylinder. 

\begin{definition}
We define an Ancient Ricci Oval to be any rotationally symmetric closed $\kappa$-solution which has the round cylinder as its asymptotic soliton.
\end{definition}

\smallskip Let  $g_i = ds^2 + \psi_i^2 \, g_{can}$, for $i\in \{1,2\}$ be two reflection symmetric  Ancient Ricci Oval solutions satisfying the assumptions of Theorem \ref{thm-main}. In particular we have fixed   the axis of symmetry and the center of 
refection symmetry  ($s=0$) for both solutions.  Under the Ricci flow each profile $\psi_i(s,t)$ satisfies \eqref{eq-RF}. 
In the statement of our main Theorem \ref{thm-main} we claim the {\em uniqueness} of $g_1$ and $g_2$   up to parabolic scaling  in space-time and translations in time.
Since  each solution $g_i(t)$ gives rise to a two  parameter family of solutions
\begin{equation}
  \label{eq-two-parameters}
  g_i^{\beta \gamma}(\cdot,t) = e^{\gamma/2} \, g_i(\cdot,e^{-\gamma}(t - \beta))
\end{equation}
the  theorem claims the following: \emph{ given two ancient oval solutions we can find $\beta, \gamma$ and $t_0 \in \R$ such that
\[
g_1(\cdot,t) = g_2^{\beta \gamma}(\cdot,t), \qquad \mbox{for} \,\, t \leq t_0.
\]}
The profile function $\psi_2^{\beta \gamma}$ corresponding to the modified solution $g_2^{\beta \gamma}(\cdot,t)$ is given by
\begin{equation}
  \label{eq-Ualphabeta}
  \psi_2^{\beta \gamma} (s, t) = e^{\gamma/2} \psi_2 \Bigl (e^{-\gamma/2}\,s, e^{-\gamma} (t - \beta) \Bigr),
\end{equation}

where ${\ds s(x,t) = \int_0^x \phi(x',t)\, dx'}$.

We rescale the solutions $g_i(\cdot,t)$ by a factor $\sqrt{-t}$ and introduce a new time variable $\tau = -\log(-t)$, that is, we set
\begin{equation}
  \label{eq-type-1-blow-up}
  g_i(\cdot,t) = \sqrt{-t} \, \bar g_i(\cdot,\tau), \qquad \tau:= - \log (-t).
\end{equation}
These are again rotationally and reflection  symmetric with profile function $u$, which is related to $\psi$ by
\begin{equation}
  \label{eq-cv1}
  \psi(s,t) = \sqrt{-t}\, u(\sigma,\tau), \qquad \sigma=\frac s{ \sqrt{-t}}, \quad \tau=-\log (-t).
\end{equation}
If the $\psi_i$ satisfy the Ricci flow equation \eqref{eq-psi-in}, then the rescaled profiles $u_i$ satisfy
\begin{equation}  \label{eq-u}
  u_{\tau} = u_{\sigma\sigma} + \frac{u_{\sigma}^2}{u}  -  \frac 1u + \frac u2
\end{equation}
with boundary conditions at the tips 
$u_\sigma(\sigma_-) = 1, \,  u_\sigma(\sigma_+) = -1$. 

The vector fields  $\partial_{\tau}$ and $\partial_{\sigma}$ do not commute. However, we can make them commute by adding a non-local term in
equation \eqref{eq-u} (see  \cite{ADS3},  Section 2 for details).  In fact in   commuting variables  the equation for $u$ becomes 
\begin{equation}\label{eqn-u} 
 u_\tau = u_{\sigma\sigma} - \frac{\sigma}{2} \, u_{\sigma}
 -  J(\sigma,\tau) \, u_{\sigma}  +  \frac{u_{\sigma}^2}{u}  -  \frac 1u + \frac u2
\end{equation}
where
\begin{equation}\label{eqn-defn-J}
J(\sigma,\tau) = 2 \int_0^\sigma \frac{u_{\sigma\sigma}}{u} \, d\sigma'.
\end{equation}
 
\smallskip

Changing $g_i(\cdot,t)$ to $g_i^{\beta \gamma}(\cdot,t)$ has the following effect on $u_i(\sigma,\tau)$:
\begin{equation}
  \label{eq-ualphabeta}
  u_i^{\beta \gamma}(\sigma,\tau) = \sqrt{1+\beta e^{\tau}} u_i\Bigl( \frac{\sigma} {\sqrt{1+\beta e^\tau} }, \tau+\gamma-\log\bigl(1+\beta e^\tau\bigr) \Bigr).
\end{equation}

To prove the uniqueness theorem we will look at the difference $\psi_1 - \psi_2^{\beta \gamma}$, or equivalently at $u_1 - u_2^{\beta \gamma}$.
The parameters $ \beta, \gamma$  will be chosen so that the projections of $u_1 - u_2^{\beta \gamma}$ onto positive eigenspace (that is spanned by two independent eigenvectors) and zero eigenspace of the linearized operator $\cL$ at the cylinder are equal to zero at time $\tau_0$, which will be chosen sufficiently close to $-\infty$. 
Correspondingly, we denote the difference $\psi_1 - \psi_2^{\beta \gamma}$ by $\psi_1-\psi_2$ and $u_1 - u_2^{\beta \gamma}$ by $u_1-u_2$.  
We will actually observe that the parameters $\beta, \gamma$  can be chosen to lie in a certain range, which allows our main estimates to hold independently of the choice of these parameters during the proof, as long as they stay in the specified range.
We will show in Section \ref{sec-conclusion} that for a given small $\epsilon >0$ there exists $\tau_0 \ll -1 $ sufficiently negative for which we have
\begin{equation}
  \label{eq-asymp-par1}
 \beta \leq \epsilon \, \frac{e^{-\tau_0}}{|\tau_0|}   \qquad \mbox{and} \qquad \gamma \leq \epsilon \, |\tau_0|
\end{equation}
and our estimates hold for $(u_1 - u_2^{\beta \gamma})(\cdot,\tau)$, for all $\tau \leq \tau_0$, as long as $\beta$ and $\gamma$ satisfy \eqref{eq-asymp-par1}.
This inspires the following definition.
\begin{definition}[Admissible triple of parameters $(\beta,\gamma)$]
  \label{def-admissible}
  We say that the pair of  parameters $(\beta,\gamma)$ is admissible with respect to time $\tau_0$ if they satisfy \eqref{eq-asymp-par1}.
\end{definition}

Throughout the proof of Theorem \ref{thm-main} we will make sure hat our choice of parameters $(\beta,\gamma)$ 
satisfies the admissibility condition given by this definition.


\medskip

\subsection{The two regions:  equations, norms and crucial estimates}

The proof of Theorem \ref{thm-main} relies  on sharp coercive estimates in appropriate norms for the difference of our
two solutions $w:= u_1 - u_2^{\beta\gamma}$.  Since the behavior of our solutions changes from being a cylinder near the
equator to being the Bryant soliton at the two tips (see Theorem \ref{thm-asym}) we will need to consider these two regions 
separately.  Namely for a given small  positive constant $\theta$, we define   the {\em cylindrical region}   by
\[
\cyl_{\theta} = \bigl\{(\sigma,\tau) : u_1(\sigma,\tau) \ge \frac{\theta}{4} \bigr\}
\]
and the  {\em tip region } 
\[
\tip_{\theta} = \{(u, \tau):\,\, u_1 \le 2\theta, \, \tau \leq \tau_0\}.
\]

\sk

 We will next outline how we treat each region separately and obtain a coercive estimate for the difference of the two solutions in appropriate  weighted  norms.
 At the end of the outline we will show how these estimates imply unqueness.

\subsubsection{The cylindrical region}\label{subsec-cylindrical}
For a given $\tau \leq \tau_0$ and constant $\theta$ positive and small, consider the cylindrical region $\cyl_{\theta} = \bigl\{(\sigma,\tau) : u_1(\sigma,\tau) \ge  {\theta}/{4} \bigr\}$. 
Let $\varphi_\cyl(\sigma,\tau)$ denote a standard   \emph{cut-off function}    with the following properties:
\[
(i) \,\, \supp \varphi_\cyl \Subset \cyl_{\theta} \qquad (ii)\,\, 0 \leq \varphi_\cyl \leq 1 \qquad (iii) \,\, \varphi_\cyl \equiv 1 \mbox{ on } \cyl_{2\theta}.
\]

\sk 

The solutions $u_i$, $i=1,2$, satisfy equation \eqref{eq-u}.
Setting
\[
w:=u_1-u_2^{\beta \gamma} \qquad \text{ and } \qquad w_\cyl:= w\, \varphi_\cyl
\]
we see that $w_\cyl$ satisfies the equation
\begin{equation}\label{eqn-w100}
  \frac{\pd}{\pd\tau} w_\cyl = \cL[w_\cyl] + \cE[w,\varphi_\cyl]
\end{equation}
where the operator $\cL$ is given by
\[
\cL = \partial_\sigma^2 - \frac \sigma 2 \partial_\sigma + 1
\]
and where the error term $\cE$ is described in detail in Section~\ref{sec-cylindrical}.
We will see that
\[
\cE[w,\varphi_\cyl] = \cE(w_\cC) + \bar \cE[w,\varphi_\cyl] + \cE_{nl}
\]
where $\cE(w_\cC) $ is the error introduced due to the nonlinearity of our equation and is given by \eqref{eq-E10}, $\bar \cE[w,\varphi_\cyl]$ is the error introduced due to the cut off function $\varphi_\cC$ and is given by \eqref{eq-bar-E} (to simplify the notation we have set $u_2:=u_2^{\beta \gamma}$) and $ \cE_{nl}$ is a nonlocal error term and is given by \eqref{eq-nonlocal-error}.  \smallskip

The differential operator $\cL$ is a well studied self-adjoint operator on the Hilbert space $\hilb := L^2(\R, e^{-\sigma^2/4}d\sigma)$ with respect to the norm and inner product
\begin{equation}\label{eqn-normp0}
  \|f\|_{\hilb}^2 = \int_\R f(\sigma)^2 e^{-\sigma^2/4}\, d\sigma, \qquad \langle f, g \rangle = \int_\R f(\sigma) g(\sigma) e^{-\sigma^2/4}\, d\sigma.
\end{equation}
\smallskip

We split $\hilb$ into the unstable, neutral, and stable subspaces $\hilb_+$, $\hilb_0$, and $\hilb_-$, respectively.
The unstable subspace $\hilb_+$ is spanned by the  eigenfunction   $\psi_0 \equiv  1$ corresponding  to the only positive eigenvalue 1
(that is, $\hilb_+$ is one  dimensional, due to our assumption on reflection symmetry).
The neutral subspace $\hilb_0$ is the kernel of $\cL$, and is one dimensional space spanned by $\psi_2 = \sigma^2-2$.
The stable subspace $\hilb_-$ is spanned by all other eigenfunctions.
Let $\pr_\pm$ and $\pr_0$ be the orthogonal projections on $\hilb_\pm$ and $\hilb_0$.  \smallskip

For any function $f : \R \times (-\infty,\tau_0]$, we define the cylindrical norm
\[
\|f\|_{\hilb,\infty}(\tau) = \sup_{\tau' \le \tau} \Bigl(\int_{\tau' - 1}^{\tau'} \|f(\cdot,s)\|_{\hilb}^2\, ds\Bigr)^{\frac12}, \qquad \tau \leq \tau_0
\]
and we will often simply set
\begin{equation}
  \label{eqn-normp}
  \|f\|_{\hilb,\infty}:= \|f\|_{\hilb,\infty}(\tau_0).
\end{equation}
\medskip

In the course of proving necessary estimates in the cylindrical region we define yet another Hilbert space $\hv$ by
\[
\hv = \{f\in\hilb : f, f_\sigma \in \hilb\},
\]
equipped with a norm
\[
\|f\|_\hv^2 = \int_{\R} \{f(\sigma)^2 + f'(\sigma)^2\} e^{-\sigma^2/4}d\sigma.
\]
We will write
\[
( f, g )_\hv = \int_\R \{f'(\sigma)g'(\sigma) + f(\sigma)g(\sigma)\} e^{-\sigma^2/4} d\sigma,
\]
for the inner product in $ \hv $.

Since we have a dense inclusion $ \hv \subset \hilb $ we also get a dense inclusion $ \hilb \subset \hv^* $ where every $ f\in\hilb $ is interpreted as a functional on $ \hv $ via
\[
g\in \hv \mapsto \langle f, g \rangle.
\]
Because of this we will also denote the duality between $ \hv $ and $ \hv^* $ by
\[
(f, g) \in \hv\times \hv^* \mapsto \langle f, g \rangle .
\]
Similarly as above define the cylindrical norm
\begin{equation}
\label{eqn-cyl-norm}
\|f\|_{\hv,\infty}(\tau) = \sup_{\sigma\le\tau}\Bigl(\int_{\sigma-1}^{\sigma} \|f(\cdot,s)\|_{\hv}^2\, ds\Bigr)^{\frac12},
\end{equation}
and analogously we define the cylindrical norm $\|f\|_{\hv^*,\infty}(\tau)$ and set for simplicity $\| f\|_{\hv^*,\infty}:=\|f\|_{\hv^*,\infty}(\tau_0)$.

\smallskip In Section \ref{sec-cylindrical} we will show a coercive estimate for $w_\cyl$ in terms of the error $E[w, \varphi_\cyl]$.
However, as expected, this can only be achieved by removing the projection $\pr_0 w_\cyl$ onto the kernel of $\cL$, generated by $\psi_2$.
More precisely, setting
\[
\hat  w_\cyl := \pr_+ w_\cyl + \pr_- w_\cyl = w_\cyl - \pr_0 w_\cyl
\]
we will prove that for any given $\theta  \in (0, \sqrt{2})$ and $\epsilon >0$ there exist $\tau_0 \ll -1$ (depending on $\theta$ and 
$\epsilon$)  such that the following bound holds
\begin{equation}\label{eqn-cylindrical100}
  \| \hat w_\cyl \|_{\hv,\infty} \leq C \, \|E[w,\varphi_\cyl]\|_{\hv^*,\infty}
\end{equation}
provided that $\pr_+ w_\cyl(\tau_0) =0$.
In fact, we will show in Proposition \ref{lem-rescaling-components-zero} that the parameters $\beta$ and $\gamma$ can be adjusted so that for $w^{\beta \gamma}:= u_1 - u_2^{\beta \gamma}$, we have
\begin{equation}\label{eqn-projections}
  \pr_+ w_\cyl(\tau_0) =0 \qquad \mbox{and} \qquad \pr_0 w_\cyl(\tau_0) =0.
\end{equation}
Thus \eqref{eqn-cylindrical100} will hold for such a choice of $ \beta, \gamma$ and $\tau_0 \ll -1$.
The condition $\pr_0 w_\cyl(\tau_0) =0$ is essential and will be used in Section \ref{sec-conclusion} to give us that $w^{\beta \gamma} \equiv 0$.
In addition, we will show in Proposition \ref{lem-rescaling-components-zero} that $\beta$ and $\gamma$ can be chosen to be admissible according to our Definition \ref{def-admissible}.

The norm of the error term $\|E[w,\varphi_\cyl]\|_{\hv^*,\infty}$ on the right hand side of \eqref{eqn-cylindrical100} will be estimated in Section \ref{sec-cylindrical}.
We will show that given $\epsilon >0$ small, there exists a  $\tau_0 \ll -1$ such that
\begin{equation}\label{eqn-error-111}
  \|E[w,\varphi_\cyl]\|_{\hv^*,\infty} \leq \epsilon \, \big ( \| w_\cC \|_{\hv,\infty} + \| w\, \chi_{D_\theta}\|_{\hv, \infty} \big ).
\end{equation}
where $D_\theta:= \{ (\sigma,\tau): \,\, \theta/4 \leq u_1(\sigma,\tau) \leq \theta/2 \}$ denotes the support of the derivative of $\varphi_{\cC}$.
Combining \eqref{eqn-cylindrical100} and \eqref{eqn-error-111} yields the bound
\begin{equation}\label{eqn-cylindrical1}
  \| \hat w_\cyl \|_{\hv,\infty} \le \epsilon \, (\|w_\cyl\|_{\hv,\infty} + \| w\, \chi_{D_\theta} \|_{\hilb,\infty}),
\end{equation}
holding for all $\epsilon >0$ and $\tau_0 := \tau_0(\epsilon) \ll -1$.

\smallskip To close the argument we need to estimate $\| w \, \chi_{D_\theta}\|_{\hilb,\infty}$ in terms of $\|w_\cyl \|_{\hv,\infty}$.
This will be done by considering the tip region and establishing an appropriate \emph{a priori} bound for the difference of our two solutions there.

\subsubsection{The tip region}\label{subsec-tip}
The tip region is defined by
$
\tip_{\theta} = \{(u, \tau):\,\, u_1 \le 2\theta, \, \tau \leq \tau_0\}.
$ 
Since equation \eqref{eq-u} becomes singular at the tips $\sigma_{\pm}(\tau)$, in the tip region  we  introduce the change of variables $Y:=u_\sigma^2$
and we view $Y$ as a function of  $(u,\tau)$ (see in \cite{ADS3}, Section 2 for details).   A simple calculation shows that  $Y(u,\tau)$ satisfies the equation 
\begin{equation}\label{eqn-Y}
  Y_\tau + \frac u2 \, Y_u =  Y Y_{uu} - \frac 12\, (Y_u)^2 + (1 - Y)\, \frac{Y_u}{u}  + 2  (1-Y) \, \frac{Y}{u^2}.
\end{equation}

Under this change of variables, our solutions $u_1(\sigma,\tau) $ and $u_2^{\beta \gamma}(\sigma,\tau)$ become $Y_1(u,\tau)$ and $Y_2^{\beta \gamma}(u,\tau)$.
In this region we consider a {\em cut-off function $\varphi_T(u)$} with the following properties:
\begin{equation}\label{eqn-cutofftip}
  (i) \,\, \supp \varphi_T \Subset \tip_{\theta} \qquad (ii) \,\, 0 \leq \varphi_T \leq 1 \qquad (iii) \,\, \varphi_T \equiv 1, \,\, \mbox{on} \,\, \tip_{\theta/2}.
\end{equation}

Let $\Psi_1(u,\tau) = \sqrt{Y_1(u,\tau)}$ and $\Psi_2^{\beta\gamma}(u,\tau) = \sqrt{Y_2^{\beta\gamma}(u,\tau)}$.
We will see  in Section \ref{subsec-tip-energy} that   the difference $W(u,\tau):= \Psi_1(u,\tau) - \Psi_2^{\beta \gamma}(u,\tau)$ satisfies
equation \eqref{eqn-W}. 
%
Our next goal is to define an appropriate weighted norm in the tip region $\tip_\theta$, by defining the weight $\mu(u,\tau)$.
To this end we need to further distinguish between two regions in $\tip_{\theta}$: for $L >0$ sufficiently large to be determined in Section \ref{sec-tip}, we define the {\it collar} region to be the set
\be\label{eqn-collar-region}
\collar_{\theta,L} := \Bigl\{y \mid \frac{L}{\sqrt{|\tau|}} \le u_1 (\sigma,\tau) \le 2\theta\Bigr\}
\ee
and the {\it soliton} region to be the set
\be\label{eqn-soliton-region}
\mathcal{S}_L := \Bigl\{y \mid 0 \le u_1(\sigma,\tau) \le \frac{L}{\sqrt{|\tau|}}\Bigr\}.
\ee
This is necessary,  as the behavior of our solutions changes from being cylindrical in the collar region to resembling  the Bryant soliton in the soliton region. 

\sk
In the soliton region we further rescale our solutions by setting for each solutions 
$$Z(\rho,\tau) = Y(u,\tau), \qquad \rho:= u\, \sqrt{|\tau|}.$$
A direct calculation using \eqref{eqn-Y} shows that each rescaled solution $Z$ satisfies the equation
\be\label{eqn-Z}
\frac{1}{|\tau|}\, \Big(Z_{\tau} - \frac{1}{2|\tau|} \rho Z_{\rho} + \frac{\rho Z_{\rho}}{2}\Big) = Z Z_{\rho\rho} - \frac12 Z_{\rho}^2 + (1 -Z) \frac{Z_{\rho}}{\rho} + \frac{2(1-Z)Z}{\rho^2}.
\ee
The collar region can be viewed as the transition region between the cylindrical and soliton regions.  
Furthermore, the result in Theorem \ref{thm-asym} implies that each solution $Z(\rho,\tau) $ converges smoothly on compact subsets 
of $\rho \leq L$ to the translating Bowl soliton $\hZ_0(\rho)$ with maximal curvature one  (see in  \cite{Bryant}    for more details about the Bryant soliton). It was shown in \cite{Bryant} that $\hZ_0(\rho)$ satisfies the following asymptotics 
\begin{equation}\label{eqn-Zasym}
  \hZ_0 (\rho)= \begin{cases} 1 - \rho^2/6 + O(\rho^4), &\qquad \mbox{as} \,\,\,\,\, \rho\to 0 \\
    \rho^{-2 } + O(\rho^{-4}), &\qquad \mbox{as} \,\,\,\, \rho\to\infty.
  \end{cases}
\end{equation}
\smallskip

Let us next define our {\em weight}  $\mu(u,\tau)$ in the {\em tip region} as a function of $(u,\tau)$.  Let  $\zeta (u)$ be a  nonnegative smooth decreasing function defined on $u \in (0, \infty)$ such that 
$$\zeta(u)=1, \quad \mbox{for} \,\, u \geq \theta/2 \qquad \mbox{and}  \quad \zeta(u)=0, \quad \mbox{for}\,\, u \leq \theta/4.$$
Such a function can be chosen to satisfy the derivative estimate $|\zeta'(u)| \leq 5 \theta^{-1}$. 

For our given solution $u(\sigma,\tau)$ which after the coordinate change gives rise to $\sigma(u,\tau)$
and $Y(\sigma,\tau) := u_{\sigma}^2(u,\tau)$ (recall that we have dropped the index and denote $\sigma, Y_1$ 
by  $\sigma, Y$ respectively) we define our {\em weight}   $\mu(u,\tau)$ in the {\em tip region}  to be
\be\label{eqn-weight1}
 \mu(u,\tau) = -   \frac{\sigma^2(\theta,\tau)}4 + \int_\theta^u  \mu_u(u',\tau) \, du'
\ee
where 
\be\label{eqn-weight2}
\mu_u := \zeta (u) \, \big ( - \frac{\sigma^2(u,\tau)}{4} \big ) _u + (1-\zeta(u)) \, \frac{1-Y(u,\tau)}{ u \, Y(u,\tau)}.
\ee

Note that since $\zeta \equiv 1$ for $u \geq \theta/2$, we have $\mu(u,\tau) =  - \frac{\sigma^2(u,\tau)}{4} $
in this region, hence it coincides with our weight in the cylindrical region. This is important as our norms in
the intersection of the cylindrical and tip regions need to coincide. 

\smallskip 
Now that we have defined the weight $\mu(u,\tau)$, let us define the {\em norm} in the {\em tip region}. 
For a function For a function $W:[0,2\theta]\times(-\infty, \tau_0] \to \R$  and any $\tau\leq \tau_0$, we define the weighted $L^2$ norm with respect to the weight $\Psi^{-2} \, e^{\mu(\cdot,\tau)} \, du$ by
\be\label{eqn-normt00}
\|W(\cdot, \tau)\|^2 := \int_0^{2\theta}  W^2(u, \tau) \, \Psi^{-2}\, e^{\mu(u,\tau)}\, du, \qquad \tau \leq \tau_0
\ee
where  $\Psi := \Psi_1 := \sqrt{Y_1}$ denotes one of our  solutions in the tip coordinates. 
Furthermore,  we define the tip norm to be
\be\label{eqn-normt0}
 \|W\|_{2,\infty}(\tau) = \sup_{\tau' \le \tau}|\tau'|^{-1/4}\, \Bigl(\int_{\tau'-1}^{\tau'} \|W(\cdot, s)\|^2\, ds\Bigr)^{1/2} 
\ee
 for any $\tau\leq \tau_0$.
We include the weight in time $|\tau|^{-1/4}$ to make the norms equivalent in the transition region, between the cylindrical and the tip regions, as will become apparent in Corollary \ref{cor-equiv-norm}.
We will also abbreviate
\be\label{eqn-normt}
\|W\|_{2,\infty}:= \|W\|_{2,\infty}(\tau_0).
\ee

\smallskip

For a cutoff function $\varphi_T$ as in \eqref{eqn-cutofftip}, we set $W_T(u,\tau):= W(u,\tau) \, \varphi_T$,  where $W:=\Psi_1 - \Psi_2^{\beta\gamma}$. 
In   Section \ref{sec-tip} we will show  the following crucial estimate which roughly states that the norm of 
the difference $W$ of our two solutions in the tip region can be estimated by the norm of $W$ in the region $\theta \leq u \leq 2\theta$ which 
is included in our cylindrical region $u \geq \theta/4$. More precisely, we will show that there exists a small $\theta >0$ and $\tau_0 \ll -1$ depending on $\theta$ for which 
\begin{equation}
  \label{eqn-tip2}
  \| W_T \|_{2,\infty} \leq \frac{C}{\sqrt{|\tau_0|}} \, \| W \, \chi_{[\theta,2\theta]} \|_{2,\infty}
\end{equation}
for $\tau \leq \tau_0$. Here   $\chi_{[\theta,2\theta]}$ is the characteristic function of the interval $[\theta,2\theta]$. We will next outline how this estimate is combined with that in the cylindrical region  to close the argument and  conclude uniqueness.

\subsection{The conclusion}
\label{subsec-conclusion}
The statement of Theorem \ref{thm-main} is equivalent to showing there exist parameters  $\beta$ and $\gamma$ so that $u_1(\sigma,\tau) = u_2^{\beta \gamma}(\sigma,\tau)$, where $u_2^{\beta \gamma}(\sigma,\tau)$ is defined by \eqref{eq-ualphabeta} and both functions, $u_1(\sigma,\tau)$ and $u_2^{\beta \gamma}(\sigma,\tau)$, satisfy equation \eqref{eq-u}.
We set $w:= u_1 - u_2^{\beta \gamma}$, and $W:= \Psi_1 - \Psi_2^{\beta\gamma}$, where $(\beta,\gamma)$ is an admissible pair  of parameters with respect to $\tau_0$, such that \eqref{eqn-projections} holds for a $\tau_0 \ll -1$.

\sk

Fix $\theta >0$ and small such that \eqref{eqn-tip2} holds for $\tau_0 \ll -1$ depending on $\theta$. For that fixed $\theta$ and any $\epsilon >0$
we also have that  \eqref{eqn-cylindrical1} holds, for $\tau_0 \ll -1$ depending on $\theta, \epsilon$. 
Combining  \eqref{eqn-cylindrical1} and \eqref{eqn-tip2} with the estimates in Corollary \ref{cor-equiv-norm} 
which compare our norms in the intersection of cylindrical and tip  regions, we finally show that  $\| \pr_0 w_\cyl \|_{\hv,\infty}$ dominates over the norms of $\pr_- w_{\cyl}$ and $\pr_+ w_{\cyl}$ (this happens by applying \eqref{eqn-cylindrical1} for  $\epsilon$ sufficiently small depending 
only on  $\theta$). 
We will use this fact in Section \ref{sec-conclusion} to conclude that $w(\sigma,\tau):= w^{\beta \gamma}(\sigma,\tau) \equiv 0$ for our choice of parameters $\beta$ and $\gamma$.
To do so we will look at the projection $a(\tau):= \pr_0 w_\cyl$ and define its norm
\[
\|a\|_{\hilb,\infty}(\tau)
= \sup_{\sigma \le \tau} \Bigl(\int_{\sigma - 1}^{\sigma} \|a(s)\|^2\, ds\Bigr)^{\frac12},
\qquad
\tau \leq \tau_0
\]
with $\|a\|_{\hilb,\infty}:=\|a\|_{\hilb,\infty}(\tau_0)$.

By projecting equation \eqref{eqn-w100} onto the zero eigenspace spanned by $\psi_2$ and estimating error terms by $\| a\|_{\hilb,\infty}$ itself, we will conclude in Section \ref{sec-conclusion} that $a(\tau)$ satisfies a certain differential inequality which combined with our assumption that $a(\tau_0)=0$ (that follows from the choice of parameters $\beta$, $\gamma$  so that \eqref{eqn-projections} hold) will yield that $a(\tau)=0$ for all $\tau \leq \tau_0$.
On the other hand, since $\| a\|_{\hilb,\infty} $ dominates the $\| w_\cyl \|_{\hilb,\infty}$, this will imply that $w_\cyl \equiv 0$, thus yielding $w \equiv 0$, as stated in Theorem \ref{thm-main}.

\section{Cylindrical region}
\label{sec-cylindrical}

Let $u_1(\sigma,\tau)$ and $u_2(\sigma,\tau)$ be two solutions to equation
\eqref{eq-u} as in the statement of Theorem \ref{thm-main} and let
$u_2^{\beta \gamma}$ be defined by \eqref{eq-ualphabeta}.  In this section
we will estimate the difference $w:= u_1-u_2^{\beta \gamma}$ in the
cylindrical region
$\cyl_{\theta} = \{\sigma\,\,\, |\,\,\, u_1(\sigma,\tau) \ge {\theta}/{2}\, \}$, for a
given number $\theta > 0$ small and any $\tau \leq \tau_0 \ll -1$.  Recall that all
the definitions and notations have been introduced in Section \ref{subsec-cylindrical}.
Before we state and prove the main estimate in the cylindrical region we give a
remark that the reader should be aware of throughout the whole section.

\begin{remark}
  \label{rem-cylindrical}
  Recall that we write simply $u_2(\sigma,\tau)$ for
  $u_2^{\beta \gamma}(\sigma,\tau)$, where
  \[
    u_2^{\beta \gamma}(\sigma,\tau) = \sqrt{1+\beta e^{\tau}}\,
    u_2\Bigl(\frac{\sigma}{\sqrt{1+\beta e^{\tau}}}, \tau + \gamma - \log(1 + \beta
    e^{\tau})\Bigr),
  \]
  is still a solution to \eqref{eq-u} and simply write $w(\sigma,\tau)$ for
  $w^{\beta \gamma}(\sigma,\tau) := u_1(\sigma,\tau) -
  u_2^{\beta \gamma}(\sigma,\tau)$.  As it has been already indicated in Section
  \ref{subsec-conclusion}, we will choose 
  $\beta = \beta(\tau_0)$ and $\gamma = \gamma(\tau_0)$ (as it will be explained
  in Section \ref{sec-conclusion}) so that the projections
  $\pr_+ w_\cyl(\tau_0) = \pr_0 w_\cyl(\tau_0) = 0$, at a suitably chosen
  $\tau_0 \ll -1$.  In Section \ref{sec-conclusion} we show the pair
  $(\beta,\gamma)$ is admissible with respect to $\tau_0$, in the sense
  of Definition \ref{def-admissible}, if $\tau_0$ is sufficiently small.  That
  will imply all our estimates that follow are independent of the parameters
  $ \beta, \gamma$, as long as they are admissible with respect to
  $\tau_0$, and will hold for $u_1(\sigma,\tau) - u_2^{\beta \gamma}(\sigma,\tau)$,
  for $\tau\le \tau_0$ (as explained in section \ref{sec-outline}).
\end{remark}

Our goal in this section is to prove that the bound~\eqref{eqn-cylindrical1}
holds as stated next. Recall the notation   $w_C := w \, \varphi_C$, where $\varphi_C$ is the cut off supported 
in the cylindrical region.  

\begin{proposition}\label{prop-cylindrical}
  For every $\epsilon > 0$ and $\theta > 0$ small, there exists a $\tau_0 \ll -1$
  so that if $w:= u_1 - u_2^{\beta\gamma}$ satisfies 
  $\pr_+w_\cyl(\tau_0) = 0$, then we have
  \[
    \| \hat w_\cyl \|_{\hv,\infty} \leq \epsilon\, \big(\| w_\cyl
    \|_{\hv,\infty} + \|w\, \chi_{D_{\theta}}\|_{\cH,\infty}\big),
  \]
  where $D_{\theta} := \{\sigma\,\,\,|\,\,\, \theta/2 \le u_1(\sigma,\theta) \le \theta\}$
  and $\hat{w}_C = \pr_- w_\cyl + \pr_+ w_\cyl$.
\end{proposition}

\smallskip

Our linear operator $\cL(f) = f_{\sigma\sigma} - \frac{\sigma}{2} + f$ is the same is in \cite{ADS2}, and hence, the linear theory we derived in \cite{ADS2} carries over to the Ricci flow case as well. In order for this article to be self-contained, we will state the results from \cite{ADS2} that we will use later, but for the proofs of the same we refer reader to look at \cite{ADS2}. More precisely, in \cite{ADS2} we obtained energy type estimate for ancient solutions
$ f:(-\infty, \tau_0] \to \hv $ of the linear cylindrical equation
\begin{equation}
  \label{eqn-linear1}
  \frac{\partial f}{\partial \tau} - \cL f(\tau) = g(\tau).
\end{equation}

\begin{lemma}[Lemma 5.8 in \cite{ADS2}]
  \label{lem-linear-cylindrical-estimates-sup-L2-version}
  Let $f:(-\infty, \tau_0] \to \hv$ be a bounded solution of equation
  \eqref{eqn-linear1}.  If $T>0$ is sufficiently large, then there is a constant
  $C_{\star}$ such that
  \begin{equation}
    \begin{split}
      \sup_{\tau\leq\tau_0}\| \hat f(\tau)\|_\cH^2
      &+ C_{\star}^{-1} \,  \sup_{n\geq 0} \int_{I_n} \|\hat f(\tau)\|_\hv^2\, d\tau \\
      &\leq \|f_+(\tau_0)\|_\cH^2 + C_\star \sup _{n\geq 0}\int_{I_n} \|\hat
      g(\tau)\|_{\hv^*}^2\, d\tau,
      \label{eq-basic-cylindrical-estimate}
    \end{split}
  \end{equation}
  where $I_n$ is the interval $ I_n = [\tau_0-(n+1)T, \tau_0-nT] $ and where
  $ f_+ = \pr_+ f $ and $ \hat f = \pr_+f + \pr_- f$.
\end{lemma}

We also summarize in the next lemma the following bounds on various operators between the Hilbert spaces $\cH$ and $\hv$ with norms \eqref{eqn-normp0} and \eqref{eqn-normhv} respectively (see in  Section  \ref{subsec-cylindrical} for their definitions). 
 All these bounds were shown  in 
\cite{ADS2}. 

\begin{lemma}\label{lemma-esti10}  The following hold:
\begin{enumerate}
\item [i.] $f \to \sigma f$ is bounded from $\cD$ to $\cH$.
\item[ii.] $f \to \sigma f$, $f \to \partial_{\sigma} f$, $f \to \partial^*_{\sigma} f = (-\partial_{\sigma}f + \frac{\sigma}{2} f)$ are
 bounded   from $\cH$ to $\cD^*$.
\item[iii.]
$f \to \sigma^2 f$, $f \to \sigma \partial_{\sigma} f$, $f \to \partial_{\sigma}^2 f$ are bounded from $\cD$ to $\cD^*$.
\item[iv.]
$f \to f$ is bounded from $\cD$ to $\cH$ and hence from $\cH$ to $\cD^*$.
\end{enumerate}
\end{lemma}

The rest of this section will be devoted to the proof of Proposition
\ref{prop-cylindrical}.  To simplify the notation we
will simply denote $u_2^{\beta \gamma}$ by $u_2$ and set $w:=u_1-u_2$.  The difference $w$ satisfies the equation 
\begin{equation} 
\begin{split}
  \frac{\partial}{\partial\tau} w = \cL[w] 
  & +  \frac{w_{\sigma}^2}{u_1} - \frac{w^2}{2u_1} + \frac{2 w_{\sigma} u_{2\sigma}}{u_1} - \frac{u_{2\sigma}^2 w}{u_1 u_2} \\& +  \frac{w}{2u_1u_2} (2-u_2^2 ) - J_1 w_{\sigma} +  u_{2\sigma}\, (J_2 - J_1) 
\end{split}
\end{equation}
Note that
\be\label{eqn-J12}
J_1 - J_2 = - 2 \int_0^\sigma \frac{u_{1\sigma\sigma}}{u_1} d\sigma'  +  2 \int_0^\sigma \frac{u_{2\sigma\sigma}}{u_2}  d\sigma' =
- 2\int_0^{\sigma} \Big(\frac{w_{\sigma\sigma}}{u_1} - \frac{u_{2\sigma\sigma} w}{u_1 u_2}\Big) d\sigma'.\ee
\sk
\sk 
Let $\varphi_C$ be a cut off function as in section \ref{subsec-cylindrical} and let $w_C = w \varphi_C$. An easy computation shows that
\begin{equation}
\label{eq-WC}
\frac{\partial}{\partial \tau} w_C = \cL[w_C] + \cE(w_C) + \bar{\cE}[w,\varphi_C] + \cE_{nl}
\end{equation}
where 
\begin{equation}
\label{eq-E10}
\cE(w_C) := \Big(\frac{w_{\sigma}}{u_1} +  \frac{2 u_{2\sigma}}{u_1}  - J_1\Big)\, (w_C)_{\sigma} - \Big(\frac{w}{2u_1} +\frac{u_{2\sigma}^2}{u_1 u_2} + \frac{u_2^2- 2}{2u_1 u_2}\Big)\, w_C
\end{equation}
and
\begin{equation}
\label{eq-bar-E}
\begin{split}
\bar{\cE}[w,\varphi_C] &:= \varphi_{C,\tau} w - (\varphi_C)_{\sigma\sigma} w - 2(\varphi_C)_{\sigma} w_{\sigma} + \frac{\sigma}{2} (\varphi_C)_{\sigma} w \\
&\quad \, - \frac{ (\varphi_C)_{\sigma}w w_{\sigma}}{u_1} -\frac{2 (\varphi_C)_{\sigma} u_{2\sigma} w}{u_1} + J_1 (\varphi_C)_{\sigma} w
\end{split}
\end{equation}
and
\begin{equation}
\label{eq-nonlocal-error}
\cE_{nl} := u_{2\sigma} \varphi_C \, (J_2 - J_1).
\end{equation}

\sk

The proof of Proposition \ref{prop-cylindrical}  will follow easily by combining  Lemma \ref{lem-linear-cylindrical-estimates-sup-L2-version}
with the next estimate on the error terms in equation \eqref{eq-WC}. 

\begin{proposition}
\label{prop-error}
For every $\epsilon > 0$ and $\theta > 0$ small, there exists a $\tau_0 \ll -1$, so that the error term $\cE := \cE(w_C) + \bar{\cE}[w,\varphi_C] + \cE_{nl}$
satisfies the estimate 
\be\label{eqn-error12}
\|\cE\|_{\cD^*,\infty} \leq  \epsilon\, (\|w_C\|_{\cD,\infty} + \|w\,\chi_{D_{\theta}}\|_{\cH,\infty})
\ee
where $D_{\theta} := \{\sigma\,\,\,|\,\,\, \theta/2 \le u_1(\sigma,\tau) \le \theta\}$.
\end{proposition}

\begin{proof}
Note that our unique asymptotics result in \cite{ADS3}, together with more refined asymptotics in the collar region (see section \ref{sec-appendix} for precise statements and their proofs) we have that for every $\epsilon > 0$ there exists a $\tau_0 \ll -1$ so that for $\tau \le \tau_0$ we have $|w(\sigma,\tau)| \leq  \epsilon$ in $\cC_{\theta}$. Combining this, estimate \eqref{eq-E10} and \eqref{eq-der-inter}, since $u \ge \theta/2$ in $\cC_{\theta}$, yield
\begin{equation}
\label{eq-error-120}
\cE(w_C)| \le \frac{\epsilon}{6}\, \Big(|w_C| + |(w_C)_{\sigma}|\Big) +  \frac{|u_2^2 - 2|}{2u_1 u_2} |w_C|
\end{equation}
and 
\begin{equation}
\label{eq-error-121}
|\bar{\cE}[w,\varphi_C]| \le \frac{C(\theta)}{\sqrt{|\tau_0|}}\,  |w\, \chi_{D_{\theta}}|  + C(\theta)|(\varphi_C)_{\sigma}\,w_{\sigma}|
\end{equation}
in $\cC_{\theta}$, for $\tau \le \tau_0$. Then, using the estimates in Lemma \ref{lemma-esti10}  and \eqref{eq-error-120} we have
\begin{equation*}
\begin{split}
\|\cE(w_C)\|_{\cD^*} &\le \frac{\epsilon}{6}\, \|w_C\|_{\cD} + \Big\|\frac{(u_2^2 - 2) w_C}{2u_1 u_2}\Big\|_{\cD^*} \\
&\le \frac{\epsilon}{6}\, \|w_C\|_{\cD} + C_0\, \Big\|\frac{1}{\sigma+1} \frac{(u_2^2 - 2)}{2u_1 u_2}\, w_C\Big\|_{\cH} \\
&\le \frac{\epsilon}{6}\, \|w_C\|_{\cD} + C(\theta) \Big(\int_0^K (u_2 - \sqrt{2})^2 w_C^2\, d\mu\Big)^{\frac12} + \frac{C(\theta)}{K}\Big(\int_K^{\infty} w_C^2\, d\mu\Big)^{\frac12},
\end{split}
\end{equation*} 
where $C_0$ is a uniform constant, that is, an upper bound of the norm of the operator $f \to \sigma f$   from $\cH$ to $\cD^*$. 
For a given $\epsilon > 0$, choose $K$ large so that ${C(\theta)}/{K} \leq  {\epsilon}/{12}$ and then for this $K$ choose 
$\tau_0 \ll -1$ so that for $\tau\le \tau_0$ and for $\sigma\in [0,K]$ we have $|u_2 - \sqrt{2}| \leq {\epsilon}/{12}$. The latter follows from the fact that the $\lim_{\tau\to -\infty} u_2(\sigma,\tau) = \sqrt{2}$, uniformly on compact sets. This finally yields
\begin{equation}
\label{eq-E-wC}
\|\cE(w_C)\|_{\cD^*,\infty} \leq  \frac{\epsilon}{3} \|w_C\|_{\cD,\infty}.
\end{equation}

Furthermore, using \eqref{eq-error-121},  Lemma \ref{lemma-esti10} and the definition of the cut off function $\varphi_C$, for all $\tau \le \tau_0$ we have
\begin{equation*}
\begin{split}
\|\bar{\cE}[w,\varphi_C]\|_{\cD^*} &\le \frac{C(\theta)}{\sqrt{|\tau_0|}} \, \| w \chi_{D_{\theta}} \|_{\cD^*} + C(\theta) \, \| (\varphi_C)_{\sigma} w_{\sigma}\|_{\cD^*} \\
&\le  \frac{C(\theta)}{\sqrt{|\tau_0|}}\,  \| w \chi_{D_{\theta}} \|_{\cH} + C(\theta) \Big ( \|\big((\varphi_C)_{\sigma} w\big)_{\sigma}\|_{\cD^*} + \|(\varphi_C)_{\sigma\sigma} w\|_{\cH} \Big ) \\
&\le \frac{C(\theta)}{\sqrt{|\tau_0|}} \, \|w \chi_{D_{\theta}} \|_{\cH},
\end{split}
\end{equation*}
where the constant $C(\theta)$ may vary from line to line, but is uniform in time. This leads to
\begin{equation}
\label{eq-E-phiC}
\|\bar{\cE}[w,\varphi_C]\|_{\cD^*,\infty} \le \frac{C(\theta)}{\sqrt{|\tau_0|}}\, \|\chi_{D_{\theta}} w\|_{\cH,\infty}.
\end{equation}

It remains to deal with the more delicate bound of the  non-local term $\cE_{nl}$. Using \eqref{eqn-J12} we have 
\[\cE_{nl} =  -2\Big(u_{2\sigma}\varphi_C \int_0^{\sigma} \frac{w_{\sigma\sigma}}{u_1}\, d\sigma' + u_{2\sigma}\varphi_C \int \frac{u_{2\sigma\sigma}}{u_1 u_2}\, w\, d\sigma'\Big) =: -2 (I_1 + I_2)\]
and hence,
\[\|\cE_{nl}\|_{\cD^*} \le 2 \, ( \|I_1\|_{\cD^*} + \|I_2\|_{\cD^*}).\]
Using \eqref{eq-der-inter} we have
\be
\begin{split}
\|I_2\|^2_{\cD^*} \le C_0\, \|I_2\|^2_{\cH} &\le \frac{C(\theta)}{|\tau|^2}\,  \int_{\sigma \ge 0} \varphi_C^2\, \Big(\int_0^{\sigma}|w|\, d\sigma'\Big)^2 e^{-\frac{\sigma^2}{4}}\, d\sigma\\ 
&\qquad + \frac{C(\theta)}{|\tau|^2}\, \int_{\sigma \le 0} \varphi_C^2\, \Big(\int_0^{\sigma}|w|\, d\sigma'\Big)^2 e^{-\frac{\sigma^2}{4}}\, d\sigma.
\end{split}
\ee
It is enough to show how we deal with one of the two terms since the other one is handled similarly. Note that by definition  $\varphi_C$ is nonincreasing for $\sigma \ge 0$ and is nondecreasing for $\sigma \le 0$. Using this and the Fubini theorem we get
\begin{equation*}
\begin{split}
\int_{\sigma \ge 0} \varphi_C^2\, \Big(\int_0^{\sigma}|w|\, d\sigma'\Big)^2 & e^{-\frac{\sigma^2}{4}}\, d\sigma \le \int_{\sigma \ge 0} \varphi_C^2 \Big(\int_0^{\sigma} w^2\, d\sigma'\Big) \sigma e^{-\sigma^2/4}\, d\sigma \\
&\le \int_{\sigma \ge 0} \Big(\int_0^{\sigma} \varphi_C^2(\sigma') w^2(\sigma',\tau)\, d\sigma'\Big) \sigma e^{-\sigma^2/4}\, d\sigma \\
&= \int_0^{u_1^{-1}(\theta,\tau)} w_C^2\, \Big(\int_{\sigma'}^{u_1^{-1}(\theta,\tau)} \sigma e^{-\sigma^2/4}\, d\sigma\Big)\, d\sigma' \\
&\le C_0 \int_{\sigma \ge 0} w_C^2\, d\mu \le C_0 \, \|w_C\|_{\cH}^2.
\end{split}
\end{equation*}
This yields the bound 
\begin{equation}
\label{eq-I2-est}
\|I_2\|_{\cD^*,\infty} \le \frac{C(\theta)}{|\tau_0|} \|w_C\|_{\cH,\infty}.
\end{equation}
We deal with the term $I_1$, similarly as with term $I_2$ above.  Using  \eqref{eq-der-inter}, integration by parts and Fubini's theorem, gives that for all $\tau \leq \tau_0$ we have 
\begin{equation}
\label{eq-I1-est}
\begin{split}
\|I_1\|_{\cD^*}^2 \le C_0 \|I_1\|_{\cH}^2 \le &\frac{C(\theta)}{|\tau|}\Big(\Big\|\varphi_C\int_0^{\sigma} \frac{w_{\sigma} u_{1\sigma}}{u_1^2}\, d\sigma'\Big\|_{\cH}^2 +  \Big\|\varphi_C \, \frac{w_{\sigma}}{u_1}(\sigma,\tau)\Big)\Big\|_{\cH}^2\Big)\\
&\le \frac{C(\theta)}{|\tau|}\,\big  (\|\varphi_C w_{\sigma}\|_{\cH}^2  \big ) 
\le \frac{C(\theta)}{|\tau_0|}\, \big (\|w_C\|_{\cH}^2 + \|w\chi_{D_{\theta}}\|_{\cH}^2  \big). 
\end{split}
\end{equation}
It follows that 
\begin{equation}
\label{eq-I1-est-final}
\|I_1\|_{\cD^*,\infty}^2 \le \frac{C(\theta)}{|\tau_0|}\, (\|w_C\|^2_{\cD,\infty} + \|w\chi_{D_{\theta}}\|^2_{\cH,\infty}).
\end{equation}
Finally, \eqref{eq-E-wC}, \eqref{eq-E-phiC}, \eqref{eq-I2-est} and \eqref{eq-I1-est-final} imply that for every $\epsilon > 0$, there exists a $\tau_0 \ll -1$ so that \eqref{eqn-error12} holds, hence 
concluding the proof of Proposition.
\end{proof}

We can finally finish the proof of  Proposition \ref{prop-cylindrical}.

\begin{proof}[Proof of Proposition \ref{prop-cylindrical}]
Apply Lemma \ref{lem-linear-cylindrical-estimates-sup-L2-version} to $w_C$ solving \eqref{eq-WC}, to conclude that there exist $\tau_0 \ll -1$ and constant $C_0 > 0$, so that if the parameters $( \beta, \gamma)$ are chosen to ensure that $\mathcal{P}_+ w_C(\tau_0) = 0$, then $\hat{w}_C := \mathcal{P}_+ w_C + \mathcal{P}_-w_C$ satisfies the estimate
\[\|\hat{w}_C\|_{\cD,\infty} \le C_0 \, \|\cE\|_{\cD^*,\infty}, \qquad \mbox{for all} \,\, \tau \leq \tau_0\]
where $\cE := \cE(w_C) + \bar{\cE}[w,\varphi_C] + \cE_{nl}$. Combining this together with Proposition \ref{prop-error} concludes the proof of Proposition \ref{prop-cylindrical}.
\end{proof}

\section{The tip region}
\label{sec-tip}

Let $u_1(\sigma,\tau)$ and $u_2(\sigma,\tau)$ be the two solutions to equation
\eqref{eq-u} as in the statement of Theorem \ref{thm-main} and let
$u_2^{\beta \gamma}$ be defined by \eqref{eq-ualphabeta}.  We
will now estimate the difference of these solutions in the {\em tip region} 
$$\tip_{\theta} = \{(\sigma,\tau)\,\,\,|\,\,\, u_1(\sigma,\tau)  \le 2\theta\}$$
for
$\theta >0$ sufficiently small, and $\tau \le \tau_0 \ll -1$, where
$\tau_0$ is to be chosen later. Recall from Section \ref{subsec-tip} that the tip region is further divided into the {\em collar region}  $\collar_{L,\theta}$ defined by \eqref{eqn-collar-region} and
the {\em  soliton region}  $S_L$ defined by \eqref{eqn-soliton-region}. 

\sk
As we have seen in Section \ref{subsec-tip}, in the tip region we exchange the variables $\sigma$ and $u$
and consider $Y(u,\tau):=u_\sigma^2(\sigma,\tau)$ as a function of $u$. This means that for our given solutions   $u_1(\sigma,\tau), u_2^{\beta \gamma}(\sigma,\tau)$ of \eqref{eqn-u},  
we consider  
$$Y_1(u,\tau):=u_{1\sigma}^2(\sigma,\tau), \quad  \mbox{where} \,\, \sigma=\sigma_1(u,\tau) \iff u=u_1(\sigma,\tau)$$
and similarly
$$ Y_2^{\beta \gamma}(u,\tau) := (u^{\beta \gamma}_{2\sigma})^2(\sigma,\tau),  \quad \mbox{where} \,\, \sigma=\sigma_2^{\beta \gamma}(u,\tau) \iff u=u_2^{\beta \gamma}(\sigma,\tau).$$
Note that  by the definition of $ u_2^{\beta \gamma}(\sigma,\tau)$  (see \eqref{eq-ualphabeta}) we have that 
\be\label{eqn-Yabc}
  Y_2^{\beta \gamma}(u,\tau)
  = Y_2\Big(\frac{u}{\sqrt{1+\beta e^{\tau}}}, \tau + \gamma - \log(1+ \beta e^{\tau})\Big).\ee
 \sk

We have seen in Section \ref{subsec-tip} that in the soliton region we   need   further rescale  of our solutions, namely define 
\be\label{eqn-defnZ12}
Z_1(\rho,\tau) := Y_1(u,\tau)   \quad   \mbox{and} \quad Z_2(\rho,\tau) := Y_2(u,\tau), \qquad \rho  := \sqrt{|\tau|}\, u.
\ee
Both rescaled solutions satisfy equation \eqref{eqn-Z}. Also, using  \eqref{eqn-Yabc} we see that 
\be\label{eqn-Zabc}
Z_2^{\beta \gamma}(\rho,\tau) = Z_2\Big(\frac{\rho\, \sqrt{|\tau + \gamma - \log(1 + \beta e^{\tau})|}}{\sqrt{|\tau|}\, \sqrt{1+ \beta e^{\tau}}}, \tau + \gamma - \log(1 + \beta e^{\tau})\Big).
\ee

\sk 

The following simple consequence of Theorem \ref{thm-asym} will be used in the sequel. 

\begin{proposition} \label{prop-convergence-tip}
  If $(\beta, \gamma)$ are $\tau_0$ admissible in the sense of 
  Definition~\ref{def-admissible}, then
  \[
Z_1(\rho,\tau)   \to Z_0(\rho)  \quad \mbox{and} \quad  Z^{\beta \gamma}_2(\rho, \tau) \to Z_0(\rho) \qquad
  \text{as }\tau\to -\infty,
  \]
uniformly on compact sets  and smoothly, where  $Z_0(\rho)$ is  the unique rotationally symmetric Bryant soliton  with maximal scalar curvature equal to one.  
\end{proposition}

\begin{proof}

Lets first show  that each of  the rescaled solutions  $Z_i(\rho,\tau)$ according to \eqref{eqn-defnZ12} converges, as $\tau\to -\infty$, uniformly smoothly to the unique rotationally symmetric Bryant soliton $Z_0(\rho)$ whose maximum curvature is equal to one. Let's  drop the subscript $i$ from the solutions setting for simplicity, that is set   $Z:=Z_i$. Denote by $(S^3,g(\cdot,t))$ the unrescaled solution of \eqref{eq-rf}. 

\sk

By Theorem \ref{thm-asym} we know  that the maximal scalar curvature $k(t)$ of the solution $(S^3, g(\cdot,t))$ satisfies  $ k(t) =  \frac{\log(-t)}{(-t)} \big ( 1 + o(1) \big )$, as $t \to -\infty$. Moreover, according to  Theorem \ref{thm-asym}, the 
rescaled solution  $(S^3,\tilde{g}(\cdot,\tau))$, defined by $$\tilde{g}(\cdot,\tau) = k(t) \, g(\cdot,t), \qquad \tau = -\log(-t),$$ 
whose maximal scalar curvature is equal to one, converges, as $\tau\to -\infty$,   to the unique rotationally symmetric Bryant soliton whose maximal scalar curvature is equal to one. 
Since, 
\[\tilde{g}(\cdot,\tau) = |\tau|\, \Big(\frac{du^2}{Y(u,\tau)} + u^2 g_{S^2}\Big) = \frac{d\rho^2}{Z(\rho,\tau)} + \rho^2 \, g_{S^2},\]
we conclude from the above discussion that  $Z(\rho,\tau)$ converges, as $\tau \to -\infty$, in $C^\infty_{\rm loc}$ to $Z_0(\rho)$, where 
${\ds \frac{d\rho^2}{Z_0(\rho)} + \rho^2\, g_{S^2}}$ is the  Bryant soliton metric whose maximal scalar curvature is one.  

This in particular shows that $Z_1(\rho, \tau) \to Z_0(\rho)$ and  $Z_2^{\beta\gamma}(\rho, \tau) \to Z_0(\rho)$  in $C^\infty_{\rm loc}$.
\end{proof}

%
%

For a  cutoff function $\varphi_T(u)$ supported in the tip region (see  \eqref{eqn-cutofftip} for its definitions), we set 
\be\label{eqn-W-15}
W_T := \varphi_T\, W  \qquad \mbox{where} \qquad   W:=\Psi_1-\Psi_2^{\beta \gamma},
\ee
and $\Psi_1 := \sqrt{Y_1}$ and $\Psi^{\beta\gamma} := \sqrt{Y_2^{\beta\gamma}}$.
The reason for passing from $Y$ to $\Psi:= \sqrt{Y}$ is that it allows us to easier recognize the divergence structure 
of the equation for $W$ which is essential in establishing our energy estimate in the soliton region (see in Section \ref{subsec-tip-energy}). 

\sk

Recall  the definition of the norm $\| \cdot \|_{2, \infty}$ in the tip region is given by   \eqref{eqn-normt0}-\eqref{eqn-normt}  in Section \ref{subsec-tip}. 
The main goal in this section is to establish the following estimate.

\begin{proposition}\label{prop-tip} There exist $\theta$ with
  $0 < \theta \ll 1$, $\tau_0 \ll -1$ and $C< +\infty$ such that
  \begin{equation}
    \label{eqn-tip}
    \| W_T \|_{2,\infty} \leq \frac{C}{\sqrt{|\tau_0| } } \, \| W \, \chi_{_{[\theta, 2\theta]} } \|_{2,\infty}
  \end{equation}
  holds. Constant $C$ is a uniform constant, independent of $\tau_0$, as long as $\tau_0 \ll -1$.
\end{proposition}

To simplify the notation throughout this section we will denote $Y_1$ by $Y$ and $\Psi_1$ by $\Psi$.  Also, we will denote
$Y_2^{\beta \gamma}$ by $\bar Y$ and $\Psi_2^{\beta\gamma}$ by $\bar \Psi$.  The proof of this Proposition will be based on a Poincar\'e
inequality for the function $W_T$ which is supported in the tip region.  These
estimates will be shown to hold with respect to an appropriately chosen weight
$e^{\mu(u,\tau)}\, du$, where $\mu(u,\tau)$ is given by \eqref{eqn-weight1}-\eqref{eqn-weight2}.  We
will begin by establishing various properties on the weight $\mu(u,\tau)$, especially estimates
which we later use in this section.  We
will continue with the  Poincar\'e inequality and the energy estimate. Note that the energy estimate will require careful integration
by parts which is based on the divergence structure of the equation for $W$ with respect to our  appropriately  defined weight
$\mu(u,\tau)$. 
This  estimate is quite more delicate than its analogue in \cite{ADS2}. 
The  proof of Proposition \ref{prop-tip}
will be finished  in Section \ref{subsec-tip-proof}. 

\subsection{Properties of the  weight $\mu(u,\tau)$}

Let  $\zeta (u)$ be a  nonnegative smooth decreasing function defined on $u \in (0, \infty)$ such that 
$$\zeta(u)=1, \quad \mbox{for} \,\, u \geq \theta/2 \qquad \mbox{and}  \quad \zeta(u)=0, \quad \mbox{for}\,\, u \leq \theta/4.$$
Such a function can be chosen to satisfy the derivative estimate $|\zeta'(u)| \leq 5 \theta^{-1}$. 

For our given solution $u(\sigma,\tau)$ which after the coordinate change gives rise to $\sigma(u,\tau)$
and $Y(\sigma,\tau) := u_{\sigma}^2(u,\tau)$ (recall that we have dropped the index and denote $\sigma, Y_1$ 
by  $\sigma, Y$ respectively) we define our {\em weight}   $\mu(u,\tau)$ in the {\em tip region} as in \eqref{eqn-weight1} where
we recall that 
$$
\mu_u := \zeta (u) \, \big ( - \frac{\sigma^2(u,\tau)}{4} \big ) _u + (1-\zeta(u)) \, \frac{1-Y(u,\tau)}{ u \, Y(u,\tau)}.
$$

Note that since $\zeta \equiv 1$ for $u \geq \theta/2$, we have $\mu(u,\tau) =  - \frac{\sigma^2(u,\tau)}{4} $
in this region, hence it coincides with our weight in the cylindrical region. This is important as our norms in
the intersection of the cylindrical and tip regions need to coincide. 

\sk 
In this section we will prove sharp  estimates on  our  weight $\mu$ and its  derivatives which  will be used  in the following two sub-sections. Crucial role will play the convexity  estimate which is shown in the Appendix, Proposition \ref{claim-2}, namely that 
\be\label{eqn-concave} 
(u^2)_{\sigma\sigma} \leq 0, \qquad \mbox{on} \,\, u \geq \frac L{\sqrt{|\tau|}}
\ee
holds for $L \gg 1$ and $\tau \leq \tau_0 \ll -1$. We will also use its 
 consequence \eqref{eqn-good}. To facilitate future references, let us remark that 
\eqref{eqn-concave} expressed in terms of $Y:=u_\sigma^2$ implies that 
\be\label{eqn-concaveY}
u\, Y_u + 2 \, Y \leq 0, \qquad \mbox{on} \,\, \collar_{L,\theta}
\ee
holds for $L \gg 1$ and $\tau \leq \tau_0 \ll -1$. Also, throughout this whole  section we will use the  bound 
\be\label{eqn-diam5}
1- \frac \eta{10}  \leq \frac{\sigma^2}{4|\tau|} \leq 1+ \frac \eta{10}, \qquad \mbox{on}  \,\, \tip_\theta
\ee
which holds for $\tau  \leq \tau_0 \ll -1$ and $\theta=\theta(\eta)$ sufficiently small. This bound is an immediate  consequence of Theorem \ref{thm-asym}. 
 
\sk

\begin{lemma}\label{lemma-good1} Fix $\eta >0$. There exist $\theta >0$ small,  $L \gg 1$ and $\tau_0 \ll -1$ such that 
\be\label{eqn-good1}
(1-\eta)\, \mu_u \leq \frac{1-Y}{u\,  Y} \leq (1+\eta)\, \mu_u
\ee
and
\be\label{eqn-good5} 
(1-\eta)\, |\tau| \leq \frac{ 1}{u^2 \, Y} \leq (1+\eta)\, |\tau|
\ee
hold  on $\collar_{L,\theta}$  and $\tau \leq \tau_0$.

\end{lemma} 

\begin{proof} Both bounds follow from our crucial estimate \eqref{eqn-good}, which says that 
\be\label{eqn-good10}
 1- \frac \eta{10} \leq - \frac{\sigma u u_\sigma}{2} \leq  1 + \frac \eta{10}
 \ee
holds  on  $\collar_{L,\theta}$ for $\theta >0$ small,  $L \gg 1$ and $\tau \leq \tau_0 \ll -1$. 

\sk 
The  first estimate  simply follows from the definition of $\mu_u$, and the fact that 
$$\big ( - \frac{\sigma^2}{4} \big ) _u = - \frac{\sigma \sigma_u}2 = - \frac{ \sigma \, u u_\sigma}2 \, \frac 1{u \, u_\sigma^2} = 
- \frac{\sigma u u_\sigma}2 \, \frac{1}{u \, Y}$$
which with the aid of \eqref{eqn-good10} gives 
$$\Big | \mu_u - \frac{1-Y}{uY} \Big | \leq  \Big | \big ( - \frac{\sigma^2}{4} \big ) _u - \frac{1-Y}{uY} \Big | \leq 
\frac \eta{10}\,  \frac{1}{Yu} + \frac 1{u} \leq \eta \, \frac{1-Y}{Yu}$$
where in the last inequality we used that $Y \ll \eta$ in the considered region. 
\sk

For the second inequality \eqref{eqn-good5}, we observe that 
$$|\tau | \, u^2\,  Y = \frac{4 |\tau|}{\sigma^2} \,  \frac{\sigma^2 u^2 u_\sigma^2 }{4}$$
hence the bound readily follows by \eqref{eqn-good10} and \eqref{eqn-diam5}.

\end{proof}

\begin{corollary} There exists an absolute $\theta \ll 1$ and $\tau_0 \ll -1$ such that  
\be\label{eqn-mub1}
\mu(u,\tau) \leq - \frac{|\tau|}4
\ee
holds on $\tip_\theta $ for $\tau \leq \tau_0 \ll -1$. 
\end{corollary} 
\begin{proof}
We claim that $\mu_u \ge 0$ in $\mathcal{T}_{\theta}$. Indeed, in $\mathcal{K}_{\theta,L}$ this is true by \eqref{eqn-good1}. In $\mathcal{S}_L$,  $\zeta = 0$ and thus $\mu_u = \frac{1-Y(u,\tau)}{uY(y,\tau)} \ge 0$. Hence, in $\mathcal{T}_{\theta}$ we have
\[\mu(u,\tau) = -\frac{\sigma^2(\theta,\tau)}{4} - \int_u^{\theta} \mu_u(u',\tau)\, du' \le -\frac{\sigma^2(\theta,\tau)}{4}.\]
The claim \eqref{eqn-mub1} now immediately follows by \eqref{eqn-diam5}, by taking $\eta$ there sufficiently small.
\end{proof}

We will  next estimate $u_\tau$ in the region $\theta/4 \leq u \leq 2\theta$. This will be used later  to estimate the time derivative $\mu_\tau(u,\tau)$ of our weight.  

\begin{lemma} Fix $\eta >0$. There exist $\theta >0$ small,  $L \gg 1$ and $\tau_0 \ll -1$ such that the bounds 
\be\label{eqn-utau}
|u_\tau(\sigma,\tau)| \leq  \frac {\eta}u \qquad \mbox{and} \qquad  \Big | \big ( - \frac{\sigma^2(u,\tau)}{4} \big )_\tau \Big |\leq 
\frac \eta { u^2  Y }
\ee
hold in the collar region $\collar_{\theta,L}$  and $\tau \leq \tau_0$.  

\end{lemma}

\begin{proof} We recall equation \eqref{eqn-u}, namely that 
$$ u_\tau = u_{\sigma\sigma} - \frac{\sigma}{2} \, u_{\sigma}
 -  J(\sigma,\tau) \, u_{\sigma}  +  \frac{u_{\sigma}^2}{u}  -  \frac 1u + \frac u2.$$
Using the bounds 
$$0 \leq -J \leq 2\, \frac {|u_\sigma|}u, \quad |u_{\sigma\sigma}| \leq \frac{\eta}{10 \, u}, \quad 
\big | 1 + \frac{ \sigma \, u u_\sigma}2 \big | \leq  \frac{\eta}{10}, \quad |u_\sigma|\leq \frac \eta{20}$$
which hold for $\tau_0 \ll -1$,  we conclude that
$$|u_\tau| \leq \frac{\eta}{ 2u} + \frac 1u \big | \frac{\sigma u u_\sigma}2  + 1  \big | \leq \frac {\eta}{2u} + \frac{2\theta^2}{u} \leq  \frac{\eta}{u},$$
if we take $\theta$ so that $4\theta^2 \leq  \eta$.
Now, using this bound, we compute 
$$\Big | \Big ( - \frac{\sigma^2}{4} \Big )_\tau \Big | = \Big |  \frac{\sigma \sigma_\tau}{2}  \Big | =
\Big |  \frac{\sigma u_\tau}{2 u_\sigma}  \Big | = \frac 1{u\, u_\sigma^2} \, \Big |  \frac{\sigma \, u u_\sigma}{2}  \Big | \, |u_\tau| 
\leq \frac{\eta}{u^2 \, Y}.$$

\end{proof}

We will also need an estimate for the time $\mu_\tau(u,\tau)$ derivative of our weight in the whole tip region $\tip_\theta$ where 
$u \leq 2\theta$.  But before we do so, we will  estimate
$Y_u, Y_{uu}$ as well as $Y_\tau$ in this region. 
Recall that $Y$ satisfies  equation \eqref{eqn-Y}. 
Our estimates will be based on the bounds in Proposition \ref{claim-1}, namely that  given an $\eta >0$ there exist $\theta$ and $L \gg 1$ and $\tau_0 \ll -1$ such that for $\tau \le \tau_0$,
\be\label{eqn-uss}
0 \leq -  u\, u_{\sigma\sigma}  \leq \eta \qquad  \mbox{and} \qquad u^2 |u_{\sigma\sigma\sigma}| \leq  \eta
\ee
hold on the collar  region $\collar_{\theta, L}$.

\begin{lemma}\label{lemma-Yb}  Fix $\eta >0$.   There exist $\theta >0$ small,   $L \gg 1$ and $\tau_0 \ll -1$ such that $Y$ satisfies 
derivative bounds 
\be\label{eqn-Yb}
|Y_u| \leq  \frac {\eta}{ u},  \qquad |Y_{uu} | \leq \frac{\eta}{u^2 \sqrt{Y}}, \qquad Y_\tau \leq  \frac{\eta \, \sqrt{Y}}{u^2} + \frac{\eta}{4},
\qquad |Y_\tau| \leq  \frac{\eta }{u^2}
\ee
hold on  $\collar_{\theta,L}$ and for $\tau \leq \tau_0$.  It follows that $\Psi:=\sqrt{Y}$ satisfies the
derivative bounds 
\be\label{eqn-Psib}
|\Psi_u| \leq  \frac {\eta}{ u \Psi},  \qquad |\Psi_{uu} | \leq \frac{\eta}{u^2 \Psi^2}, \qquad \Psi_\tau  \leq \frac{\eta}{u^2} + \frac{\eta}{8\Psi}, \qquad |\Psi_\tau| \leq 
 \frac{\eta }{u^2\, \Psi}
\ee
in the same region. 
\end{lemma} 

\begin{proof}
Recall that
${\ds Y_u = 2 u_\sigma  u_{\sigma u} =  2 u_{\sigma\sigma}}$ and ${\ds Y_{uu} = \big ( u_{\sigma\sigma} )_u =  u_{\sigma\sigma\sigma} \, \sigma_u.}$
Hence, the first two bounds in \eqref{eqn-Yb} readily  follow from  the two bounds in \eqref{eqn-uss}.   
For the third bound we combine \eqref{eqn-Y} with  \eqref{eqn-concaveY} and the two bounds we just obtained 
(with $\eta/2$ instead of $\eta$)   to conclude 
$$Y_\tau \leq Y\, Y_{uu}  - \frac u2 \, Y_u \leq \frac{\eta \sqrt{Y}}{2\, u^2} + \frac{\eta}4.$$
The last bound in \eqref{eqn-Yb} follows the same way, using   the bounds we just obtained  (with $\eta/10$ instead of $\eta$)  and 
$Y \ll 1$.  
The bounds \eqref{eqn-Psib} readily follow from \eqref{eqn-Yb}. 

\end{proof}

We will next combine the estimates   above  to obtain the  bounds for 
$\mu_{uu}(u,\tau)$ and $\mu_\tau(u,\tau)$ which will  be used in the  the next two subsections. 

\begin{lemma}\label{cor-muu} Fix $\eta >0$. There exist $\theta>0$, $L  \gg 1$ and $\tau_0 \ll -1$ such that the bound 
\be\label{eqn-muu}
\mu_{uu}  \leq  \eta \,   \mu_u^2
\ee
holds    in the collar region  $\collar_{\theta, L} $ and  for all $\tau \leq \tau_0 \ll -1$. 
\end{lemma} 

\begin{proof} Fix  $\eta >0$, and   assume that $\sigma >0$ and $u_\sigma <0$,   as the case $\sigma <0 $ and  $u_\sigma >0$ is similar.  
Recall that $\mu(u,\tau)$ is defined so that it satisfies \eqref{eqn-weight2}. We know that the bounds   \eqref{eqn-good}, \eqref{eqn-good1} and  
\eqref{eqn-Yb} hold on $\collar_{\theta, L}$,    for  $\bar \eta:= \eta/10$ and  $\theta$ small, $L$  large and  $\tau \leq \tau_0 \ll -1$.   Using these bounds, we find that  in the region where $\mu_u = \big (   - \frac{\sigma^2}{4} \big )_u$,
we have $$\mu_{uu}  =  \big ( - \frac{\sigma \sigma_u}{2} \big ) _{u} = - \frac {\sigma \sigma_{uu} + \sigma^2_u}2 \leq  - \frac {\sigma \sigma_{uu}}2 = \frac{\sigma u_{\sigma\sigma}}{2u_\sigma^3}
\leq \frac 12 \, \frac{|\sigma u u_\sigma|}{2} \frac {\bar \eta}{u^2\, Y^2} \leq \bar  \eta \, \mu_u^2$$
while in the region where  ${\ds \mu_u = \big ( \frac {1-Y}{uY}  \big )_u}$, we have 
$$\mu_{uu} =  \Big (  \frac{Y^{-1}-1}{u} \Big )_u  =  - \frac{1- Y}{u^2 Y } -  \frac {Y_u}{u\, Y^2 } \leq  \frac {|Y_u|}{u\, Y^2 }  \leq   \frac{\bar \eta}{u^2\, Y^2} \leq 2 \bar \eta \, \mu_u^2.$$
We conclude, with  the aid of \eqref{eqn-good1}, \eqref{eqn-good}, applied with  $\eta/10$ instead of $\eta$,
 (using also the bound $\zeta'(u) \leq 5 \theta^{-1} \leq 20 u^{-1}$ and $Y \ll 1$) that 
$$\mu_{uu} \leq (1-\zeta) \, \frac{\eta}{u^2\, Y^2}  + \zeta'(u) \Big |  \big ( - \frac{\sigma \sigma_u}{2} \big  ) - \frac{Y^{-1}-1}{u} \Big |  \leq   4 \bar  \eta \, \mu_u^2 \leq  \eta \, \mu_u^2$$
holds on $\collar_{\theta,L}$ and  $\tau \leq \tau_0 \ll -1$. 

\end{proof}

\begin{lemma}\label{lemma-mutau} Fix $\eta >0$. There exist $\theta>0$  and $\tau_0 \ll -1$ such that the bound 
\be\label{eqn-mutau}
\mu_\tau(u,\tau)  \leq  \eta \, |\tau| \big ( 1 + \rho^{-1}\,  \chi_{[0,1]}(\rho) \big ) \leq C_0 \, \frac {\eta}{u^2 Y}, \qquad \rho:= u\, \sqrt{|\tau|}
\ee
holds  in the whole tip region $\tip_\theta$,   for all $\tau \leq \tau_0 \ll -1$. 
\end{lemma} 

\begin{proof}
We use the  definition of $\mu(u,\tau) $ in \eqref{eqn-weight1}-\eqref{eqn-weight2} and that $\zeta \equiv 1$ for $u \geq \theta/2$. 
Integration by parts gives 
\be\label{eqn-mutau5}
\begin{split}
 \mu_\tau &=   \big ( -\frac{\sigma^2(\theta,\tau)}4 \big )_\tau + \int_\theta^u  \zeta \, \big ( - \frac{\sigma^2(u,\tau)}4 \big )_{u\tau }  +
   (1-\zeta) \,  \big ( \frac {Y^{-1}(u,\tau) -1}{u}  \big )_{\tau}    \, du\\
& =    \zeta \, \big ( - \frac{\sigma^2(u,\tau)}4 \big )_{\tau }   + \int_u^\theta  \zeta' \, \big ( - \frac{\sigma^2}4 \big )_{\tau }  \, du  + 
\int_u^\theta (1-\zeta)\, \frac {Y_\tau}{u Y^2} \, du
 \end{split}
 \ee
where, to simplify the notation,  we will  denote the variable of integration by $u$ (instead of $u'$) when there is no danger of confusion. 

\sk

Fix  $\eta >0$ and small. 
We will treat  separately the two  cases of the collar and soliton regions, $L/\sqrt{|\tau|} \leq u  \leq \theta$ and $u \leq L/\sqrt{|\tau|}$,
respectively.  

\sk
\sk
\noindent{\em Case 1:  Given $\eta >0$, there exists $0 < \theta \ll 1$ and  $L \gg 1$ such that \eqref{eqn-mutau} holds on 
 $L/\sqrt{|\tau|} \leq u  \leq 2 \theta$.}  

\sk

Observe first that on the region $u \geq \theta$, we have $\zeta(u)=1, \zeta'(u)=0$, hence the desired  bound simply follows 
from  \eqref{eqn-utau} and \eqref{eqn-good5}   for $\theta $ small and   $\tau \leq \tau_0(\theta)  \ll  -1$.   

\sk

Assume now that $u \leq  \theta$. Then the second bound in \eqref{eqn-utau} (with $\eta$ replaced by $\eta/10$)  and \eqref{eqn-good5}
readily give that
\be\label{eqn-zeta15}  
 \Big |   \zeta \, \big ( - \frac{\sigma^2(u,\tau)}4 \big )_{\tau }   + \int_u^\theta  \zeta' \, \big ( - \frac{\sigma^2}4 \big )_{\tau }  \, du  \Big |
 \leq \frac {\eta}{4} \,|\tau| 
\ee
holds on the {\em  whole tip region $\tip_\theta$},  for  $\theta$ small and $\tau \leq \tau_0 \ll-1$ (recall that both $\zeta$ and $\zeta'$ are zero for $u \leq \theta/4$).  

\sk
Let's now  look  at the last integral in \eqref{eqn-mutau5}.  Using   equation   \eqref{eqn-Y} to substitute  for  $Y_\tau$
and integrating by parts we obtain 
\bee
\begin{split}
\int_u^\theta (1-\zeta)\, \frac {Y_\tau}{u Y^2} \, du &= \int_u^\theta (1-\zeta)  
\Big 
(  \frac{Y_{uu}}{u Y}   - \frac 12\, \frac{Y_u^2}{u Y^2}     + 
\frac{1 - Y}{u ^3 Y^2} \, \big ( u Y_u   +  2 Y \big )    -   \frac 12 \, \frac { Y_u}{Y^2}   \Big ) \, du   \\
&=\int_u^\theta (1-\zeta)  
\Big 
(  \frac 12\, \frac{Y_u^2}{u Y^2}  + 
\frac{1 - Y}{u ^3 Y^2} \, \big ( u Y_u   +  2 Y \big )     -   \frac 12 \, \frac { Y_u}{Y^2}  \Big ) \, du \\
&+ \int_u^{\theta} \zeta_u \frac{Y_u}{uY}\, du + (1 -\zeta(u))\, \frac{Y_u}{uY} (u,\tau)
\end{split}
\eee
To obtain the desired bound, we cannot just use the estimates \eqref{eqn-Yb}, we need to use careful integration by parts together
with \eqref{eqn-concaveY}. In fact we will use the negative term  $u\, Y_u + 2 Y \leq 0$
 in our favor to bound the first term on  the right hand side of the last formula. 
To this end, we write  
\bee
\begin{split}
\int_u^\theta (1-\zeta)\, \frac {Y_\tau}{u Y^2} \, du = &\int_u^\theta (1-\zeta)  
\Big (  \frac 12\, \frac{Y_u}{u^2 Y^2}  \big ( u Y_u   +  2 Y \big )    + 
\frac{1 - Y}{u ^3 Y^2} \, \big ( u Y_u   +  2 Y \big )  \Big ) \, du \\
& - \int_u^\theta (1-\zeta)\,  \Big (   \frac { Y_u}{u^2 Y}    +   \frac 12 \, \frac { Y_u}{Y^2}  \Big ) \, du  + \int_u^{\theta} \zeta_u \frac{Y_u}{uY}\, du \\
&+ (1 -\zeta(u))\, \frac{Y_u}{uY} (u,\tau)
\\
=&\int_u^\theta (1-\zeta) \,  \frac 1{u^3Y^2} \, \Big (  \frac 12\,  u^2 Y_u    + 1 - Y \Big ) \,  \big ( u Y_u   +  2 Y \big )  \, du \\
& - \int_u^\theta (1-\zeta)\,  \Big (   \frac { Y_u}{u^2 Y}    +   \frac 12 \, \frac { Y_u}{Y^2}  \Big ) \, du + \int_u^{\theta} \zeta_u \frac{Y_u}{uY}\, du \\
&+ (1 -\zeta(u))\, \frac{Y_u}{uY} (u,\tau)\\
&\leq  \int_u^\theta  \frac { (-Y_u)}{u^2 Y}  +   \int_u^\theta   \frac {(- Y_u)}{ 2Y^2}  du  + \int_u^{\theta} \zeta_u \frac{Y_u}{uY}\, du \\
&+ (1 -\zeta(u))\, \frac{Y_u}{uY} (u,\tau), \,\,\, \mbox{since} \,\,\, 0\le \zeta \le 1,
\end{split}
\eee
where in the last inequality we used  ${\ds  \frac 12\,  u^2 Y_u    + 1 - Y \geq -\eta u + 1 - Y >0}$ and  $uY_u + 2Y \leq 0$.  By   \eqref{eqn-good5} (with $\eta = 1$), we know    that throughout collar region  the values
of $1/u^2 Y$ are comparable. Hence, we can pull out the  $1/{u^2 Y(u,\tau)}$ from the first integral 
in  the last line  above, evaluated at the end point $u$,  and use  $Y_u \leq 0$,  $u \leq \theta <1$    to  obtain that 
\bee
\begin{split}
\int_u^\theta   \Big (   \frac { (-Y_u)}{u^2 Y}    +   \frac 12 \, \frac {(- Y_u)}{Y^2}  \Big ) \, du 
& \leq   \frac 2{u^2  Y(u,\tau)} \, \int_u^\theta (-Y_u) \, du'   +  \frac 12 \, \int_u^\theta \Big ( \frac 1Y \Big )_u \, du' \\
& \leq \frac  2{ u^2} + \frac 1{2 \, Y(\theta,\tau)} \leq |\tau| \, \big ( \frac{2}{L^2} + \theta^2 \big )  \leq  \frac \eta{20} \, |\tau|
\end{split}
\eee
holds,  for    $L \gg 1$ and  $0 < \theta \ll 1$ both  depending on $\eta$.  Furthermore, by \eqref{eqn-Yb} we have
\[\int_u^{\theta} \zeta_u \frac{Y_u}{uY}\, du + (1 -\zeta(u))\, \frac{Y_u}{uY} (u,\tau) <  \frac \eta{20} \, |\tau|.\]
We conclude that 
\be\label{eqn-zeta5} 
\begin{split}
\int_u^\theta (1-\zeta)\, \frac {Y_\tau}{u \, Y^2} \, du   \leq \frac \eta{10} \, |\tau|
\end{split}
\ee
Finally combining the  bounds   \eqref{eqn-zeta15} and \eqref{eqn-zeta5} we conclude that the first bound in \eqref{eqn-mutau} holds,
in $\collar_{\theta,L}$,  provided $L \gg 1$, $0 < \theta \ll 1$ and   $\tau  \leq \tau_0 \ll -1$. 

\sk
\sk
\noindent{\em Case 2:  For the given $\eta >0$, let $L \gg 1$ be such that \eqref{eqn-mutau} holds in $\collar_{L,\theta}$
and for $\tau \leq \tau_0 \ll -1$.  
Then, we may choose  $\tau_0 \ll -1$ such that \eqref{eqn-mutau} also holds on $S_L$ and $\tau \leq \tau_0$.}

\sk
\sk 
Since \eqref{eqn-zeta15} holds in the whole tip region, using also \eqref{eqn-zeta5} at $u=L/\sqrt{|\tau|}$ and that $\zeta=0$ on $u < \theta/4$,  we obtain 
\be\label{eqn-zeta35} 
\begin{split}
\mu_\tau &\leq \frac \eta{4}\, |\tau|  +  \int_{L/\sqrt{|\tau|}}^{\theta}  (1-\zeta)\, \frac {Y_\tau}{u \, Y^2} \, du
+  \int_u^{L/\sqrt{|\tau|}}   (1-\zeta)\, \frac {Y_\tau}{u \, Y^2} \, du\\
&\leq   \frac \eta{2}\, |\tau| +  \int_u^{L/\sqrt{|\tau|}}  \frac {|Y_\tau|}{u \, Y^2} \, du.
\end{split} 
\ee
Recall that in the soliton region $S_L$, the rescaled solution  $Z(\rho, \tau) := Y(u,\tau) $, with $\rho := u\, \sqrt{|\tau|}$ converges to the Bowl soliton 
$Z_0(\rho)$, which implies (using \eqref{eqn-Z}) that 
\be\label{eqn-Ytau}
 \frac 1{|\tau|} Y_\tau = \frac 1{|\tau|} \, \big ( Z_\tau - \frac {\rho}{2|\tau|} Z_\rho \big )  \to 0, \qquad \mbox{as} \,\, \tau \to -\infty.
 \ee
uniformly on $\rho \in [0,L]$. We conclude that  for our given  constant $L >0$ and any $\eta' >0$,  there exists $\tau_0 \ll -1$, such that for all $\rho \leq L$, we have 
$$\int_{L/\sqrt{|\tau|}}^u  \frac {|Y_\tau|}{u \, Y^2} \, du    \leq \eta' \,   |\tau| \,  \int_\rho^L  \frac 1{\rho \, Z^2_0} \, d\rho
\leq \eta' \,   |\tau| \, \frac{C(L)}{\rho} \leq \frac{\eta}{10} \frac{|\tau|}{\rho}$$
by  taking  $\eta'=\eta/(10 \,  C(L))$, for our given $\eta >0$. Plugging  this last estimate in \eqref{eqn-zeta35} concludes the first bound in 
\eqref{eqn-mutau} holds  in the soliton region. 

\sk

Lets now check that the second bound in \eqref{eqn-mutau} holds for $\tau \ll \tau_0 \ll -1$.  To this end, we fix $L_0$ universal constant so that
${\ds |\tau| \leq \frac 2{u^2Y} }$ holds on $u \geq L_0/|\sqrt{|\tau|}$ (we use again \eqref{eqn-good5}). 
On the other region where $\rho  \leq L_0$, we  use $Z(\rho,\tau) = Y (u,\tau)\leq 1$ to get 
$\rho^2 Z \,  \big ( 1 + \rho^{-1}\,  \chi_{[0,1]}(\rho) \big ) \leq C_0=C(L_0)$, which readily gives the desired bound.

\end{proof}

%
%
%
%

\subsection{Poincar\'e inequality in the tip region}

Our goal in this section is to derive the following Poincar\'e inequality:

\begin{proposition}[Poincar\'e inequality] 
  \label{prop-Poincare}
 There exists an absolute  constant   $C_0 > 0$,  a small absolute constant $\theta_0$,  and $\tau_0 \ll -1$, 
 such that the inequality 
  \begin{equation}
    \label{eq-Poincare}
  \int \mu_u^2  \, f^2  \, e^{-\mu(u,\tau)}\, du 
    \le C_0\, \Big ( \int  f_u^2 \, e^{-\mu(u,\tau)}\, du + 
 \int  \frac{f^2}{u^2}\, e^{-\mu(u,\tau)}\, du  \Big )
  \end{equation}
  holds, for any smooth  compactly supported function $f$ in $\tip_{\theta_0}$  and for all  $\tau \le \tau_0$.  
\end{proposition}

\begin{proof} 
First we show the following weighted estimate
  \be
    \label{eqn-poinc0}
 \frac 12    \int  \mu_u^2\,   f^2  \, e^{-\mu} \, du -   \int  \mu_{uu}  f^2 \, e^{-\mu}   \leq   2  \int    f^2_u \, e^{-\mu} du
  \ee
 which simply follows by completing the square and integrating by parts. To this end, we write 
\[
  0\leq \bigl( f_u  -   \frac  {\mu_u} 2  \, f \bigr)^2 = f_u^2 + \frac{\mu_u^2}{4} f^2  -  \mu_u  f  f_u =   f_u^2 + \frac{\mu_u^2}{4} f^2 -    \frac 12 \, \mu_u  (f^2 )_u\]
and integrate by parts to obtain 
\begin{align*}
  \int  \bigl(  f_u^2 +   \frac{\mu_u^2}{4} f^2\bigr)\, e^{-\mu}  du  &\geq    \frac 12  \int \mu_u  (f^2 )_u \, e^{-\mu}  du\\
 & =   \frac 12 \int \big ( - \mu_{uu} +  \mu_u^2  \big ) \,  f^2 \, e^{-\mu}  du.  
\end{align*}
Rearranging terms leads to 
\bee
\frac 14 \int \mu_u^2   f^2 \, e^{-\mu}  du -  \frac 12 \int  \mu_{uu}  f^2 \, e^{-\mu}   \leq   \int    f^2_u \, e^{-\mu} du.
\eee
which readily implies \eqref{eqn-poinc0}. 
\sk 
\sk 

We will now  apply \eqref{eqn-poinc0} to our special case where the weight $\mu(u,\tau)$ is given by \eqref{eqn-weight1}-\eqref{eqn-weight2}.   
Let $\theta_0, L_0$ be universal constants such that Corollary \ref{cor-muu} holds with $\eta:=1/8$, namely that ${\ds \mu_{uu} \leq \frac 18 \, \mu_u^2} $  holds on $L_0/\sqrt{|\tau|} \leq u \leq 2\theta_0$, 
for  $\tau \leq \tau_0 \ll -1$.  To deal with 
the region $u \leq L_0/\sqrt{|\tau|}$ we will next consider the change of variable  $\rho:= u\, \sqrt{|\tau|}$
and we will use the  $C^\infty$ convergence of $Z(\rho, \tau):= Y(u, \tau) $ to the Bryant  soliton $Z_0(\rho)$. Since, 
$$ 2 \frac {|\Psi_u|}{u\, \Psi^3} =   \frac {|Y_u|}{u\, Y^2} =  |\tau| \, \frac {|Z_\rho|}{\rho\, Z^2}$$
the convergence of $Z(\rho,\tau) \to Z_0$ implies that 
$$ \frac {|Z_\rho|}{\rho\, Z^2} \leq \frac{C(L_0)}{\rho} \leq \frac{\bar C(L_0)}{\rho^2}$$
where $\bar C(L_0) = L_0 \, C(L_0)$.  Hence, 
\be\label{eqn-mu136}
\mu_{uu} \leq 2 \frac {|\Psi_u|}{u\, \Psi^3}  \leq |\tau|\,  \frac{\bar C(L_0)}{\rho^2} = \frac{\bar C(L_0)}{u^2}
\ee
holds for $u \leq L_0/\sqrt{|\tau|}$. Combining the two estimates \eqref{eqn-muu} with \eqref{eqn-mu136} finally gives  
the bound 
$$\mu_{uu} \leq \frac 18 \mu_u^2 +  \frac{\bar C(L_0)}{u^2}$$
which holds on the whole tip region $\tip_{\theta_0}$. The last estimate combined  with \eqref{eqn-poinc0} readily gives \eqref{eq-Poincare}.

\end{proof}

\sk
\sk

\subsection{Energy estimate for $W_T$}\label{subsec-tip-energy} We  will next derive an  energy estimate for  difference $W:= \Psi_1 - \Psi_2^{\beta\gamma} $ in the weighted $L^2$-space with respect to our weight $e^{\mu}\, du$,  as defined 
in Section \ref{subsec-tip} (see \eqref{eqn-weight1}-\eqref{eqn-weight2} for the definition of $\mu:=\mu(u,\tau)$ and    \eqref{eqn-normt00}-\eqref{eqn-normt} 
for the definition of the $L^2$-norm).  Recall that  we 
denote $\Psi_1, Y_1$by  $\Psi, Y$  and 
$Y_2^{\beta\gamma}, \Psi_2^{\beta \gamma}$ by $\bar Y, \bar \Psi$.

\sk 
A direct calculation shows that both $\Psi (u,\tau), \bpsi(u,\tau)$  satisfy  the equation
\begin{equation}
\label{eq-Psi}
\Psi^{-2}  \Big ( \Psi_\tau + \frac u2\,  \Psi_u \Big )  = \Psi_{uu} + \frac{\Psi^{-2}-1}{u} \, \Psi_u + \frac{\Psi^{-1} - \Psi}{ u^2 } 
\end{equation}
since $Y(u,\tau), \bar Y(u,\tau)$ satisfy  \eqref{eqn-Y}. In fact the reason for considering $\Psi:=\sqrt{Y}$ instead of $Y$ is that  the equation \eqref{eq-Psi} 
for $\sqrt{Y}$ is simpler than that of $Y$ and,  in particular,  it has a  nice divergence structure which will help us derive a sharp energy estimate, suitable for our purposes.

\sk
\sk 
Let $W_T = W \vft$, where  $\vft(u)$ is the cut-off function defined in \eqref{eqn-cutofftip}. In this subsection we will derive a weighted   energy estimate for $W_T$  
and combine it with  our  Poincar\'e inequality to obtain  following differential inequality.

\begin{prop}[Integral Differential inequality]
\label{prop-energy}
There exist absolute constants  $\theta >0$ small, $\lambda >0$ and $\tau_0 \ll -1$ and a constant $C(\theta)$ such that 
\be\label{eqn-diff-ineq}
\begin{split}
\frac {d}{d\tau} \int  W^2_T  \, \Psi^{-2}  \emu   \leq   - 2 \lambda  |\tau|  \int  W^2_T \, \Psi^{-2}  \emu  + C(\theta) \int_\theta^{2\theta}  W^2  \Psi^{-2}  \emu. 
\end{split}
\ee
holds for all  $\tau \leq \tau_0$. 
\end{prop}

\begin{proof}

We will begin by computing the equation that the  difference $W:= \Psi - \bpsi$ satisfies. 
Subtracting the equation for $\bpsi$ from the equation for $\Psi$, we find 
\bee
\begin{split}
\frac 1{\Psi^2}   \Big ( W_\tau + \frac u2\,  W_u  \Big )  &= W_{uu} + \frac{\Psi^{-2}-1}{u} \, W_u  +  \frac{\Psi^{-2}-\bpsi^{-2}}{u} \, \bpsi_u\\
& - 
\Big (  \frac{1}{u^2\Psi \bpsi}  + \frac 1{u^2} \Big )  W    - \big ( \Psi^{-2} - \bpsi^{-2} \big ) \, \Big ( \bpsi_\tau + \frac u2\,  \bpsi_u  \Big ) .
\end{split}
\eee
We can further express 
$$ \Psi^{-2} - \Psi_{2}^{-2}  = \frac{\bpsi^2 - \Psi^2}{\Psi^2 \bpsi^2} = - \frac{\Psi + \bpsi}{\Psi^2 \bpsi^2} \, W  =
- \frac{2 \Psi_{12}}{\Psi^2}  W $$
where, to simplify the notation,  we have denoted by  $\Psi_{12} := \frac 12  \frac{\Psi +\bpsi}{\bsi^2}.$
Under this notation, the above equation for $W$ becomes 
\bee
\begin{split}
\frac 1{\Psi^2}  \big (  W_{\tau} + \frac u2 \, W_u \big ) = &W_{uu} + \frac{\Psi^{-2} -1}{u} \, W_u - 
\Big ( \frac{2 \bpsi_u \Psi_{12}}{u \Psi^2} +  \frac{1}{u^2\Psi\bpsi}  + \frac 1{u^2} \Big )  W  \\
 & +  \frac {2 \Psi_{12}}{\Psi^2}\,  \Big ( \bpsi_\tau + \frac u2\,  \bpsi_u  \Big )  \, W 
\end{split}
\eee
where we have arranged terms in such a way that the terms in the second line will be considered error terms. 

\sk

In an attempt to recognize a divergence structure in the above equation for $W$,  we next observe that  since $\Psi \approx \bpsi \approx \Psi_{12}^{-1}$ we have
$$ \frac{2 \bpsi_u}{u \Psi_{12}^{-1} \Psi^2}+  \frac{1}{u^2\Psi\bpsi}   + \frac 1{u^2}  \approx   \frac {2 \Psi_{u}} {u\, \Psi^3}  +  \frac{\Psi^{-2} - 1}{u^2} +  
\frac 2{u^2}=     -   \Big ( \frac{\Psi^{-2} - 1}{u}  \Big )_u +  \frac 2{u^2}.$$
It follows that the equation above becomes
\bee\label{eqn-W1}
\begin{split}
\frac 1{\Psi^2}  \big (  W_{\tau} + \frac u2 \, W_u \big )  &= W_{uu} + \frac{\Psi^{-2} -1}{u}  \, W_u + 
 \frac{\partial}{\partial u}     \Big ( \frac{\Psi^{-2} - 1}{u}  \Big ) W  - \frac 2{u^2} \, W   + \frac 1{\Psi^2} \, \cB  \, W
\end{split}
\eee
where 
\bee
\cB := \Big (  \frac {2 \Psi_{u}} {u\, \Psi}   - \frac{2 \bpsi_u\, \Psi_{12}}{u} +  \frac{\Psi-\bpsi}{u^2 \bpsi} \Big ) 
+ 2\Psi_{12}\,  \Big ( \bpsi_\tau + \frac u2\,  \bpsi_u  \Big ).
\eee
The second and third term on the right-hand side of the above equation  can be combined together as one term in divergence form. This finally gives us the equation
\be\label{eqn-W}
\begin{split}
 \frac 1{\Psi^2} W_{\tau}&= W_{uu}  +
 \frac{\partial}{\partial u}  \Big (  \frac{\Psi^{-2} - 1}{u} \, W \Big ) -   \frac u{2\Psi^2}  \, W_u   - \frac 2{u^2} \, W   + \frac 1{\Psi^2} \, \cB  \, W.
\end{split}
\ee
Also, using  the bounds in \eqref{eqn-Psib} (with $\eta/10$)  and the bounds $\Psi /2 \leq \bpsi \leq 2 \Psi$ (which readily follow 
from \eqref{eqn-good5}) we may  estimate $\cB$ as 
\be\label{eqn-B}
 \cB \leq \frac \eta{u^2 \Psi^2}. 
 \ee

\sk
\sk 
Let $W_T = W \vft$, where  $\vft(u)$ is the cut-off function defined in \eqref{eqn-cutofftip}. To simplify the notation,  we drop the subscript $T$ from $\vft$,
and simply set $\vf:=\vft$.  Since, $\vf$ is independent of $\tau$, integration with respect to our measure $\emu$ and differentiation in time $\tau$ yields
\bee
\begin{split}
\frac 12 \frac {d}{d\tau} \int \frac{W^2_T}{ \Psi^2} \, \emu  
=  \int \frac 1{\Psi^2} \, W W_\tau  \,  \vf^2   \emu  +  \int \big (   \frac 12 \mu_\tau  -  \frac{\Psi_\tau}{\Psi} \big ) \frac{W_T^2}{\Psi^2}  \emu . 
\end{split}
\eee
Using equation    \eqref{eqn-W} while integrating by parts the first two  terms we obtain 
\begin{equation}
\label{eqn-WW1}
\begin{split}
\frac 12 \frac {d}{d\tau} \int \frac{W^2_T }{\Psi^2} &\, \emu  
=-  \int W_u^2  \,  \vf^2 \,  \emu  - \int \Big ( \mu_u  + \frac{\Psi^{-2} -1}u + \frac{u}{2\Psi^2} \, \Big ) W W_u \vf^2 \, \emu \\
&- \int \frac{2}{u^2} \, W_T^2 \, \emu  - \int   \mu_u \,  \frac{\Psi^{-2} -1}u \, W_T^2 \, \emu + \int  \frac 1{\Psi^2} \Big (  \cB  +  \frac 12 \mu_\tau  -  \frac{\Psi_\tau}{\Psi}\Big )\,  W_T^2\,   \emu\\
&- 2 \int W W_u \vf  \vf_u \, \emu - 2 \int  \frac{\Psi^{-2}-1}u \, W^2 \, \vf  \vf_u \, \emu. 
\end{split}
\end{equation}

Set  
\be\label{eqn-G12} 
G_1:=  \mu_u  + \frac{\Psi^{-2} -1}u + \frac{u}{2\Psi^2} \qquad \mbox{and} \qquad G_2 :=  \frac 1{\Psi^2} \, \Big ( \cB  +  \frac 12 \mu_\tau  -  \frac{\Psi_\tau}{\Psi} \Big ).
\ee
Furthermore, use 
 $(W_T)_u = W_u \vf + W \vf_u$ to write
$$\int  G_1\,    W  W_u \, \vf^2 \, \, \emu =  \int  G_1  \,   W_T (W_T)_u\, \emu - \int G_1 \,  W^2 \, \vf \vf_u \, \emu$$
and $(W_T)_u^2  = W_u^2 \vf^2 + 2 W W_u \vf \vf_u +  W^2 \vf_u^2$ to write 
$$-  \int W_{u}^2 \vf^2  \, \emu =  -    \int  (W_T)_u^2\, \emu  +  \int  W^2 \,  \vf_u^2  \, \emu +  2 \int  W W_u \, \vf \vf_u \, \emu.$$
Inserting this in \eqref{eqn-WW1} we obtain (after cancelling terms) 
\begin{align*}\label{eqn-WW2}
\frac 12 \frac {d}{d\tau} \int  \frac{W^2_T}{\Psi^2} \, & \emu  
=-  \int (W_{T})_u^2   \,  \emu  - \int  G_1 W_T  (W_{T})_u   \, \emu \\
&- \int   \mu_u \,  \frac{\Psi^{-2} -1}u \, W_T^2 \, \emu - \int \frac{2}{u^2} \, W_T^2 \, \emu + \int G_2 \, W_T^2 \,  \Psi^{-2}\,  \emu\\
&- \int  \Big ( -G_1 + 2  \frac{\Psi^{-2}-1}u  \Big ) \, W^2 \vf  \vf_u \, \emu  +   \int  W^2 \,  \vf_u^2  \, \emu.
\end{align*}
We will see in the sequel that our weight $e^{\mu(u,\tau)}$ is chosen so that $| \mu_u -  \frac{\Psi^{-2} -1}u |  $ is small compared to $\mu_u$ in the whole  tip region. 
Moreover the third term  in $G_1$ is small compared to the other two terms. This inspires us to combine  the first three terms on the right hand side of the above formula  to complete a square  
$$ - \int   ( (W_T)_u  + \mu_u W_T)^2 \, \emu = - \int    \big ( e^{\mu} \, W_T )_u^2 \, e^{-\mu}\, du$$
plus the  remaining terms
$$ \int \Big ( \mu_u -  \frac{\Psi^{-2} -1}u - \frac{u}{2\Psi^2} \Big ) \, W_T  (W_{T})_u   \, \emu 
+ \int    \Big ( \mu_u -   \frac{\Psi^{-2} -1}u \Big ) \, \mu_u\, W_T^2 \, \emu.$$ 
Setting 
$$G_0:= \mu_u -   \frac{\Psi^{-2} -1}u = \mu_u - \frac{1-Y}{u Y}  \qquad \mbox{(recall that  $Y=\Psi^2$)} $$
we   finally obtain the  following energy inequality which holds on the whole tip region:
\bee
\begin{split}
\frac 12 \frac {d}{d\tau} \int  \frac{W^2_T}{\Psi^2}   \,  \emu  
= &-  \int    \big ( e^{\mu} \, W_T )_u^2 \, e^{-\mu}\, du +  \int \Big ( G_0  - \frac{u}{2\Psi^2} \Big ) \, W_T  (W_{T})_u   \, \emu \\
& - \int \frac{2}{u^2} \, W_T^2 \, \emu + \int  \Big (  G_0  \, \mu_u + G_2 \Big ) \, W_T^2 \,  \emu\\
&- \int  \Big ( \Big ( -G_1 + 2  \frac{\Psi^{-2}-1}u  \Big )  \vf  \vf_u  +  \vf_u^2 \Big )\,  W^2   \, \emu.
\end{split}
\eee
To absorb  the cross term with $W_T (W_T)_u$,  we set
${\ds G_3:= G_0 - \frac u{2 \Psi^2}}$ and write 
\bee
\begin{split}
 \int G_3 \, W_T  (W_T)_u   \, \emu  &=  \int  G_3 \, W_T  \big ( e^{\mu} \, W_T \big )_u   \, du  -  \int  G_3 \, \mu_u W_T^2   \, \emu \\
&\leq \frac 12  \int    \big (e^{\mu} \, W_T  \big )_u^2    e^{-\mu}  \, du + \int \big (  \frac 12  G_3^2 - G_3 \, \mu_u \big ) \, W_T^2\, \emu.
 \end{split}
\eee
Combining the last two estimates, we finally obtain that 
\be\label{eqn-energy}
\begin{split}
\frac 12 \frac {d}{d\tau} \int  \frac{W^2_T}{\Psi^2}   \,  \emu  
\le &- \frac 12  \int    \big ( e^{\mu} \, W_T )_u^2 \, e^{-\mu}\, du - \int \frac{2}{u^2} \, W_T^2 \, \emu\\
&+ \int  \Big (  \frac 12 \big (  G_0 - \frac u{2 \Psi^2}   \big )^2 + \frac u{2\Psi^2}\, \mu_u + G_2 \Big ) \, W_T^2 \,  \emu\\
&- \int  \Big ( \big (- G_1 + 2  \frac{\Psi^{-2}-1}u  \big )  \vf  \vf_u  +  \vf_u^2 \Big )\,  W^2   \, \emu.
\end{split}
\ee

\sk

Lets now  combine our energy estimate \eqref{eqn-energy} 
with the Poincar\'e inequality \eqref{eq-Poincare} applied to $f:= e^{\mu} \, W_T$. The latter can be written as 
\be\label{eqn-poin2} \int    \big ( e^{\mu} \, W_T )_u^2 \, e^{-\mu}\, du \geq C_0^{-1} \int \mu_u^2 \, W_T^2 \, \emu -  \int \frac{1}{u^2} \, W_T^2 \, \emu. 
\ee
Combining \eqref{eqn-energy} with \eqref{eqn-poin2} yields the differential inequality 
\be\label{eqn-energy2}
\begin{split}
\frac 12 \frac {d}{d\tau} \int  \frac{W^2_T}{\Psi^2}   \,  \emu  
\le &- c_0 \,  \int \mu_u^2\, W_T^2 \, \emu  - \frac 32  \int \frac 1{u^2}  \, W_T^2 \, \emu\\
&+ \int  \Big (  \frac 12 \big (  G_0 - \frac u{2 \Psi^2}   \big )^2 + \frac u{2\Psi^2}\, \mu_u + G_2 \Big ) \, W_T^2 \,  \emu\\
&- \int_\theta^{2\theta}   \Big ( \big ( -G_1 + 2  \frac{\Psi^{-2}-1}u  \big )  \vf  \vf_u  +  \vf_u^2 \Big )\,  W^2   \, \emu
\end{split}
\ee
where $c_0 :=C^{-1}_0/2$ is an absolute constant. 

\sk
\sk 
We will see below that the terms in the first line of the above formula are our main order terms while the terms in the second and third lines
are small. The negative term $- \frac 32 \int \frac{1}{u^2} \, W_T^2 \, \emu$ is a low order term in the collar region and will not help us estimating
the error terms, however near the tip $u=0$ it becomes large and will help us deal with errors in the soliton region. These estimates
will be done next and we will do them separately in the  collar and soliton regions.

\sk
\sk 

\begin{claim}[Estimate of  error terms  in $\collar_{\theta, L}$] \label{claim-collar}   Fix $\eta >0$.  There exists $\theta >0$ small,  $L \gg 1$  and $\tau_0 \ll -1$ (all three depending on $\eta$) and an absolute constant $C_1$, such that 
\be\label{eqn-Gcollar}
 \Big |   \frac 12 \big (  G_0 - \frac u{2 \Psi^2}   \big )^2 + \frac u{2\Psi^2}\, \mu_u + G_2 \Big | \leq C_1 \, \eta \, \mu_u^2 
 \ee 
holds on $\collar_{\theta,L}$ for $\tau \leq \tau_0$.  Furthermore, the bound 
\be\label{eqn-Gcollar2}
\Big | \big ( -G_1 + 2  \frac{\Psi^{-2}-1}u  \big )  \vf  \vf_u  +  \vf_u^2 \Big | \leq  C(\theta) \,  \Psi^{-2} \, \chi_{_{[\theta, 2\theta]}}
\ee
holds on the support of $\vf_u$ and for $\tau \leq \tau_0\ll -1$. 
\end{claim}

\begin{proof}[Proof of Claim \ref{claim-collar}]  First, lets use the bounds in \eqref{eqn-Psib}, \eqref{eqn-B}, \eqref{eqn-mutau}  
and the lower bound $\mu_u^2 \geq |\tau|/(2Y)$ (which follows from \eqref{eqn-good1} and \eqref{eqn-good5})  to  obtain the bound
$$ G_2   \leq C_0\, \frac{\eta}{u^2 \Psi^4} \leq  2C_0 \, \eta \, \mu_u^2 $$
for an absolute constant $C_0$. Moreover, the bound 
$|G_0 | \leq \eta  \, \mu_u$ (which readily follows from \eqref{eqn-good1})  implies that 
$$ \frac 12 \big (  G_0 - \frac u{2 \Psi^2}   \big )^2 + \frac u{2\Psi^2}\, \mu_u \leq \eta \, \mu_u^2$$
for  $u \leq 2\theta$, if  $\theta$ is sufficiently small depending on $\eta$ and $\eta^2 \ll \eta$.  Combining the two estimates yields \eqref{eqn-Gcollar} for a different absolute constant $C_1$.  

The second estimate \eqref{eqn-Gcollar2} easily follows from the definition of $G_1$ and \eqref{eqn-good1}. 

\end{proof}

\sk

We will  now {\em  fix  $\theta >0$ small and  $L \gg 1$}  so that 
\be\label{eqn-collarc0}
 \Big |   \frac 12 \big (  G_0 - \frac u{2 \Psi^2}   \big )^2 + \frac u{2\Psi^2}\, \mu_u + G_2 \Big | \leq C_1 \,  \eta \, \mu_u^2 
 \leq  \frac{c_0}4  \, \mu_u^2
 \ee 
holds on $\collar_{L,\theta}$, for $\tau \leq \tau_0 \ll -1$, where $c_0$ is the absolute constant from our Poincar\'e inequality \eqref{eqn-poin2}. 
For this choice of $L$, we will 
consider the soliton region $S_L$ and we will use the $C^\infty$ convergence of $Z(\rho,\tau), Z_2(\rho,\tau)$ to the Bryant $Z_0(\rho)$ (see Proposition \ref{prop-convergence-tip})
to absorb the error terms in \eqref{eqn-energy} by the good negative terms. We will next show the analogous estimate on $S_L$,
where we notice that $G_0=0$ (since $\zeta \equiv 0$ there). Hence we claim the following:

\begin{claim}[Estimate of  error terms   in $S_L$] \label{claim-soliton} Let $c_0$ be the  constant from \eqref{eqn-poin2}
and assume that $c_0 <1$. 
 For the given choice of  $L$  so that  \eqref{eqn-collarc0} holds on $\collar_{\theta,L}$,
there exists  $\tau_0 \ll -1$ such that 
\be\label{eqn-Gsoliton}
 \frac {u^2}{8 \Psi^4}  + \frac u{2\Psi^2}\, \mu_u + G_2  \leq  \frac{c_0}4 \frac 1{u^2 \Psi^4}
 \ee 
holds on $S_L$, for all $\tau \leq \tau_0$. \end{claim} 

\begin{proof}[Proof of Claim] First, observe that 
${\ds  \frac {u^2}{8 \Psi^4}  + \frac u{2\Psi^2}\, \mu_u  \ll  \frac{c_0}{8}  \frac 1{u^2 \Psi^4}}$
on $S_L$ hence, it is sufficient to show that 
${\ds G_2  \leq  \frac{c_0}{8}  \frac 1{u^2 \Psi^4}}$ for $\tau \leq \tau_0 \ll -1$, where $G_2$ is defined in \eqref{eqn-G12}. 
Transforming to soliton region variables, this is equivalent to showing that 
$$| \cB|  + \frac 12 \mu_\tau +  \frac 12 \big | Z_\tau - \frac {\rho}{2|\tau|} Z_\rho \big |  \leq \frac{c_0}{8}  \frac {|\tau|}{\rho^2 Z}$$
and it is sufficient to show that each of these three terms is bounded by ${\ds  \frac {\eta |\tau|}{\rho^2 Z}}$ for $\tau \leq \tau_0 \ll -1$,
with $\eta $ sufficiently small.  The desired   bound for the second term follows by \eqref{eqn-mutau} 
and for  the third term by   \eqref{eqn-Ytau} by  taking $\tau \leq \tau_0 \ll -1$.  For the first term, using the convergence $Z(\rho,\tau) \to Z_0$,
$Z_2(\rho,\tau) \to Z_0$ and \eqref{eqn-Ytau}, we find that for any given $\eta >0$ we can find $\tau_0 \ll -1$, such that
$$|\cB | \leq \frac {\eta |\tau|}{\rho^2} \leq  \frac {\eta |\tau|}{\rho^2 Z} \qquad (\mbox{since} \,\, Z \leq 1)$$
holds on $S_L$ and for $\tau \leq \tau_0 \ll -1$.  Combining these three bounds for $\eta = c_0/40$, we finally conclude 
\eqref{eqn-Gsoliton}. 

%

\end{proof} 

\sk
\sk 
We will  combine \eqref{eqn-energy2} with \eqref{eqn-Gcollar}, \eqref{eqn-Gcollar2} and \eqref{eqn-Gsoliton}. In fact,  
using all our bounds, it is easy to see that  \eqref{eqn-energy2} implies the following differential inequality
\bee
\begin{split}
\frac {d}{d\tau} \int  \frac{W^2_T}{\Psi^2}   \,  \emu  
\le  - \int \big ( c_0  \, \mu_u^2 + \frac 2{u^2} \big )  W_T^2 \, \emu  + C(\theta) \,  \int \frac{ W^2}{\Psi^2}  \, \chi_{_{[\theta, 2\theta]}}  \, \emu. 
\end{split}
\eee
We may assume that $c_0 <1$. Then, using \eqref{eqn-good1} once more
$$c_0  \, \mu_u^2 + \frac 2{u^2}  \geq c_0  \, \frac{(1-\Psi^2)^2}{2 u^2 \Psi^4} + \frac 2{u^2} \geq 
\frac{c_0}{100}  \,  \frac 1{u^2 \Psi^4}.$$
By \eqref{eqn-good1} and  the soliton  asymptotics we have that 
there exists a constant $\lambda >0$ (depending on $c_0$) such that on the whole tip region 
$$\frac{c_0}{100}  \,  \frac 1{u^2 \Psi^2} = \frac{c_0}{100}  \,  \frac 1{u^2 Y} \geq 2  \lambda \, |\tau|.$$
Hence, we finally conclude the desired  integral differential inequality \eqref{eqn-diff-ineq}. 
\end{proof}

\sk 

\subsection{Proof of Proposition \ref{prop-tip}}\label{subsec-tip-proof}
The proof of  Proposition \ref{prop-tip}  easily follows from the integral differential inequality \eqref{eqn-diff-ineq}.   
Setting   
\[
  f(\tau) := \int W_T^2 \, \Psi^{-2}  \, e^{\mu}\, du, \qquad g(\tau) := \int
W^2  \, \Psi^{-2} \, \chi_{[\theta,2\theta]} \, e^{\mu}\, du
\]
we may express \eqref{eqn-diff-ineq} as 
\[
  \frac{d}{d\tau} f(\tau) \le - 2 \lambda \,  |\tau|\, f(\tau) + C(\theta)
 \, g(\tau).
\]
Furthermore, setting ${\displaystyle F(\tau) := \int_{\tau-1}^{\tau} f(s)\, ds}$
and ${\displaystyle {G}(\tau) := \int_{\tau-1}^{\tau} g(s)\, ds}$, we have
\begin{equation*}
  \begin{split}
    \frac{d}{d\tau}  {F}(\tau) = f(\tau) - f(\tau-1)
    &= \int_{\tau-1}^{\tau} \frac{d}{ds} f(s)\, ds \\
    &\le - 2 \lambda  \,\int_{\tau-1}^{\tau} |s| f(s)\, ds +  C(\theta)    \int_{\tau-1}^{\tau} g(s)\, ds
  \end{split}
\end{equation*}
implying
\[
  \frac{d}{d\tau} {F} \le - 2 \lambda\, |\tau| \, {F} +
 C(\theta)\, {G}.
\]
This is equivalent to
\[
  \frac{d}{d\tau} \bigl( e^{-  \lambda \tau ^2} F(\tau) \bigr) \le
  \frac{C(\theta)}{|\tau|} e^{-\lambda \tau^2}\, G(\tau).
\]
Furthermore, since $W^2 :=  (\Psi -\bpsi)^2  \leq 1$  and $\Psi^{-2} \emu \leq 1$ on $\tip_\theta$, 
for $\tau \leq \tau_0 \ll -1$ and $\theta \ll 1$ (which follows from \eqref{eqn-mub1}), the functions 
 $F(\tau)$ and  $G(\tau)$ are uniformly bounded functions for
$\tau \leq \tau_0$.  Hence, 
${\displaystyle \lim_{\tau\to-\infty} e^{-  \lambda |\tau|^2} F(\tau) = 0}$, so that
from the last differential inequality we get
\begin{equation*}
  \begin{split}
    e^{-\lambda |\tau|^2}\, F(\tau)
    &\le C\, \int_{-\infty}^{\tau} \frac{G(s)}{|s| } (|s|\, e^{-\lambda s^2} )\, ds \\
    &\le \frac{C}{\sqrt{|\tau|}} \, \sup_{s\le\tau} \frac{G(s)}{\sqrt{|s|}} \, \int_{-\infty}^{\tau} |s|\,  e^{-\lambda  s^2} \, ds\\
    &\le \, \frac{C}{\sqrt{|\tau|}}\,   \sup_{s\le\tau} \frac{G(s)}{\sqrt{|s|}}
  \end{split}
\end{equation*}
where $C=C(\theta)$. It follows that  for all $\tau \leq \tau_0 \ll -1$ we have 
\[|\tau|^{-1/2}\, F(\tau) \le  \frac{C}{|\tau|}\, \sup_{s\le\tau} |s|^{-1/2} \, G(s)\]
and hence,
\[
  \sup_{\tau \le \tau_0} |\tau|^{-1/2} F(\tau ) \le \frac{C}{|\tau_0|}\, \sup_{\tau \le\tau_0} |\tau|^{-1/2}\, G(\tau)
\]
or equivalently,
\[
  \|W_T\|_{2,\infty} \le \frac{C(\theta)}{\sqrt{|\tau_0|}} \,  \|W
  \chi_{[\theta,2\theta]}\|_{2,\infty}
\]
therefore concluding the proof of Proposition \ref{prop-tip}.


\section{Proof of Theorem \ref{thm-main}}\label{sec-conclusion}
We will now combine Propositions \ref{prop-cylindrical} and \ref{prop-tip} to
conclude the proof of our main result Theorem \ref{thm-main}.  
Throughout  this section we will fix the  constant $\theta >0$ given in Proposition \ref{prop-tip}. 
For this constant $\theta$,  let us recall  our notation of the various regions 
$$C_\theta = \{ u_1 \geq \theta/4  \}, \qquad  D_\theta = \{ \theta/4 \leq u_1 \leq \theta/2  \},  \qquad \tip_\theta  = \{ u_1 \leq 2\theta \} $$
and that $\varphi_C$, $\varphi_T $ are supported on $C_\theta$, $\tip_\theta$ respectively and $\varphi_C \equiv 1$ on $C_{2\theta}$ and $\varphi_T \equiv 1$ on
$\tip_{\theta/2}$. 

\smallskip We have seen at the beginning of Section \ref{sec-outline}
that translating and dilating the original solution has an effect on the
rescaled rotationally symmetric solution $u(\sigma,\tau)$, as given in
formula \eqref{eq-ualphabeta}. Let $u_1(\sigma,\tau)$ and $u_2(\sigma,\tau)$ be any
two solutions to equation \eqref{eq-u} as in the statement of Theorem
\ref{thm-main} and let $u_2^{\beta \gamma}$ be defined by
\eqref{eq-ualphabeta}.  Our goal is to find parameters $(\beta,\gamma)$
so that the difference
\[
  w^{\beta \gamma}:= u_1-u_2^{\beta \gamma}\equiv 0.
\]

Proposition \ref{prop-tip} says that the weighted $L^2$-norm
$\|W^{\beta \gamma} \|_{2,\infty}$ of the difference of our solutions
$W^{\beta \gamma}(u,\tau):=\Psi_1(u,\tau) - \Psi_2^{\beta \gamma}(u,\tau)$
(after we change the variables) in the whole tip region 
$\tip_\theta := \{(y, \tau) : u_1(y,\tau) \le 2 \theta\}$ is controlled by
$\| W^{\beta \gamma} \, \chi_{_{[\theta,2\theta]}}\|_{2,\infty}$, where
$\chi_{[\theta,2\theta]}(u)$ is supported in the transition region between the
cylindrical and tip regions. By \eqref{eq-other-one} which will be shown below,  
$\| W^{\beta \gamma} \, \chi_{_{[\theta, 2\theta]}}\|_{2,\infty}$ can be estimated in terms of 
$\| w^{\beta \gamma} \, \chi_{_{D_{4\theta}}}\|_{\hilb,\infty}$, where $D_\theta = \{(\sigma, \tau) :  \theta/4 \leq u_1(\sigma,\tau) \leq \theta/2  \}$.
Therefore combining Propositions \ref{prop-cylindrical} and \ref{prop-tip} gives
the crucial estimate \eqref{eqn-w1230} which will be shown in detail in
Proposition \ref{prop-cor-main} below.  This estimate says that the norm of the
difference $w^{\beta \gamma}_{\cC}$ of our solutions when restricted in the
cylindrical region is dominated by the norm of its projection onto the zero eigenspace of the operator $\cL$
(the linearization of our equation on the limiting cylinder).  Note that 
Proposition \ref{prop-cylindrical} holds under the assumption that the
projection of $w_\cyl^{\beta \gamma}$ onto the positive eigenspace of $\cL$
is zero, that is $\pr_+ w_\cyl(\tau_0)^{\beta \gamma} =0$.  

\smallskip After having established that the projection onto the zero eigenspace
$a(\tau):= \langle w_\cyl^{\beta \gamma}, \psi_2 \rangle$ dominates in the
$\| w^{\beta \gamma}_{\cC} \|_{\hilb,\infty}$, the conclusion of Theorem
\ref{thm-main} will follow by establishing an appropriate differential
inequality for $a(\tau)$, for $\tau \leq \tau_0\ll -1$ and also having that
$a(\tau_0) = \mathcal{P}_0 w_C^{\beta \gamma}(\tau_0) = 0$. 
It follows from this  discussion that it is essential to have 
\begin{equation}\label{eqn-abc}
  \pr_+ w_\cyl^{\beta \gamma}(\tau_0) = \pr_0 w_\cyl^{\beta \gamma}(\tau_0) = 0
\end{equation}
holding for the same $\tau_0$ when $\tau_0 \ll  -1$. This will be done by appropriately choosing the parameters $\beta$ and $\gamma$. 
In fact, we will  next show that for every $\tau_0 \ll -1$ we can find parameters
$\beta=\beta(\tau_0)$ and $\gamma=\gamma(\tau_0)$ such
that \eqref{eqn-abc} holds and we will also give their asymptotics relative to
$\tau_0$.  Let us emphasize that we need to be able {\em for every}
$\tau_0 \ll -1$ to find parameters $ \beta, \gamma$ so that
\eqref{eqn-abc} holds, since up to the final step of our proof we have to keep
adjusting $\tau_0$ by taking it even more negative so that our estimates hold
(see Remark \ref{rem-choice-par} below).

\smallskip

We will need the following result whose proof is identical to the analogous result in \cite{ADS2}.

\begin{proposition}\label{lem-rescaling-components-zero}
  There is a number $ \tau_* \ll -1 $ such that for all $\tau\leq \tau_*$ there
  exist $ b $ and  $ \Gamma $ such that the difference
  $w^{\beta \gamma} :=u_1 - u_2^{\beta \gamma}$ satisfies
   \[
    \langle \varphi_\cyl \, w^{\beta \gamma}, \psi_0  \rangle =  0 \qquad \mbox{and} \qquad 
 \langle   \varphi_\cyl \,w^{\beta \gamma}, \psi_2 \rangle = 0
  \]where $\psi_0(\sigma)  \equiv 1$ and $\psi_2(\sigma) = \sigma^2 -2$ are the positive and null eigenvesctors of the
 operator $\cL$. 
   In addition, the parameters $\beta$ and $\gamma$ can be chosen so that
\begin{equation}\label{btheta}
  b := \sqrt{1+\beta e^\tau} -1  \qquad \mbox{and} \qquad \Gamma := \frac{\gamma - \log(1+\beta e^\tau)}{\tau}.
\end{equation}
 satisfy
  \begin{equation}\label{eq-b-Theta-bound}
    |b| = o\bigl( |\tau|^{-1} \bigr) \qquad \mbox{and}   \qquad |\Gamma| = o(1), \qquad \mbox{as}\,\, \tau \to -\infty.
  \end{equation}
  Equivalently, this means that $(\beta, \gamma)$ is
  admissible with respect to $\tau$, according to our Definition
  \ref{def-admissible}.
\end{proposition}

\begin{proof}
The proof relies only on the asymptotics of our solution in the cylindrical region (see Theorem \ref{thm-asym} and its proof in \cite{ADS3}). Since the asymptotics of our Ricci flow rotationally symmetric solution in the cylindrical region are very similar to the cylindrical region asymptotics  of  ancient mean curvature flow solution with rotational symmetry (see \cite{ADS1}) and they differ just by 
a constant, the proof of this Proposition is identical to the proof of corresponding Proposition 7.1 in \cite{ADS2}.
\end{proof}

\begin{remark}[The choice of parameters $(\beta, \gamma)$]\label{rem-choice-par}
  We can choose $\tau_0 \ll -1$ to be any small number so that
  $\tau_0 \le \tau_*$, where $\tau_*$ is as in Proposition
  \ref{lem-rescaling-components-zero} and so that all our uniform estimates in
  previous sections hold for $\tau \le \tau_0$.  Note also that having
  Proposition \ref{lem-rescaling-components-zero} we can decrease $\tau_0$ if
  necessary and choose parameters $\beta$ and $\gamma$ again so that we
  still have $\pr_+ w_\cyl(\tau_0) = \pr_0 w_\cyl(\tau_0) = 0$, without
  effecting our estimates.  Hence, from now on we will be assuming that we have
  fixed parameters $\beta$ and $\gamma$ at some time $\tau_0 \ll -1$, to
  have both projections zero at time $\tau_0$.  As a consequence of Proposition
  \ref{lem-rescaling-components-zero} which shows that the parameters
  $(\beta, \gamma)$ are {\em admissible} with respect to $\tau_0$ and
  Remark \ref{rem-cylindrical}, all the estimates for
  $w = u_1- u_2^{\beta \gamma}$ will then hold for all $\tau \le \tau_0$,
  {\em independently of our choice} of $(\beta, \gamma)$.
\end{remark}

As we pointed out above, we need to show next that the norms of the difference
of our two solutions with respect to the weights defined in the cylindrical and
the tip regions satisfy comparison inequalities in the intersection between the regions, so
called {\em transition} region. We need those inequalities to conclude the proof of Theorem \ref{thm-main}. 
 We have the following. 

\begin{lemma}
\label{lem-weighted-H1-L2}
Given $0 < \theta \ll 1$, there exists $C(\theta) >0$ such that 
\be\label{eqn-help}
|\tau| \int  (w \,   \chi_{_{D_\theta}})^2  \, e^{-\sigma^2/4} \, d\sigma \leq  C(\theta) 
\Big (  \int  \chi_{_{\tip_{\theta/2}}}  w_\sigma^2 \, e^{-\sigma^2/4} \, d\sigma + o(1) \int  w^2_C \, e^{-\sigma^2/4} \, d\sigma \Big ). 
 \ee
\end{lemma}
\begin{proof}
Fix a $\tau \leq \tau_0 \ll -1$ and consider a smooth  function $\eta(\sigma, \tau)$ defined on $\sigma \geq 0$, satisfying   $0 \leq \eta \leq 1$, and 
$$\eta \equiv 1\,\, \mbox{on}\,\,  \{ \sigma:  u_1(\sigma,\tau) \leq \theta/2 \} \qquad \mbox{and} \qquad \eta \equiv 0 \,\,   \mbox{on} \,\, \{ \sigma: u_1(\sigma,\tau) \geq \theta \}.$$
Denote by $\ell_1:=\ell_1(\theta,\tau) >0$, $\ell_2:=\ell_2(\theta,\tau) >0$ the points  at which   $u_1(\ell_1, \tau) = \theta$ and $u_1(\ell_2, \tau) = \theta/4$.  Clearly, $\ell_2 > \ell_1$ by the monotonicity  of 
$u_1(\cdot,\tau)$ on the set $ u_1(\sigma,\tau) \leq \theta$, for $\tau \ll -1$.

Next set,  $v:= \eta \, w$ and use the inequality 
\begin{multline}
\label{eq-weighted-H1-L2}
\int_0^{\ell_2}  v_\sigma^2 \, e^{-\sigma^2/4} \, d\sigma + \frac 14 \, \int_0^{\ell_2}   v^2 \,e^{-\sigma^2/4} \, d\sigma\\
\geq \frac14 \ell_2 e^{-\ell_2^2/4}v(\ell_2)^2 + \frac 1{16} \int_0^{\ell_2} \sigma^2 \, v^2 \, e^{-\sigma^2/4} \, d\sigma 
\end{multline}
where $\ell_2=\ell_2(\theta,\tau)$ is the number defined above.  This standard weighted Poincar\'e type inequality was shown in Lemma 4.12 in \cite{ADS1}. 
Since $v \equiv 0$ on $0 \leq \sigma \leq \ell_1$ (corresponding to $u_1(\sigma,\tau)  \geq  \theta$) we obtain (after dropping the positive boundary term)
the bound 
\bee
 \frac 1{16} \int_{\ell_1}^{\ell_2} \sigma^2 \, v^2 \, e^{-\sigma^2/4} \, d\sigma \leq \int_{\ell_1}^{\ell_2}  v_\sigma^2 \, e^{-\sigma^2/4} \, d\sigma + \frac 14 \, \int_{\ell_1}^{\ell_2}   v^2 \,e^{-\sigma^2/4} \, d\sigma. \\
\eee
Using that $\ell_1 \gg 1000$ for $\tau \ll -1$, we get that 
\bee
 \frac 1{32} \int_{\ell_1}^{\ell_2} \sigma^2 \, v^2 \, e^{-\sigma^2/4} \, d\sigma \leq \int_{\ell_1}^{\ell_2}  v_\sigma^2 \, e^{-\sigma^2/4} \, d\sigma.
\eee
Since, $v: = \eta w$ and $v_\sigma^2 \leq 2 \, \big ( \eta^2 w_\sigma^2 +  \eta_\sigma^2 \, w^2 \big ) $, we conclude the bound
\bee
 \int_{\ell_1}^{\ell_2} \sigma^2 \, \eta^2 w^2  \, e^{-\sigma^2/4} \, d\sigma \leq 64 \int_{\ell_1}^{\ell_2}  \eta^2 w_\sigma^2 \, e^{-\sigma^2/4} \, d\sigma + 
 64 \int_{\ell_1}^{\ell_2}  \eta^2_\sigma  \, w^2 \, e^{-\sigma^2/4} \, d\sigma. 
\eee
However,  by definition $\eta_\sigma \neq  0$ only on the set $C_{2\theta}:=\{ u_1 \geq \theta/2\}$ which is contained in the set were $\varphi_C \equiv 1$, i.e. on the set where $w = w_C$. In addition,
$\eta_\sigma \ll o(1)$, as $\tau \to -\infty$.  Using also that $\eta \equiv 1$ on $D_\theta:= \{ \sigma: \, \theta/4 \leq u_1(\sigma,\tau)  \leq \theta/2\}$ and that $\eta$ is supported on 
$\tip_{\theta/2}:= \{ \sigma:\, u_1(\sigma,\tau) \leq \theta \}$, 
 we finally conclude the bound
\bee
 \int  \sigma^2 (w \,   \chi_{D_\theta})^2  \, e^{-\sigma^2/4} \, d\sigma \leq 64  \int  \chi_{_{\tip_{\theta/2}} }  w_\sigma^2 \, e^{-\sigma^2/4} \, d\sigma + o(1)\,  \int  w^2_C \, e^{-\sigma^2/4} \, d\sigma.
 \eee
Furthermore we have shown that on the region $D_\theta $ we have $\sigma^2 \geq c_\theta |\tau|$. Combining the above readily implies \eqref{eqn-help}. 
\end{proof}

We recall  the definitions of the norms $\| w \|_{\hv\,,\infty}$ and $\| W \|_{2, \infty}$
given in \eqref{eqn-cyl-norm}  and \eqref{eqn-normt00}-\eqref{eqn-normt}  respectively. As a  corollary of the previous Lemma  we have the following relations between our norms.

\begin{corollary}
\label{cor-equiv-norm}
Given $0 < \theta \ll 1$, there exists $C(\theta)$ and $\tau_0 \ll -1$ such that for all $\tau \le \tau_0$ we have,
 \be\label{eqn-help1}
\| w \, \chi_{_{D_{\theta}}}\|_{\hilb,\infty}  \leq  \frac{C(\theta)}{\sqrt{|\tau_0|}}  \,
 \Big (  \| W_T \|_{2,\infty} + o(1) \, \| w_C \|_{\hilb,\infty} \Big )
 \ee
and  also 
 \begin{equation}
 \label{eq-other-one}
\| W \, \chi_{_{[\theta, 2\theta]}}\|_{2,\infty}  \le C(\theta) \,  \| w_\sigma   \chi_{_{D_{4\theta}}}\|_{\hilb,\infty}. 
 \end{equation}
\end{corollary}

\begin{proof}
To prove \eqref{eqn-help1} we recall that  $W^2_T= w_{\sigma}^2$ and $\Psi:= |u_{1\sigma}| \leq  c(\theta) \, |\tau|^{-1/2}$  on $\tip_{\theta/2}$ (since this bound  holds in $D_{2\theta}$ and $u_{i}$ are concave). 
Hence, also   using the change of variables $du = u_{1\sigma}\, d\sigma$, we deduce from \eqref{eqn-help} that for any $\tau \leq \tau_0 \ll -1$ we have 
\bee\begin{split}
 \int  (w \,   \chi_{_{D_\theta}})^2  \, e^{-\sigma^2/4} \, d\sigma &\leq  \frac{C(\theta)}{|\tau|}  
\Big (  \int   W_T^2  \, \Psi^{-1} \,  e^\mu  \, du  + o(1) \int  w^2_C \, e^{-\sigma^2/4} \, d\sigma \Big )\\
&\leq  \frac{C(\theta) }{|\tau|} \, 
\Big (  |\tau|^{-1/2}  \int   W_T^2  \, \Psi^{-2} \,  e^\mu  \, du  + o(1) \,   \int  w^2_C \, e^{-\sigma^2/4} \, d\sigma \Big )
\end{split}
\eee
and  \eqref{eqn-help1} readily follows using the definitions of our norms.

On the other hand, using $du = u_{1\sigma}\, d\sigma$ and the bound   $\Psi  :=  |u_{1\sigma}| \geq  C(\theta) \, |\tau|^{-1/2} $  on 
$D_{4\theta} :=  \{\sigma: \,  \theta \le u_1(\sigma,\tau) \le 2\theta\}$ we also have 
\begin{equation*}
\begin{split}
\int (W \chi_{[\theta,2\theta]})^2 \Psi^{-2} \, e^{\mu}\, du &= \int (w_{\sigma} \chi_{D_{4\theta}})^2\,   \Psi^{-1} \,
e^{-\frac{\sigma^2}{4}} \, d\sigma \\
&\le  C(\theta) \, |\tau|^{\frac12} \, \int (w_{\sigma} \chi_{D_{4\theta}})^2 e^{-\frac{\sigma^2}{4}}\, d\sigma
\end{split}
\end{equation*}
which combined with  he definition of our norms proves \eqref{eq-other-one}. 
\end{proof}

We next combine our  main results in the previous two sections, Propositions
\ref{prop-cylindrical} and  \ref{prop-tip}, with Corollary \ref{cor-equiv-norm} to establish our {\em crucial estimate} which says
that what actually dominates in the norm $\|w_\cyl \|_{\hv,\infty}$ is
$\|\pr_0 w_\cyl\|_{\hv,\infty}$. 

\begin{proposition}
  \label{prop-cor-main}
  For any $\epsilon >0$ there exists a $\tau_0 \ll -1$ so that we have
  \begin{equation}
    \label{eqn-w1230}
    \| \hat w_\cyl \|_{\hv,\infty} \leq \epsilon\, \|\pr_0 w_\cyl\|_{\hv,\infty}.
  \end{equation}
\end{proposition}

\begin{proof}
By Proposition \ref{lem-rescaling-components-zero} we know that for every
$\tau_0 \ll -1$  sufficiently close to negative infinity,  we can choose parameters $(\beta, \gamma)$  which are admissible with respect to $\tau_0$ and such that
$\pr_+ w_\cyl(\tau_0) = \pr_0 w_\cyl(\tau_0) = 0$.  From now on we will always
consider $w(\sigma,\tau) = u_1(\sigma,\tau) - u_2^{\beta \gamma}(\sigma,\tau)$, for these
chosen parameters $\beta$ and $\gamma$.

Let $0 < \theta \ll 1$ be a sufficiently small number so that Proposition \ref{prop-tip} holds, namely  there exists  $\tau_0 \ll -1$
so that
\[
  \|W_T\|_{2,\infty} \leq  \frac{C(\theta)}{\sqrt{|\tau_0|}} \, \|W
  \chi_{[\theta,2\theta]}\|_{2,\infty}.
\]
In addition, we may choose $\tau_0 \ll -1$ so that  Corollary \ref{cor-equiv-norm}  holds, that is 
\[\|W \chi_{[\theta,2\theta]}\|_{2,\infty} \le C(\theta) \, \|w_{\sigma}\, \chi_{D_{4\theta}}\|_{\cH,\infty}.\]
Since $D_{4\theta} :=  \{\sigma:\,  \theta \le u_1(\sigma,\tau) \le 2\theta\}$ is contained in the set where  $\varphi_C = 1$, we have $ \|w_{\sigma}\, \chi_{D_{4\theta}}\|_{\cH,\infty}
\leq  \|w_C \|_{\hv,\infty}$. Hence, the last two estimates yield 
\begin{equation}
  \label{eq-dominates2}
  \|W_T\|_{2,\infty} \leq   \frac{C(\theta)}{\sqrt{|\tau_0|}} \, \|w_C\|_{\cD,\infty}.
\ee

\sk

Now given any $\epsilon  >0$ and the  $\theta >0$ as above, Proposition  \ref{prop-cylindrical} implies  for $\tau_0$ sufficiently negative we have 
\[
  \|\hat{w}_C\|_{\hv,\infty} \leq  \frac{\epsilon}{3} \, \big ( \|w_\cyl\|_{\hv,\infty} +
  \|w \chi_{D_{\theta}\|_{\hilb,\infty}} \big )
\]
for all $\tau \leq  \tau_0$,  where as before
$D_{\theta} = \{\sigma: \, {\theta}/{4} \le u_1(\sigma,\tau) \le \theta/2\}$.
Using  \eqref{eqn-help1} to estimate $ \|w \chi_{D_{\theta}\|_{\hilb,\infty}}$, we get  the bound 
\begin{equation}
  \label{eq-dominates1}
  \begin{split}
    \|\hat{w}_C\|_{\hv,\infty}  \leq  \frac{\epsilon}{2}\, (\|w_\cyl\|_{\hv,\infty} + C(\theta)\,
    \|W_T\|_{2,\infty}). 
  \end{split}
\end{equation}
 Combining \eqref{eq-dominates2} 
with \eqref{eq-dominates1} yields
\[
  \|\hat{w}_C\|_{\hv,\infty} \leq  \frac{\epsilon}{2}\,\Bigl( \|w_\cyl\|_{\hv,\infty} +
  \frac{C(\theta)}{\sqrt{|\tau_0|}}\, \|w_\cyl\|_{\hv,\infty}\Bigr) \leq \epsilon \, 
 \|w_\cyl\|_{\hv,\infty}
\]
by choosing $|\tau_0|$ sufficiently large relative to $C(\theta)$.  The last estimate readily yields \eqref{eqn-w1230} finishing the proof of the Proposition.

\end{proof}

\begin{proof}[Proof of the Main Theorem \ref{thm-main}]
Fix the  constant $\theta >0$ given in Proposition \ref{prop-tip} and for this constant consider the cylindrical and tip regions 
as defined before.  Recall that
$w^{\beta \gamma}(\sigma,\tau) = u_1(\sigma,\tau) - u_2^{\beta \gamma}(\sigma,\tau)$,
which we shortly denote by $w(\sigma,\tau)= u_1(\sigma,\tau) - u_2(\sigma,\tau)$, where
$u_2^{\beta \gamma}(\sigma,\tau)$ is given by \eqref{eq-ualphabeta}.

\smallskip 

Following the notation from previous sections we have
\[
  \frac{\partial}{\partial \tau} w_\cyl = \mathcal{L}[w_\cyl] + \cE[w_\cyl] +
  \bar{\cE}[w,\varphi_\cyl] + \cE_{nl},
\]
with $w_\cyl = \hat{w}_\cyl + a(\tau) \, \psi_2$, where
$a(\tau) = \langle w_\cyl, \psi_2\rangle$.  The error terms $\cE[w_{\cyl}]$, $\bar{\cE}[w,\varphi_\cyl]$ and $\cE_{nl}$ are given by formulas \eqref{eq-E10}, \eqref{eq-bar-E} and \eqref{eq-nonlocal-error}, respectively. Projecting the above equation on the null 
eigenspace generated by $\psi_2$ while using that
$\langle \mathcal{L}[w_\cyl] , \psi_2 \rangle =0$ we obtain
\[
  \frac{d}{d\tau}a(\tau) = \langle \cE[w_\cyl] + \bar{\cE}[w,\varphi_\cyl] + \cE_{nl} ,
  \psi_2 \rangle.
\]
Since ${\displaystyle \frac{\langle \psi_2^2,\psi_2\rangle}{\| \psi_2\|^2} = 8}$
we can write the above equation as
\[
  \frac{d}{d\tau}a(\tau) = 2\, \frac{a(\tau)}{|\tau|} + F(\tau)
\]
where
\begin{equation}
  \label{eq-F-tau}
    F(\tau) := \frac{ \langle \big (\cE(w_\cyl) - \frac{a(\tau)}{8|\tau|} \psi_2^2 ) + \bar{\cE}[w,\varphi_\cyl] +(\cE_{nl} - \frac{a(\tau)}{8|\tau|}\, \psi_2^2 \big ) , \psi_2\rangle}{\|\psi_2\|^2} 
\end{equation}
Solving the above ordinary differential equation for $a(\tau)$ in terms of $F(\tau)$ 
yields
\[
  a(\tau) = C_0\, \tau^{-2} - \tau^{-2} \, \int_{\tau}^{\tau_0} F(s) \, s^2\,
    ds.
\]
By Remark \ref{rem-choice-par} we may assume $\alpha(\tau_0) = 0$. This means that $C_0=0$ 
which implies
\begin{equation}
  \label{eq-alpha-CC}
  |a(\tau)| = |\tau|^{-2}\, \big | \int_{\tau}^{\tau_0} F(s) \, s^2\, ds \big |. 
\end{equation}
Define
$$\|a\|_{\hilb,\infty}(\tau) = \sup_{\tau' \le \tau} \Bigl(\int_{\tau'-1}^{\tau'}
|a(\zeta)|^2\, d \zeta \Bigr)^{\frac12} \qquad \mbox{and} \qquad \|a\|_{\hilb,\infty} :=\|a\|_{\hilb,\infty}(\tau_0).$$  Since
$\pr_0 w_\cyl(\cdot,\tau) = a(\tau)\, \psi_2(\cdot)$, we have
\[
  \|\pr_0 w_\cyl\|_{\hv,\infty}(\tau) =\|a\|_{\hilb,\infty}(\tau) \,
  \|\psi_2\|_{\hv}.
\]
 Next we
need the following claim.

\begin{claim}
  \label{claim-Fs}
  For every $\epsilon > 0$ there exists a $\tau_0$ depending only on  $\epsilon$,  $\theta$ and dimension $n$  so that
  \[\Bigl|\int_{\tau-1}^{\tau} F(s)\, ds\Bigr| \le
    \frac{\epsilon}{|\tau|}\,\|a\|_{\hilb,\infty}
  \]
  for $\tau \le \tau_0$.
\end{claim}

Assume for the moment that the Claim holds. We will finish the proof of the theorem. Set $\epsilon := 1/2$ and choose a sufficiently negative number $\tau_0 < \min (\tau^*, -100)$  (where $ \tau^*$ 
is as in Proposition  \ref{lem-rescaling-components-zero} ) and such that  Claim \ref{claim-Fs} holds. Such a number $\tau_0$ depends only on the constant  $\theta$ and on dimension $n$.
Proposition \ref{lem-rescaling-components-zero} tells us that for the chosen $\tau_0$ 
we can choose parameters $\beta$ and
$\gamma$  and such that
$\pr_+ w_\cyl(\tau_0) = \pr_0 w_\cyl(\tau_0) = 0$.  We will next see that  for
that choice of parameters $w(\sigma,\tau) \equiv 0$. To this end, observe that for $\tau \leq \tau_0$ we have 
\begin{equation*}
  \begin{split}
    \Bigl|\int_{\tau}^{\tau_0} F(s) \, s^2\, ds\Bigr| &\le \sum_{j=[\tau]-1}^{[\tau_0]} \int_j^{j+1} s^2 F(s)\, ds \le \epsilon\, \|a\|_{\hilb,\infty}\,\sum_{j=[\tau]-1}^{\tau_0} |j|\,   \\
    &\le \epsilon  \|\alpha\|_{\hilb,\infty} \, \sum_{j=[\tau]-1}^{\tau_0} | j |   \\
    &\le \frac 12 \, |\tau|^2\,\|a\|_{\hilb,\infty}.
  \end{split}
\end{equation*}
Combining the last inequality    with \eqref{eq-alpha-CC} while choosing $\epsilon = 1/2$ yields
\[
  |a(\tau)| \le \frac12\|a\|_{\hilb,\infty}, \qquad \mbox{for all} \,\,\,
  \tau\le \tau_0.
\]
This implies
\[
  \|a\|_{\hilb,\infty} \le \frac12 \, \|a\|_{2,\infty}
\]
and hence $\|a\|_{\hilb,\infty} = 0$, which further gives
\[
  \|\pr_0 w_\cyl\|_{\hv,\infty} = 0.
\]
Finally, \eqref{eqn-w1230} implies $\hat{w}_C \equiv 0$ and hence,
$w_\cyl \equiv 0$ for $\tau \le \tau_0$.  By \eqref{eq-other-one} and the fact
that $\varphi_\cyl \equiv 1$ on $D_{4\theta}$ we have
$W\chi_{[\theta,2\theta]} \equiv 0$ for $\tau \le \tau_0$.  Proposition
\ref{prop-tip} then yields that $W_T \equiv 0$ for $\tau\le \tau_0$.  All these
imply $u_1(y,\tau) \equiv u_2^{\beta\gamma}(y,\tau)$, for
$\tau\le \tau_0$.  By forward uniqueness of solutions to the mean curvature flow
(or equivalently to cylindrical equation \eqref{eq-u}), we have
$u_1 \equiv u_2^{\beta\gamma}$, and hence
$ g_1 \equiv g_2^{\alpha\beta\gamma}.$ This concludes the proof of the main
Theorem \ref{thm-main}.

\smallskip To complete the proof of Theorem \ref{thm-main} we still need to
prove Claim \ref{claim-Fs}, and  we do it next. 

\begin{proof}[Proof of Claim \ref{claim-Fs}]
Throughout the proof we will use the estimates 
\begin{equation}
  \label{eq-par-int}
  \|\hat{w}_C\|_{\cD,\infty} \leq  \epsilon \, \|a\|_{\cH,\infty} \qquad \mbox{and} \qquad \|w_\cyl\|_{\hv,\infty} \leq C \, \| a\|_{\hilb,\infty}
\end{equation}
which hold  for all $\tau_0 \ll -1$. These estimates  follow  from Proposition \ref{prop-cor-main} and we can achieve the first estimate to hold for any $\epsilon > 0$, 
for $\tau \le \tau_0$, by choosing $\tau_0 = \tau_0(\epsilon) \ll -1$ sufficiently small.  By  \eqref{eqn-help1} and \eqref{eq-dominates2}, we also have
\begin{equation}
\label{eq-transition-est}
  \|w \, \chi_{D_{\theta}}\|_{\hilb,\infty} \leq 
  \frac{C(\theta)}{\sqrt{|\tau_0|}}\, \|w_\cyl\|_{\hilb,\infty}, \qquad
  \mbox{for} \,\, \tau_0 \ll -1.
\end{equation}
\smallskip
On the other hand, similarly to the proof of Corollary \ref{cor-equiv-norm}, we have  
$$\|w_\sigma \chi_{D_\theta}\|_{\hilb,\infty} \le C \sqrt{|\tau_0| }  \, \|W_T \|_{2,\infty}$$
which combined with \eqref{eq-dominates2} implies that 
$$
\| w_\sigma \chi_{D_\theta}\|_{\hilb,\infty} \le \epsilon \,\|w_C\|_{\cD,\infty}
$$
for $\tau_0 \ll -1$. 
The last bound and \eqref{eq-transition-est} give us the bound 
\begin{equation}
\label{eq-transition-der-est}
\|w\chi_{D_\theta}\|_{\cD,\infty}  \le \epsilon \,\|w_C\|_{\cD,\infty}.
\end{equation}

\smallskip

From the definition of $\bar{\cE}[w,\varphi_\cyl]$ given in \eqref{eq-bar-E} and
the definition of cut off function $\varphi_\cyl$, we see that the support of
$\bar \cE[w,\varphi_\cyl]$ is contained in
\[
  \Bigl(\sqrt{2 - \theta^2} - \epsilon_1\Bigr)\, \sqrt{2|\tau|} \le
  |\sigma| \le \Bigl(\sqrt{2 - \frac{\theta^2}{4}} + \epsilon_1\Bigr)\,
  \sqrt{2|\tau|}
\]
where $\epsilon_1$ is so tiny that
$\sqrt{2 - \frac{\theta^2}{4}} + \epsilon_1 < \sqrt{2}$.  Also by the
\emph{a priori} estimates in Appendix \ref{sec-appendix} we have \begin{equation}
  \label{eq-der-bounds10}
  |u_{\sigma}| + |u_{\sigma\sigma}| \le \frac{C(\theta)}{\sqrt{|\tau|}}, \qquad \mbox{for} \,\, |\sigma| \leq \big (\sqrt{2 - \frac{\theta^2}{4}} + \epsilon_1\big ) \sqrt{2|\tau|}.
\end{equation}
Furthermore, in \cite{ADS3} we have showed that our ancient solutions $u_i$,
$i\in \{1,2\}$ satisfy
\begin{equation}
  \label{eq-L2-asymp0}
  \begin{split}
    \sup_{\tilde{\tau} \le \tau}  \,\,\Bigl \|u_i - \sqrt{2} + \frac{\sqrt{2}}{8|\tau|}\, \psi_2 \Bigr\|&= o(|\tau|^{-1}), \\
  \sup_{\tilde{\tau}\le\tau}\,\, \Big\|\Bigl(u_i + \frac{\sqrt{2}}{8|\tau|}\, \psi_2\Bigr)_{\sigma}\Bigr\| &=
    o(|\tau|^{-1}) \\
    \sup_{\tilde{\tau} \le \tau}\Big\|\Bigl(u_i + \frac{\sqrt{2}}{8|\tau|}\, \psi_2\Bigr)_{\sigma\sigma}\Bigr\|  &= o(|\tau|^{-1}).
  \end{split}
\end{equation}
\smallskip
In particular, this implies
\begin{equation}
  \label{eq-L2-asymp1}
  \sup_{\tau' \le \tau} \big \|u_i - \sqrt{2} \big \|  +  \sup_{\tau' \le \tau} \big \|u_{i \sigma}\big \| + \sup_{\tau' \le \tau} \big \|u_{i\sigma\sigma}\big \| = O(|\tau|^{-1}).
  \end{equation}
 Moreover, by standard regularity parabolic estimates, for every $L > 0$, there exists a $\tau_0 \ll -1$ so that for $\tau\le\tau_0$,
  \begin{equation}
  \label{eq-L2-asymp2}
 \sup_{|\sigma| \le L}\Big(| u_{i\sigma}\Big| + \Big|u_i - \sqrt{2}|\Big) = O(|\tau|^{-1}). 
\end{equation}

\sk
\sk

We will now apply the estimates above to achieve the desired bound. We need to estimate all   projections of  error terms
on the right hand side of \eqref{eq-F-tau}, and we  will treat each of the three terms separately in three different steps. 
We   start with the simplest bound,  which is the projection of the error $\bar{\cE}[w,\varphi_\cyl]$ due to the cut-off function.

\begin{step} For every $\epsilon > 0$ there exists a $\tau_0 \ll -1$ so that for $\tau \le \tau_0$ we have
\begin{equation}
\label{eq-proj-error-cut-off}
 \|\psi_2\|^{-2}\, \Bigl|\int_{\tau-1}^{\tau} \langle \bar{\cE}[w,\varphi_\cyl], \psi_2\rangle\,
  ds \Bigr| \leq  \epsilon \, \frac{\|a\|_{\hilb,\infty}}{|\tau|}.
\end{equation}
\end{step} 
\sk 

We have 
\begin{equation*}
  \label{eq-barE-est}
  \begin{split}
    \int_{\tau-1}^{\tau} |\langle \bar{\cE}[w,\varphi_\cyl], \psi_2\rangle|\, d\tau' &\le
    \int_{\tau-1}^{\tau} \|\bar{\cE}[w,\varphi_\cyl]\|_{\hv^*} \|\psi_2 \, \bar {\chi} \|_{\hv}\, d\tau' \\&\leq 
     e^{-|\tau|/4}\, \|\bar{\cE}[w,\varphi_\cyl]\|_{\hv^*,\infty}
  \end{split}
\end{equation*}
where $\bar {\chi} $ denotes a smooth function with a support in
$|\sigma| \ge (\sqrt{2-\theta^2/4} - 2\epsilon_1)\, \sqrt{2|\tau|}$, being
equal to one for
$|\sigma| \ge (\sqrt{2-\theta^2/4} - \epsilon_1)\, \sqrt{2|\tau|}$. Combining the last estimate  with  Proposition \ref{prop-error}, Proposition \ref{prop-cor-main} and \eqref{eq-transition-est}, also using  \eqref{eq-par-int}, implies that for every $\epsilon > 0$ we can find a $\tau_0 \ll -1$ so that for
$\tau \le \tau_0$ the desired bound \eqref{eq-proj-error-cut-off} holds. 
\qed
\sk

We will next estimate the main error term.  

\begin{step} For every $\epsilon > 0$ there exists a $\tau_0 \ll -1$ so that for $\tau \le \tau_0$ we have
\begin{equation}
\label{eq-proj-error-main}
\|\psi_2\|^{-2}\, \Big|\int_{\tau-1}^{\tau} \langle \cE(w_C) - \frac{a(\tau)}{8|\tau|}\, \psi_2^2, \psi_2\rangle\, d\tau'\Big| \leq  \frac{\epsilon}{|\tau|}\, \|a\|_{\cH,\infty}.
\end{equation}
\end{step}

We have 
\begin{equation*}
  \begin{split}
  \cE(w_\cyl) - \frac{a(\tau)}{8|\tau|}\, \psi_2^2 &= \Big(\frac{w_{\sigma}}{u_1} + \frac{2u_{2\sigma}}{u_1} - J_1\Big)\, (w_C)_\sigma - \Big(\frac{w}{2u_1} + \frac{u_{2\sigma}^2}{u_1 u_2}\Big)\,w_C \\
  & \qquad - \Big(\frac{u_2^2 - 2}{2u_1 u_2}\, w_C + \frac{a(\tau)}{8|\tau|}\, \psi^2\Big),
  \end{split}
  \end{equation*}
where (since $u_{1\sigma} =0$)   we have 
\[J_1 := 2\int_0^{\sigma} \frac{u_{1\sigma\sigma}}{u_1}\, d\sigma' = 2\int_0^{\sigma} \frac{u_{1\sigma}^2}{u_1^2}\, d\sigma' + 2\frac{u_{1\sigma}}{u_1}(\sigma,\tau).\]   
These imply
\begin{equation}\label{eq-E-recall}
\begin{split}
\cE(w_C) - \frac{a(\tau)}{8|\tau|}\, \psi^2 = \Big(&-\frac{w_{\sigma}}{u_1} + 2\int_0^{\sigma} \frac{u_{1\sigma}^2}{u_1^2}\, d\sigma' \Big ) \, (w_C)_\sigma  \\&- \Big(\frac{w}{2u_1} + \frac{u_{2\sigma}^2}{u_1u_2} \Big)\, w_C -  \Big(\frac{u_2^2 - 2}{2u_1 u_2}\, w_C + \frac{a(\tau)}{8|\tau|}\, \psi^2\Big). 
\end{split}
\end{equation}

\sk

We  use the same symbol $\epsilon$ to denote possibly different, but uniformly small constants from line to line. 
Also, we  will often use the bound $|\psi_2 |  \leq 2(\sigma^2 +1)$ and the fact  that the operator 
$f \to \sigma f$ is bounded from $\cD$ to $\cH$. 
We start with the last
term in \eqref{eq-E-recall}. Using the decomposition $w_C =\hat w_C + a(\tau) \, \psi_2$ we  write this  term as 
\begin{equation}
\label{eq-rewrite}
\begin{split}
&\int \Big(\frac{u_2^2 -2}{2u_1u_2}\,  w_C + \frac{a(\tau)}{8|\tau|}\, \psi_2\Big)\psi_2\, d\mu = \int\frac{u_2^2-2}{2u_1u_2} \, \hat{w}_C \, \psi_2\, d\mu \\
&+ a(\tau)  \int \frac{(u_2 - \sqrt{2})}{2u_1u_2}\, \Big((u_2 - \sqrt{2})(1- \sqrt{2}\, u_1) + 2(\sqrt{2} - u_1)\Big)\, \psi_2^2\, d\mu \\
&+ \frac{a(\tau)}{\sqrt{2}} \int \Big( u_2 - \sqrt{2}  + \frac{\sqrt{2}}{8|\tau|} \Big)\, \psi_2^2\, d\mu.
\end{split}
\end{equation}
Since the operator $f \to \sigma f$ is bounded from $\cD$ to $\cH$ we have
\begin{equation*}
\begin{split}
\Big|\int \frac{u_2^2-2}{2u_1u_2} \,  \hat{w}_C  \psi_2\, d\mu\Big| &\le C(\theta)\, \Big(\int(u_2 - \sqrt{2})^2 (\sigma^2 + 1)\, d\mu\Big)^{\frac12}\, \Big(\int \hat{w}_C^2 (\sigma^2 + 1)\, d\mu\Big)^{\frac12} \\
&\le C(\theta)\, (\|u_2 - \sqrt{2}\| + \|u_{2\sigma}\|)\, \|\hat{w}_C\|_{\cD}. 
\end{split}
\end{equation*}
Hence, Proposition \ref{prop-cor-main} and \eqref{eq-L2-asymp1} imply that for $\tau \le \tau_0 \ll -1$ we have
\begin{equation*}
\label{eq-proj-error20}
\|\psi_2\|^{-2}\,\Big|\int_{\tau-1}^{\tau} \int \frac{u_2^2-2}{2u_1u_2} \hat{w}_C \psi_2\, d\mu\, d\tau'\Big| \leq  \frac{\epsilon}{|\tau|}\, \|a\|_{\cH,\infty}.
\end{equation*}
To estimate the second term on the right hand side in \eqref{eq-rewrite}, it is enough to estimate
\begin{equation*}
\label{eq-proj-error21}
\begin{split}
\int_{\tau-1}^{\tau} & |a(\tau') |\int (u_2 - \sqrt{2})^2 \psi_2^2\, d\mu\, d\tau' \\
&\le \|a\|_{\cH,\infty} \Big(\int_{\tau-1}^{\tau}\int (u_2 - \sqrt{2})^2\, d\mu\, \int (u_2 - \sqrt{2})^2\psi_2^4\, d\mu \, d\tau'\Big)^{\frac12} \\
&\le \frac{C}{|\tau|}\, \|a\|_{\cH,\infty} \sup_{\tau\le \tau_0} \Big(\int(u_2 - \sqrt{2})^2\psi_2^4\, d\mu\Big)^{\frac12} \\
&\le \frac{C}{|\tau|}\, \|a\|_{\cH,\infty} \sup_{\tau\le \tau_0} \Big(\int_{|\sigma| \le L} (u_2 - \sqrt{2})^2\psi_2^4\, d\mu + \int_{\{|\sigma| \ge L\}\cap \cC_{\theta}} (u_2 - \sqrt{2})^2\psi_2^4\, d\mu \Big)^{\frac12} \\
&\leq  \frac{\epsilon}{|\tau|}\, \|a\|_{\cH,\infty},
\end{split}
\end{equation*}
where in the second inequality in the previous estimate we have used \eqref{eq-L2-asymp1} and in the last inequality, for $|\sigma| \le L$ we have used \eqref{eq-L2-asymp2} and for $|\sigma| \ge L$, we have found $L$ sufficiently big so that we can make 
$\int_{|\sigma| \ge L} (u  - \sqrt{2})^2\psi_2^4\, d\mu$ as small as we want. For the third term by  the Cauchy-Schwarz inequality and \eqref{eq-L2-asymp0} we  have
\begin{equation}
\label{eq-proj-error22}
\Big|\int_{\tau-1}^{\tau} a(\tau') \int (u_2 - \sqrt{2} + \frac{\sqrt{2}}{8|\tau|}\, \psi_2)\, \psi_2^2\, d\mu \, d\tau'\Big| \leq  \frac{\epsilon}{|\tau|}\, \|a\|_{\cH,\infty}.
\end{equation}
Inserting  these three last estimates into  \eqref{eq-rewrite} we conclude the following bound for the last term in \eqref{eq-E-recall}
\begin{equation}
\label{eq-proj-error2}
\Big|\int_{\tau-1}^{\tau}\int \Big(\frac{u_2^2 -2}{2u_1u_2} w_C + \frac{a(\tau)}{8|\tau|}\, \psi^2_2\Big)\psi_2\, d\mu\, d\tau'\Big| \leq  \frac{\epsilon}{|\tau|}\, \|a\|_{\cH,\infty}.
\end{equation}

\sk
\sk

Next, lets bound the middle term   in \eqref{eq-E-recall}. First, we have 
\[\Big|\int \frac{u_{2\sigma}^2}{u_1u_2} w_C \psi_2\, d\sigma\Big| \le C(\theta)\int_{|\sigma| \le L}u_{2\sigma}^2  |w_C| |\psi_2|\, d\sigma + C(\theta) \int_{|\sigma| \ge L} u_{2\sigma}^2|w_C| |\psi_2|\, d\sigma.\]
For any $\epsilon > 0$, choose $L$ big enough so that $\int_{|\sigma|\ge L} |\psi_2|\, d\mu \leq  \epsilon$. For that $L$ we can choose a $\tau_0 \ll -1$ so that \eqref{eq-L2-asymp2} holds for $\tau\le \tau_0$. By H\"older's inequality and by \eqref{eq-der-bounds10} (where the latter is used in the second  integral above) we get
\[\Big|\int \frac{u_{2\sigma}^2}{u_1u_2} w_C \psi_2\, d\mu\Big| \le C(\theta) \|w_C\|_{\cH}  \sup_{|\sigma| \le L} u_{2\sigma}^2+ \frac{\epsilon}{|\tau|}\,\|w_C\|_{\cH}.\]
If we choose $\tau_0 \ll -1$ sufficiently small, the last inequality,  \eqref{eq-par-int} and \eqref{eq-L2-asymp2} yield
\begin{equation}
\label{eq-proj-error1}
\|\psi_2\|^{-2}\,\Big|\int_{\tau-1}^{\tau} \int \frac{u_{2\sigma}^2}{u_1u_2} w_C \psi_2\, d\mu \, d\tau'\Big| \leq  \frac{\epsilon}{|\tau|}\, \|a\|_{\cH,\infty}.
\end{equation}
Furthermore, using that $f \to \sigma f$ is a bounded operator from $\cD$ to $\cH$, we can bound 
\begin{equation*}
\begin{split}
\Big|\int \frac{w w_C}{2u_1}\, \psi_2\, d\mu\Big| &\le C(\theta)\, \int w_C^2 \, |\psi_2| + C(\theta)\,\int |w_C| |w|\,  \chi_{D_{\theta}}\, |\psi_2|\, d\mu \\
&\le C(\theta) \|w_C\|^2_{\cD} + C(\theta)\, \|w_C\|_{\cH} \Big(\int w^2\,  \psi_2^2\,  \chi_{D_{\theta}}\, d\mu\Big)^{\frac12}. 
\end{split}
\end{equation*}
This together with \eqref{eq-L2-asymp1} yield
\begin{equation}
\label{eq-proj-error3}
\begin{split}
\|\psi_2\|^{-2}\, &\Big|\int_{\tau-1}^{\tau} \int \frac{w w_C}{2u_1}\, \psi_2\, d\mu\,d\tau'\Big| \le \\
&\le C(\theta) \|a\|_{\cH,\infty} \, \sum_{i=1}^2  \Big ( \int_{\tau-1}^{\tau}\int_{\supp\varphi_C}\big(u_i - \sqrt{2} + \frac{\sqrt{2}\psi_2}{8|\tau|}\big)^2\, d\mu\, d\tau'\Big)^{\frac12} \\
&+ C(\theta)\|a\|_{\cH,\infty}\, \sum_{i=1}^2  \Big( \int_{\tau-1}^{\tau} \int_{\supp_{\varphi_C}}\big(u_i - \sqrt{2} + \frac{\sqrt{2} \psi_2}{8|\tau|}\big)_{\sigma}^2\, d\mu\, d\tau'\Big)^{\frac12} \\
&+ C(\theta)\, \|a\|_{\cH,\infty} e^{-c_0|\tau|} \\
& \leq  \frac{\epsilon}{|\tau|}\, \|a\|_{\cH,\infty},
\end{split}
\end{equation}
for $\tau \le \tau_0 \ll -1$ sufficiently small. These last two bounds yield to the estimate for 
the middle term. 

\sk

To deal with the first term, we first see that similarly to the last estimate above we have
\begin{equation*}
\begin{split}
\Big|\int \frac{w_{\sigma} (w_C)_\sigma }{u_1}\, \psi_2\, d\mu\Big | &\le C(\theta) \int (w_C)_\sigma^2\, |\psi_2|\, d\mu + C(\theta)\int |w_{\sigma}||(w_C)_\sigma||\psi_2|\chi_{D_{\theta}}\, d\mu \\
&\le C(\theta) \|w_{C}\|_{\cD}^2 + C(\theta)\,\int (w_C)_{\sigma\sigma}^2\, d\mu + C(\theta) \|w_C\|_{\cD} e^{-c_0|\tau|}. 
\end{split}
\end{equation*}
This implies
\begin{equation}
\label{eq-term-123}
\Big|\int_{\tau-1}^{\tau} \int\frac{w_{\sigma} (w_C)_\sigma}{u_1}\, \psi_2\, d\mu\, d\tau'\Big| \leq  \frac{\epsilon}{|\tau|}\, \|a\|_{\cH,\infty} + \int_{\tau-1}^{\tau}\int (w_C)_\sigma\sigma^2\, d\mu\, d\tau'.
\end{equation}
By the proof of Lemma \ref{lemma-energy} we have for $\tau \le \tau_0$,
\begin{equation}
\label{eq-sec-der-energy}
\int_{\tau-1}^{\tau}\int (w_C)_{\sigma\sigma}^2\, d\mu\, d\tau' \le C\int_{\tau-2}^{\tau}\int w_C^2\, d\mu\, d\tau' + \int_{\tau-2}^{\tau}\int \cE^2\, d\mu \, d\tau',
\end{equation}
where $\cE = \cE(w_C) + \cE[w,\varphi_C] + \cE_{nl}$ is the error term  given by \eqref{eq-E10}, \eqref{eq-bar-E} and \eqref{eq-nonlocal-error}. Similarly as in Proposition \ref{prop-error} where we have estimated $\|\cE\|_{\cD^*,\infty}$, we could estimate $\|\cE\|_{\cH,\infty}(\tau)$ for $\tau \le \tau_0$. By carefully analyzing all the terms we estimated in the proof of Proposition \ref{prop-error} we conclude that for $\tau \le \tau_0 \ll -1$,
\begin{equation*}
\begin{split}
\sup_{\tilde{\tau}\le\tau} \int_{\tit-1}^{\tit} \int \cE^2\, d\mu \,d\tau' &\leq  \epsilon \, \|w_C\|^2_{\cD,\infty}(\tau) + \epsilon \, \|w \chi_{D_{\theta}}\|^2_{\cH,\infty}(\tau) \\
&+ \epsilon \, \|w_{\sigma} \chi_{D_{\theta}}\|^2_{\cH,\infty}(\tau) + \Big\|\frac{(u_2^2-2) w_C}{2u_1u_2}\Big\|^2_{\cH,\infty}(\tau).
\end{split}
\end{equation*}
Combining this and \eqref{eq-transition-der-est} yields
\begin{equation}
\label{eq-we-need1}
\sup_{\tilde{\tau}\le\tau} \int_{\tit-1}^{\tit} \int \cE^2\, d\mu \,d\tau' \leq  C(\theta)\, \epsilon \,  \|w_C\|^2_{\cD,\infty}(\tau) 
\leq C(\theta) \,  \frac{\epsilon}{|\tau|}\, \|a\|_{\cH,\infty},
\end{equation}
where in the last inequality we used similar arguments that we used to obtain \eqref{eq-proj-error3}. This estimate, \eqref{eq-sec-der-energy} and similar arguments to the ones we used to derive \eqref{eq-proj-error3} yield
\begin{equation}
\label{eq-we-need2}
\int_{\tau-1}^{\tau}\int (w_C)_{\sigma\sigma}^2\, d\mu\, d\tau' \le C \, \|w_C\|_{\cD,\infty}^2(\tau) \leq   \frac{\epsilon}{|\tau|}\, \|a\|_{\cH,\infty}.
\end{equation} 
Finally, \eqref{eq-term-123} and \eqref{eq-we-need2} imply the bound 
\begin{equation}
\label{eq-proj-error4}
\|\psi_2\|^{-2}\, \Big|\int_{\tau-1}^{\tau} \int \frac{w_{\sigma} (w_C)_\sigma}{u_1}\, \psi_2\, d\mu\, d\tau' \Big| \leq \frac{\epsilon}{|\tau|}\, \|a\|_{\cH,\infty}
\end{equation}
for a different but still arbitrarily and uniformly  small constant $\epsilon$. 
It only remains  to bound 
\begin{equation}
\begin{split}
\int (w_C)_\sigma \,\big (  \int_0^{\sigma}\frac{u_{1\sigma}^2}{u_1^2}\, d\sigma'  \big ) \psi_2\, d\mu =
2 \int_{A_+} (w_C)_\sigma \big (  \int_0^{\sigma}\frac{u_{1\sigma}^2}{u_1^2}\, d\sigma' \big ) \, \psi_2\, d\mu
\end{split}
\end{equation}
where $A_+ := \supp\varphi_C\cap \{\sigma \ge 0\}$. 
Applying Cauchy-Schwarz inequality and Fubini's theorem imply
\begin{equation*}
\begin{split}
\Big|\int_{A_+} (w_C)_\sigma \, &\big (  \int_0^{\sigma}\frac{u_{1\sigma}^2}{u_1^2}\, d\sigma' \big ) \, \psi_2\, d\mu\Big|
 \le C(\theta)\, \int |(w_C)_\sigma \, \psi_2| \big ( \int_0^{\sigma} u_{1\sigma}^2\, d\sigma' \big ) \, d\mu \\
&= C(\theta)\, \int_0^{u_1^{-1}(\theta,\tau)} u_{1\sigma}^2 \int_{\sigma'}^{u_1^{-1}(\theta,\tau)}|\psi_2| |(w_C)_\sigma| 
\, e^{-\frac{\sigma^2}{4}}\, d\sigma  d\sigma' \\
&\le C(\theta)\|w_{C}\|_{\cD}\, \int_0^{u_1^{-1}(\theta,\tau)} u_{1\sigma}^2\, \Big(\int_{\sigma'}^{\infty} \psi_2^2 \, e^{-\frac{\sigma^2}{4}}\, d\sigma\Big)^{\frac12}\, d\sigma' \\
&\le C(\theta)\, \|w_C\|_{\cD}\, \int_0^{u_1^{-1}(\theta,\tau)} u_{1\sigma}^2 \, e^{-\frac{\sigma'^2}{16}}\, d\sigma' \\
&\le C(\theta)\, \|w_C\|_{\cD}\, \Big(\int_{|\sigma| \le L}u_{1\sigma}^2 \, e^{-\frac{\sigma^2}{16}}\, d\sigma +  \int_{\{|\sigma| \ge L\}\cap\cC_{\theta}} u_{1\sigma}^2e^{-\frac{\sigma^2}{16}}\, d\sigma\Big)
\end{split}
\end{equation*}
To estimate the second integral in the last line above we use \eqref{eq-der-inter} and then find $L \gg 1$ big enough so that $\int_{|\sigma| \ge L} e^{-\frac{\sigma^2}{16}}\, d\sigma$ can be made as small as we want. Then we find $\tau_0 \ll -1$ sufficiently small so that \eqref{eq-L2-asymp2}  holds on $|\sigma| \le L$ for $\tau \le \tau_0$ and this we apply to the first integral in the last line in the previous estimate. Finally, for $\tau \le \tau_0$ we have
\begin{equation}
\label{eq-proj-error5}
\|\psi_2\|^{-2}\, \Big|\int_{\tau-1}^{\tau}\int (w_C)_\sigma \,\big (  \int_0^{\sigma}\frac{u_{1\sigma}^2}{u_1^2}\, d\sigma' \big ) \,
\psi_2\, d\mu\, d\tau'\Big| \leq  \frac{\epsilon}{|\tau|}\, \|a\|_{\cH,\infty}.
\end{equation}
Estimates  \eqref{eq-proj-error2}, \eqref{eq-proj-error1},  \eqref{eq-proj-error3},  \eqref{eq-proj-error4} and   \eqref{eq-proj-error5} lead to \eqref{eq-proj-error-main} as desired. \qed 

\sk
\sk 
 
\begin{step}
For every $\epsilon > 0$ there exists a $\tau_0 \ll -1$ so that for $\tau \le \tau_0$ we have
\begin{equation}
\label{eq-proj-nonlocal}
\|\psi_2\|^{-2}\,\Big|\int_{\tau-1}^{\tau}\langle \cE_{nl} - \frac{a(\tau)}{8|\tau|}\, \psi_2^2, \psi_2\rangle\, d\tau'\Big| \leq  \frac{\epsilon}{|\tau|}\, \|a\|_{\cH,\infty}.
\end{equation}
\end{step}

\sk
\sk

We recall that  $\langle \cE_{nl}, \psi_2\rangle =  \langle u_{2\sigma}\varphi_C (J_2 - J_1), \psi_2\rangle$. By \eqref{eqn-J12} 
this can be written as 
\begin{equation}
\label{eq-ENL-rewrite}
\begin{split}
\langle \cE_{nl}, \psi_2\rangle   &= -2\, \int u_{2\sigma}\varphi_C\, \psi_2\, \big ( \int_0^{\sigma} \frac{w_{\sigma\sigma}}{u_1}\, d\sigma' \big ) \, d\mu \\&\quad + 2\int u_{2\sigma}\varphi_C\, \psi_2  \big (   \int_0^{\sigma} \frac{u_{2\sigma\sigma}}{u_1u_2}\, w\, d\sigma'  \big ) \, d\mu
 \\
&= -2\bar{I} + 2I.
\end{split}
\end{equation}
Furthermore, we  can write
\[I := \int u_{2\sigma} \varphi_C \psi_2 \, \int_0^{\sigma}\frac{u_{2\sigma\sigma}}{u_1u_2}\, w\, d\sigma'\, d\mu =: I_1 + I_2\]
where
\[I_1 := \int u_{2\sigma}\varphi_C \psi_2 \int_0^{\sigma} \frac{u_{2\sigma\sigma}}{u_1u_2} w_C\, d\sigma' \, d\mu\]
and
\[I_2 := \int u_{2\sigma}\varphi_C \psi_2 \int_0^{\sigma} \frac{u_{2\sigma\sigma}}{u_1u_2} (1-\varphi_C) w\, d\sigma'\, d\mu.\]
To estimate $I_1$, we write $I_1=2 I_1|_{A_+}$ where   $I_1|_{A_+}$ denotes  
the same integral  $I_1$ restricted on $\supp \varphi_C \cap \{ \sigma \geq 0 \}$.  To  estimate $I_1|_{A_+}$,  we 
note that by Fubini's Theorem and \eqref{eq-der-inter} for  a small $\eta > 0$, we have 
\begin{equation}
\label{eq-I1-inter}
\begin{split}
\Big| {I_1}_{ |_{A_+}} \Big| &\le C(\theta)\, \int_0^{u_1^{-1}(\theta,\tau)}|u_{2\sigma\sigma}| |w_C|\,\int_{\sigma'}^{u_1^{-1}(\theta,\tau)} |u_{2\sigma}|\varphi_C |\psi_2| \, e^{-\frac{\sigma^2}{4}}\, d\sigma\, d\sigma' \\
&\le \frac{C(\theta)}{\sqrt{|\tau|}}\, \int_0^{u_1^{-1}(\theta,\tau)}|u_{2\sigma\sigma}| |w_C| \, e^{-\frac{(1-\eta)\sigma'^2}{4}} \, \int_{\sigma'}^{u_1^{-1}(\theta,\tau)} |\psi_2| \, e^{-\eta\sigma^2/4}\, d\sigma\, d\sigma'. 
\end{split}
\end{equation}
Integration by parts readily gives the bound 
\begin{equation*}
\begin{split}
\int_{\sigma'}^{u_1^{-1}(\theta,\tau)} |\psi_2| \, e^{-\eta\sigma^2/4}\, d\sigma &\le \int_{\sigma'}^{\infty} (\sigma^2 + 2) \, e^{-\eta \, \sigma^2/4}\, d\sigma \leq C \, \big ( 1  +  \sigma' ) \, e^{-\eta\,\sigma'^2/4}
\end{split}
\end{equation*}
which inserting in the above estimate yields 
\begin{equation*}
\begin{split}
\Big|  {I_1}_{ |_{A_+}} \Big| 
&\le \frac{C(\theta)}{\sqrt{|\tau|}}\, \int_0^{u_1^{-1}(\theta,\tau)} \,  |u_{2\sigma\sigma}||w_C| (\sigma+1)\, e^{-\frac{\sigma^2}{4}}\, d\sigma \\ 
&\le \frac{C(\theta)}{\sqrt{|\tau|}}\, \Big(\int_{\cC_{\theta}} u_{2\sigma\sigma}^2\, d\mu\Big)^{\frac12}\, \Big(\Big(\int w_C^2 \, \sigma^2\, d\mu\Big)^{\frac12} + \|w_C\|\Big) \\
&\le \frac{C(\theta)}{\sqrt{|\tau|}}\, \Big(\int_{\cC_{\theta}} u_{2\sigma\sigma}^2\, d\mu\Big)^{\frac12}\, \|w_C\|_{\cD}. 
\end{split}
\end{equation*}
Using  \eqref{eq-L2-asymp0} we conclude the bound 
\begin{equation*}
\label{eq-I1}
\int_{\tau-1}^{\tau} | I_1 |\, d\tau' = 2 \int_{\tau-1}^{\tau} | {I_1}_{ |_{A_+}}|\, d\tau' \le \frac{C(\theta)}{|\tau|^{\frac32}} \, \|a\|_{\cH,\infty} \leq 
\frac{\epsilon}{|\tau|}\, \|a\|_{\cH,\infty}.
\end{equation*}
To estimate $I_2$ we simply use 
\[|I_2| \le C(\theta)\Big(\int w^2 \chi_{D_{\theta}} \, d\mu\Big)^{\frac12} e^{-c_0(\theta)\, |\tau|}\]
which together  \eqref{eq-par-int} and \eqref{eq-transition-est} yields
\[\int_{\tau-1}^{\tau} |I_2|\, d\tau' \leq  \frac{\epsilon}{|\tau|}\, \|a\|_{\cH,\infty}.\]
Combining these two estimates for $I_1$ and $I_2$ finally gives us the bound 
\begin{equation}
\label{eq-I}
\int_{\tau-1}^{\tau} |I|\, d\tau' \leq \frac{\epsilon}{|\tau|}\, \|a\|_{\cH,\infty}.
\end{equation}
It remains to estimate $\bar I$. For that we first integrate by parts using that $w_\sigma(0,\tau)=0$, to obtain 
\begin{equation*}
\label{eq-I-term}
\begin{split}
\bar{I} &:= \int u_{2\sigma} \varphi_C \psi_2\, \int_0^{\sigma} \frac{w_{\sigma\sigma}}{u_1}\, d\sigma'\, d\mu \\
&= \int u_{2\sigma}\varphi_C \psi_2 \int_0^{\sigma} \frac{w_{\sigma} u_{1\sigma}}{u_1^2}\, d\sigma' d\mu + \int u_{2\sigma}\varphi_C \psi_2 \frac{w_{\sigma}}{u_1}\, d\mu  \\
&= \int u_{2\sigma}\varphi_C \psi_2 \int_0^{\sigma} \frac{w_{\sigma} u_{1\sigma}}{u_1^2}\, d\sigma' d\mu 
 -   \int \frac{u_{2\sigma}\psi_2}{u_1} (\varphi_C)_{\sigma} w\, d\mu  +  \int\frac{u_{2\sigma} \psi_2}{u_1}\, (w_C)_{\sigma}\, d\mu.
\end{split} 
\end{equation*}
Note that similarly to estimating $I_1$ term above (see \eqref{eq-I1-inter}) we have 
\begin{equation*}
\Big|\int u_{2\sigma}\varphi_C \psi_2 \, \int_0^{\sigma} \frac{w_{\sigma} u_{1\sigma}}{u_1^2}\, d\sigma'\, d\mu\Big| \le \frac{C(\theta)}{\sqrt{|\tau|}}\, \Big(\int_{\cC_{\theta}} u_{2\sigma}^2\, d\mu\Big)^{\frac12}\, (\|w_C\|_{\cD} + \|(w_C)_{\sigma\sigma}\|_{\cH} )
\end{equation*}
hence, using \eqref{eq-par-int}, \eqref{eq-L2-asymp0}, \eqref{eq-sec-der-energy}, \eqref{eq-we-need1} and \eqref{eq-we-need2}  we obtain
\begin{equation}
\label{eq-I-term-1}
\Big|\int_{\tau-1}^{\tau} \int u_{2\sigma}\varphi_C \psi_2 \, \int_0^{\sigma} \frac{w_{\sigma} u_{1\sigma}}{u_1^2}\, d\sigma'\, d\mu\, d\tau'\Big| \leq  \frac{\epsilon}{|\tau|}\, \|a\|_{\cH,\infty}.
\end{equation}
Similarly,  we obtain
\begin{equation}
\label{eq-I-term-2}
\Big|\int_{\tau-1}^{\tau} \int \frac{u_{2\sigma}\psi_2}{u_1} (\varphi_C)_\sigma\, w\, d\mu\, d\tau'\Big| \leq  \frac{\epsilon}{|\tau|}\, \|a\|_{\cH,\infty}.
\end{equation}
%
%
For the last term that gives  $\bar I$ we need to estimate 
\begin{equation}
\begin{split}
\label{nonlocal-main}
\int \frac{u_{2\sigma} \psi_2}{u_1}\, (w_C)_\sigma\, d\mu = \frac{1}{\sqrt{2}} &\int \frac{u_{2\sigma}\psi_2(\sqrt{2} - u_1)}{u_1}\, (w_C)_\sigma\, d\mu \\&\qquad +  \frac{1}{\sqrt{2}}\, \int u_{2\sigma}\psi_2 (w_C)_\sigma\, d\mu. 
\end{split}
\end{equation}
For  first term in \eqref{nonlocal-main}, after applying   Cauchy-Schwarz inequality, \eqref{eq-der-inter} and using Lemma 4.12 from \cite{ADS1}, \eqref{eq-sec-der-energy}, \eqref{eq-we-need1}, \eqref{eq-we-need2} and \eqref{eq-L2-asymp1} we obtain for $\tau \leq \tau_0 \ll -1$ the bound 
\begin{equation}
\label{equation-nonlocal-main1}
\begin{split}
&\Big|\int \frac{u_{2\sigma}\psi_2(\sqrt{2} - u_1)}{u_1}\, (w_C)_\sigma\, d\mu \Big| \le \\
&\le 
\frac{C(\theta)}{\sqrt{|\tau|}}\, (\|(w_C)_\sigma\|_{\cH} + \|\sigma\, (w_C)_\sigma\|_{\cH})\, (\|u_1 - \sqrt{2}\|_{\cH} + \|\sigma\,(u_1 - \sqrt{2})\, \|_{\cH})\\
&\le \frac{C(\theta)}{\sqrt{|\tau|}}\, (\|w_{C}\|_{\cD} + \|(w_C)_{\sigma\sigma}\|_{\cH})\, (\|u_1 - \sqrt{2}\|_{\cD}) \\
&\le \frac{C(\theta)}{\sqrt{|\tau|}}\, \|w_C\|_{\cD} \,\,\|u_1 - \sqrt{2}\|_{\cD} \le \frac{\epsilon}{|\tau|}\, \|a\|_{\cH,\infty}.
\end{split}
\end{equation}
For the second  term on the right hand side in \eqref{nonlocal-main} we have
\begin{equation}
\label{eq-I-term-6}
\begin{split}
\frac{1}{\sqrt{2}}\, &\int u_{2\sigma} (w_C)_\sigma\psi_2\, d\mu \\&= \frac{1}{\sqrt{2}} \int\Big(u_2 - \sqrt{2} + \frac{\sqrt{2}\psi_{2}}{8|\tau|}\Big)_{\sigma}\, (w_C)_\sigma\,\psi_2\, d\mu - \int \frac{\psi_{2\sigma}}{8|\tau|}\, \psi_2 (w_C)_\sigma\, d\mu \\
&= \frac{1}{\sqrt{2}} \int\Big(u_2 - \sqrt{2} + \frac{\sqrt{2}\psi_{2}}{8|\tau|}\Big)_{\sigma}\, (w_C)_\sigma\,\psi_2\, d\mu\\& \quad - \frac{1}{8|\tau|} \int \psi_{2\sigma}\psi_2 \hat{w}_{C\sigma}\, d\mu - \frac{a(\tau)}{8|\tau|}\,\int \psi_{2\sigma}^2 \psi_2\, d\mu.
\end{split}
\end{equation}
Note that by Cauchy-Schwarz inequality, Lemma 4.12 in \cite{ADS1},  \eqref{eq-par-int}, \eqref{eq-L2-asymp0} and \eqref{eq-we-need2} we have
\begin{equation}
\label{eq-I-term-4}
\|\psi_2\|^{-1}\,\Big|\frac{1}{\sqrt{2}}\int_{\tau-1}^{\tau} \int\Big(u_2 - \sqrt{2} + \frac{\sqrt{2}\psi_{2}}{8|\tau|}\Big)_{\sigma}\, (w_C)_\sigma\,\psi_2\, d\mu\, d\tau'\Big| \leq  \frac{\epsilon}{|\tau|}\, \|a\|_{\cH,\infty}.
\end{equation}
Also, by  Cauchy-Schwarz inequality and \eqref{eq-par-int} we get
\begin{equation}
\label{eq-I-term-5}
\Big|\int_{\tau-1}^{\tau} \frac{1}{8|\tau'|} \, \int \psi_{2\sigma}\psi_2 \, \hat{w}_{C\sigma}\, d\mu\, d\tau'\Big| \leq  \frac{\epsilon}{|\tau|}\, \|a\|_{\cH,\infty}
\end{equation}
and using   $\int \psi_{2\sigma}^2 \psi_2\, d\mu = 4\|\psi_2\|^2$ and that $\langle \psi_2^2, \psi_2\rangle = 8\, \|\psi_2\|^2$ we have 
\[\int \psi_{2\sigma}^2\psi_2\, d\mu = \frac12 \langle \psi_2^2,\psi_2\rangle.\]
The last identity   and estimates \eqref{eq-I-term-1}, \eqref{eq-I-term-2}, \eqref{nonlocal-main}, \eqref{equation-nonlocal-main1}, \eqref{eq-I-term-6}, \eqref{eq-I-term-4} and \eqref{eq-I-term-5}  yield
\be\label{eqn-barI}
\Big|\int_{\tau-1}^{\tau} \Big(-2\bar{I} - \frac{a(\tau)}{8|\tau|}\, \langle \psi_2^2, \psi_2\rangle\Big)\,d\tau'\Big| \leq  \frac{\epsilon}{|\tau|}\, \|a\|_{\cH,\infty}.\ee 
Combining this with \eqref{eq-ENL-rewrite}, \eqref{eq-I} and \eqref{eqn-barI} leads to \eqref{eq-proj-nonlocal} as desired.

\sk
\sk 
Finally, inserting  in \eqref{eq-F-tau} the main estimates from the three steps,   \eqref{eq-proj-error-cut-off}, \eqref{eq-proj-error-main} and \eqref{eq-proj-nonlocal} concludes the proof of Claim \ref{claim-Fs}.
\end{proof}
\smallskip

The proof of Theorem \ref{thm-main} is now complete.
\end{proof}

%
%
%
%
%
%
%



\appendix

\appendix

\section{A priori bounds for rotationally symmetric data}
\label{sec-appendix}

In this section we will establish some preliminary  a'priori bounds for rotationally symmetric solutions. 
We assume throughout this section that 
$u(\cdot, \tau)$ is a solution of \eqref{eq-u}. We recall that we have denoted by  $\sigma_{\pm}(\tau)$  the points of maximal scalar curvature of
our rescaled solution  solution $\bar M_\tau$. For any $\theta \in (\sqrt{2},0)$, let us recall the definition of the {\em cylindrical region} (in un-rescaled coordinates), that is 
$$\cC_{\theta} := \{ (\sigma,\tau): \,\,   u(\sigma,\tau) \geq \frac{\theta}{4}  \}.$$

\subsection{Derivative bounds in the cylindrical region} 

Results in \cite{ADS3} and monotonicity of $u_{\sigma}$ immediately imply that 
for every $ \theta \in (0, \sqrt{2})$   there exist  $C(\theta)$, $c(\theta)$ and $\tau_0 \ll -1$ such that for  all $\tau \le \tau_0$ we have
\be\label{eqn-est-der100}
(\sigma,\tau)\in \cC_\theta \implies  |u_{\sigma}| \le \frac{C(\theta)}{\sqrt{|\tau|}}.
\ee
It also follows that 
\begin{equation}
\label{eq-dist1-1}
(\sigma,\tau) \in \cC_{\theta} \implies |\sigma-\sigma_{\pm}(\tau)| \geq c(\theta) \, |\sigma_{\pm}(\tau)|,  \qquad \mbox{for}\,\, \tau \leq \tau_0 \ll -1.
\end{equation}
Moreover, using convexity of $u(\sigma,\tau)$ as in \cite{ADS1}, we have
\begin{equation}
\label{eq-conv-u}
|u_{\sigma}(\sigma,\tau)| \le \frac{C}{\sigma_+(\tau) - \sigma},
\end{equation}
if $u_{\sigma} \leq  0$. If $u_{\sigma} \ge 0$, similar estimate holds if we replace $\sigma_ +(\tau)$ by $|\sigma_-(\tau)|$.

\sk

We will next derive higher order derivative estimates which  hold away from the tips of our surface. 

\begin{lemma}
\label{lem-loc-est}
For any $\theta < \sqrt{2}$ there exist constants $C(\theta) > 0$ and $\tau_0=\tau_0(\theta)\ll -1$ so that the bounds 
\begin{equation}
\label{eq-der-inter}
|u_\sigma| + |u_{\sigma\sigma}| + |u_{\sigma\sigma\sigma}| \le \frac{C(\theta)}{\sqrt{|\tau|}}
\end{equation}
hold on $\cC_\theta$, for all $\tau \leq \tau_0$. 
\end{lemma}

\begin{proof}
We first notice that by \eqref{eq-dist1-1} we have that for any $\theta < \sqrt{2}$, there exists an $\alpha=\alpha(\theta) <1$ such that 
$$ (\sigma,\tau) \in  \cC_{\theta} \implies \sigma  \in [\alpha\sigma_-(\tau), \alpha \, \sigma_+(\tau)]  \quad \mbox{and} \quad |\sigma_{\pm} - \sigma | \geq (1-\alpha) \, |\sigma_{\pm}(\tau)|.$$

The bound on $|u_\sigma|$ readily follows by \eqref{eqn-est-der100}.  
To obtain higher order derivative estimates on $u$ we first differentiate the evolution equation \eqref{eq-u} with respect to $\sigma$.  If we write $\ud := u_{\sigma}$ then we obtain
\[\frac{\pd \ud}{\pd\tau} = \ud_{\sigma\sigma} + \frac{\ud (1 - \ud^2)}{u^2}.\]
The vector fields $\partial_{\tau}$ and $\partial_{\sigma}$ do not commute. To overcome this we introduce commuting variables as in Section \ref{sec-outline}. Then the equation becomes
\[
\frac{\pd \ud}{\pd\tau} = \ud_{\sigma\sigma} -\frac{\sigma}{2}\ud_{\sigma} - J(\sigma,\tau) \, \ud_{\sigma} + \frac{\ud (1 - \ud^2)}{u^2}
\]
where $J(\sigma,\tau)$ is given by \eqref{eqn-defn-J}. 
We will localize the proof of our desired estimate \eqref{eq-der-inter} by introducing the following change of variables.  Assume with no loss of generality that $\sigma_0 > 0$. Given a point $(\sigma_0, \tau_0)$ in space-time with $\sigma_0 \leq \alpha \, \sigma_+(\tau_0)$, we let
\[
\bar{\ud}(\eta,\bar{\tau}) := \ud\bigl(\sigma_0e^{\bar{\tau}/2} + \eta, \tau_0 + \bar{\tau}\bigr).
\]
If we choose $-\tau_0$ large enough, depending on $\alpha\in(0,1)$, then this function is defined on the rectangle
\[
\cQ := \{(\eta,\bar{\tau})\,\,\,|\,\,\, |\eta| \le 1, -1 \le \bar{\tau} \le 0\}.
\]
To see this, recall that $s_+(t)$ for the unrescaled  flow is monotonically decreasing. Hence   $e^{-\tau/2} \, \sigma_+(\tau) = s(-e^{-\tau})$,  is a decreasing function of $\tau$, and thus
\begin{equation}
\label{eq-comp-diam1}
\sigma_+(\tau_0) \le e^{-\bar{\tau}/2} \sigma_+(\tau_0 + \bar{\tau}) \quad 
\text{ for $\bar{\tau}\in [-1,0]$.}
\end{equation}
For any $\alpha<1$, we choose $\alpha'= (1+\alpha)/2$.  We also choose $\tau_0(\alpha)$ so that
\[
\alpha \, \sigma_+(\tau') +1 \leq \alpha' \, \sigma_+(\tau') \quad \text{for all $\tau'<\tau_0$}.
\]
For any $(\eta, \bar{\tau}) \in \cQ$ we then get
\[
\sigma_0 e^{\bar{\tau}/2} + \eta \leq \alpha \sigma_+(\tau_0) e^{\bar{\tau}/2} + 1 \leq \alpha\sigma_+(\tau_0+\bar{\tau}) + 1 \leq \alpha'\sigma_+(\tau_0+\bar{\tau}).
\]
It follows that $\bar{\ud}(\eta,\bar{\tau}) = \ud(\sigma_0 \, e^{\bar{\tau}/2}+\eta, \tau_0+\bar{\tau})$ is indeed defined on $\cQ$.

\sk
On the other hand, a  computation shows that $\bar{\ud}$ satisfies
\[
\frac{\pd\bar{\ud}}{\pd\bar{\tau}}
= \bar{\ud}_{\eta\eta} - \eta \frac{\bar{\ud}_{\eta}}{2} - J(\sigma,\tau_0+\bar{\tau}) \, \bar{\ud}_{\eta} + \frac{\bar{\ud} (1 - \bar{\ud}^2)}{u^2}
\]
which can be written as 
\begin{equation}
\label{eq-vbar-evolution}
\frac{\pd\bar{\ud}}{\pd\tau}
= a(\eta,\bar{\tau},\bar{\ud},\bar{\ud}_{\eta})\, \bar{\ud}_{\eta\eta} + b(\eta,\bar{\tau},\bar{\ud},\bar{\ud}_{\eta}),
\end{equation}
where
\[
a(\eta,\bar{\tau},\bar{\ud},p) =1, 
\qquad
b(\eta,\bar{\tau},\bar{\ud},p) = \frac{\bar{\ud}\, (1-\bar{\ud}^2)}{u^2} - \frac{\eta}{2}\, p - J(\bar \sigma,\tau_0+\bar{\tau})\, p,
\]
and 
\begin{equation}
\label{eq-JJ}
J(\bar \sigma,\tau_0+\bar{\tau}) =  2\, \frac{\bar z}{u(\bar \sigma,\tau_0+\bar{\tau})}  + 2\int_0^{\bar \sigma} \frac{u_{\sigma'}^2}{u^2}\, d\sigma',
\quad \mbox{for} \,\, \bar \sigma:= \sigma_0 \, e^{\bar{\tau}/2}+\eta.
\end{equation}
Estimate \eqref{eq-conv-u} combined with $u(0, \tau) = \sqrt{2(n-1)}+\delta(\tau)$ tell us that on the rectangle $\cQ$ we can bound $\bar{\ud} = u_{\sigma}(\sigma_0e^{\bar{\tau}/2} +\eta,\tau_0+\bar{\tau})$ by
\[
|\bar{\ud}|
\le \frac{C}{\sigma_+(\tau_0+\tau) - \sigma_0e^{\bar{\tau}/2}-\eta}
\le \frac{C}{\sigma_+(\tau_0+\bar{\tau}) - \sigma_0 e^{\bar{\tau}/2} -1}.
\]
By \eqref{eq-comp-diam1} we have $e^{\bar{\tau}/2}\sigma_+(\tau_0) \le \sigma_+(\tau_0+\bar{\tau})$, which then implies
\[
|\bar{\ud}|\le \frac{C}{e^{\bar{\tau}/2}\sigma_+(\tau_0) - \sigma_0 e^{\bar{\tau}/2} -1}
\le \frac{C}{\sigma_+(\tau_0) - \sigma_0 - e^{-\bar{\tau}/2}},
\]
for $\bar{\tau}\in [-1,0]$.  Since $\sigma_0 \leq \alpha \, \sigma_+(\tau_0)$, we have
\[
\sigma_0+e^{-\bar{\tau}/2} \leq \sigma_0+e^{1/2}\leq \alpha'\sigma_+(\tau_0),
\]
with $\alpha'=(1+\alpha)/2$, assuming again that $-\tau_0$ is sufficiently large.  In the end we get the following estimate for $\bar \ud$ on the rectangle $\cQ$
\begin{equation}
\label{eq-v-bound-on-Q}
|\bar{\ud}| \leq \frac{C(\alpha)}{\sigma_+(\tau_0)}.
\end{equation}
We next apply this bound to the coefficients $a$ and $b$ in the equation \eqref{eq-vbar-evolution} for $\bar{\ud}$.  We have that $a(\eta,\tau,\bar{\ud},p) = 1$. Using  \eqref{eq-JJ}, \eqref{eq-v-bound-on-Q} and asymptotics proved in \cite{ADS3}, we get
\[|J(\bar \sigma,\tau_0+\bar{\tau})| \le  \frac{C(\alpha)}{\sigma_+(\tau_0+\bar{\tau}) - \sigma_0 e^{\bar{\tau}/2} - 1} + C(\alpha) \int_0^{\bar \sigma} \frac{d\sigma'}{(\sigma_+(\tau_0+\bar{\tau}) - \sigma')^2}.\]
Using that $\bar \sigma = \sigma_0 \, e^{\bar{\tau}/2}+\eta$, $\sigma_+(\tau_0+\bar{\tau}) \ge e^{\bar{\tau}/2}\, \sigma_+(\tau_0)$,  $\sigma_0 \leq \alpha \, \sigma_+(\tau_0)$ and $\sigma_+(\tau_0) \sim \sqrt{|\tau_0|}$, for $\tau_0 \ll -1$ (the latter   follows from our results in \cite{ADS3}) we have
\begin{equation*}
\begin{split}
|J(\bar \sigma,\tau_0+\bar{\tau})| &\le \frac{C(\alpha)}{(e^{\bar{\tau}/2}\, (\sigma_+(\tau_0) - \sigma_0) -1} + C(\alpha)\, \int_0^{\bar \sigma} \frac{d\sigma'}{(e^{\bar{\tau}/{2}} \sigma_+(\tau_0) -\sigma')^2}   \\
&\le  \frac{C(\alpha)}{(e^{\bar{\tau}/2}\, (\sigma_+(\tau_0) - \sigma_0) -1} + \frac{C(\alpha)\, ( 1 + \alpha \, \sigma_+(\tau_0) e^{\bar{\tau}/2})}{(e^{\bar{\tau}/2}\sigma_+(\tau_0) (1 - \alpha) - 1)^2}\\
&\le \frac{C(\alpha)}{\sqrt{|\tau_0|}}.  
\end{split}
\end{equation*}
The above bound on $J(\bar \sigma,\tau_0+\bar{\tau})$, the lower bound $u \ge \theta$ and \eqref{eq-v-bound-on-Q} imply
\[
|b(\eta,\tau,\bar{\ud},p)| \le C\, (1+p^2) = C \, a(\eta,\tau,\bar{\ud},p)\, (1+p^2).
\]
As a consequence of these bounds on the coefficients $a$ and $b$ and   classical interior estimates   for equation \eqref{eq-vbar-evolution} (see \cite{LSU}), we obtain  
\[
|\bar{\ud}_{\eta}(0,0)| + |\bar{\ud}_{\eta\eta}(0,0)|
\le C_0\sup_{\cQ} |\bar{\ud}(\eta,\tau)|
\le \frac{C(\alpha)}{\sqrt{|\tau_0|}}.
\]
Finally, since $\bar{\ud}_\eta(0,0) = u_{\sigma\sigma}(\sigma_0,\tau_0)$, $\bar{\ud}_{\eta\eta}(0,0) = u_{\sigma\sigma\sigma}(\sigma_0,\tau_0)$ 
and $\alpha = \alpha(\theta)$ this completes the proof of Lemma~\ref{lem-loc-est}.
\end{proof}

\subsection{The concavity of $u^2$}

Our aim in this section is to prove the following result which will allow us to better estimate $u(\cdot,\tau)$ and its derivatives in the collar region. 

\begin{prop}\label{claim-2} Let $(S^3, g(t))$ be a compact, ancient solution to the Ricci flow on $S^3$. Then there exist $\tau_0 \ll -1$ and $L \gg 1$ so that the function $u^2$ is concave on $u \ge {L}/{\sqrt{|\tau|}}$, for $\tau \leq \tau_0 \ll -1$. 
\end{prop} 

The proof of this Proposition combines a contradiction argument based on scaling and the following maximum principle Lemma.

\begin{lemma}
\label{lemma-decreasing}
Under the assumptions of  Proposition \ref{claim-2}, there exist a $\tau_0 \ll -1$  such that if $\big ( (u^2)_{\sigma\sigma}\big ) _{\max} > 0$ is attained in $\{u \geq {L}/{\sqrt{|\tau|}}\}$, for a sufficiently large number  $L$, then $$\frac{d}{d\tau} \max (u^2)_{\sigma\sigma} \le 0.$$ 
\end{lemma}

\begin{proof}
Since $(u^2)_{\sigma\sigma}$ is scaling invariant quantity, we will work in the original variables $(s, t, \psi(s,t))$. Define $Q(s,t) := \psi^2(s,t)$. Note that $(\psi^2)_{ss} = (u^2)_{\sigma\sigma}$. Hence, it is sufficient to show that $\frac{d}{dt} \max (\psi^2)_{ss} \le 0$. 

It is easy to compute that in commuting variables $s$ and $t$ we have
\[Q_t = Q_{ss} - 2 Q_s\,  \int_0^s \frac{\psi_{ss}}{\psi}\, ds -2\]
implying,
\[(Q_{ss})_t = (Q_{ss})_{ss} - \left(J + \frac{Q_s}{\psi^2}\right)\, (Q_{ss})_s + \frac{4\psi_{ss}}{\psi^3}\, (Q_s^2 - Q Q_{ss}),\]
where $J = 2\int_0^s \frac{\psi_{ss}}{\psi}\, ds'$.
We also have that  $Q_s^2 - Q Q_{ss} = 2\psi^2 \psi_s^2 - 2\psi^3 \psi_{ss} \ge 0$, yielding that at the maximum of $Q_{ss}$ we have 
\[\frac{d}{d\tau} \max Q_{ss} \le 0.\]
If the maximum of $Q_{ss}$ is attained in the set $\{u > {L}/{\sqrt{|\tau|}}\}$, since $(Q_{ss})_s = 0$ at the maximum point and $\psi > 0$ there, we conclude that $\frac{d}{d\tau} \max (u^2)_{\sigma\sigma} \le 0$. 
\end{proof}

\begin{proof}[Proof of Proposition \ref{claim-2}]
Denote by $q(\sigma,\tau) = u^2(\sigma,\tau)$ and recall that $q_{\sigma\sigma} = Q_{ss}$. We claim that at the boundary of the set $\{u > {L}/{\sqrt{|\tau|}}\}$ we have 
$q_{\sigma\sigma} < 0$. To see that, lets write
\[q_{\sigma\sigma} = u Y_u + Y = \rho \, Z_{\rho} + 2\, Z\]
where $Z(\rho,\tau) := Y(u, \tau)$, $\rho:= u\, \sqrt{|\tau|}$ so that the boundary  $u = {L}/{\sqrt{|\tau|}}$ corresponds to $\rho = L$. We know that $Z(L,\tau)$ converges, as $\tau\to -\infty$ to the Bryant soliton $Z_0(L)$, whose maximal scalar curvature is equal to one. On the other hand,  the asymptotics \eqref{eqn-Zasym} of $Z_0(L)$, imply that 
\[L Z_{0\rho}(L) + 2Z_0(L) = - 4 \, L^{-4} + o(L^{-4}) < 0, \qquad \mbox{for} \,\, L \gg1. \]
This means that for $L \gg 1$  and $\tau \le \tau_0 \ll -1$ we have that $q_{\sigma\sigma}$ is {\em negative at the boundary}, namely 
\[q_{\sigma\sigma} |_{u = \frac{L}{\sqrt{|\tau|}}} = L Z_{\rho}(L,\tau) + 2Z(L,\tau) < 0.\]

\sk
Next, we claim there exist $\tau_0 \ll -1$ and $L \gg 1$ so that $\max_{u  \ge \frac{L}{\sqrt{|\tau|}}} q_{\sigma\sigma} \le 0$. Assume the statement were not true. Since, $q_{\sigma\sigma} <0$ at the boundary, this means that  there exist sequences  $\tau_j \to -\infty$, $\sigma_j$ and $L_j\to \infty$ so that 
$$q_{\sigma\sigma}(\sigma_j,\tau_j) = \max_{u \ge \frac{L_j}{\sqrt{|\tau_j}}} q_{\sigma\sigma} > 0
\qquad \mbox{and} \qquad u(\sigma_j, \tau_j) > \frac{L_j}{\sqrt{|\tau_j|}}.$$
Lemma \ref{lemma-decreasing} then implies that
$q_{\sigma\sigma}(\sigma_j,\tau_j)   \ge c > 0$, for some uniform constant $c >0$. 
Since $q_{\sigma\sigma} = 2 ( u_{\sigma\sigma} + u_\sigma^2)$ and $u_{\sigma\sigma} \le 0$, we conclude  that $u_{\sigma}^2(\sigma_j,\tau_j) \ge c/2$, or expressed in the 
tip region variables, that $Y(u_j,\tau_j) \ge  c/2$, where $u_j > \frac{L_j}{\sqrt{|\tau_j|}}$.
Since $Y_u = 2 u_{\sigma\sigma} \le 0$ we conclude that \[Z(L_j,\tau_j) = Y\Big(\frac{L_j}{\sqrt{|\tau_j|}}, \tau_j\Big) \ge Y(u_j,\tau_j) \ge \frac c2.\]
Since for any $L$ large we have $L_j \ge L$,  for $j$ sufficiently large,   and since $Z(\rho,\tau)$ is decreasing in $\rho$, we have that $Z(L,\tau_j) \ge Z(L_j,\tau_j) \ge \frac c2$.
On the other hand, we have that the $\lim_{j\to\infty} Z(L,\tau_j) = Z_0(L) \approx  1/L^2$, provided  $L \gg1$. 
All these lead to a  contradiction, if we choose $L$ sufficiently large. 
\end{proof}

\subsection{Estimates in the Collar region} In this section we will prove
that any of our solutions satisfies the  sharp estimate  \eqref{eqn-good} in the collar region $\collar_{L,\theta}$,  
provided that $L \gg 1$ and $\tau \leq \tau_0 \ll -1$.  This estimate  played a crucial role in estimating error terms in the
entire Section \ref{sec-tip} dealing with  he tip region. We first show, in the next Proposition,  that our solutions   behave  geometrically as  cylinders  in the region  $u \geq L/\sqrt{|\tau|}$,  for $L \gg 1$ and $\tau \leq \tau_0 \ll -1$.

\begin{prop}\label{claim-1} Given an $\eta >0$ there exist $L \gg 1$ and $\tau_0 \ll -1$ such that for $\tau \le \tau_0$,
\be\label{eqn-cyl-estimate}
\frac{K_1}{K_0} = - \frac{u\, u_{\sigma\sigma}}{1-u_\sigma^2}  \leq \eta, \qquad \mbox{on}\,\,\, \, u \geq \frac{L}{\sqrt{|\tau|}}.
\ee
and moreover,
\be\label{eq-der-3}
u^2 \, |u_{\sigma\sigma\sigma}| < \eta.
\ee
\end{prop} 

\begin{proof}
We will use similar arguments to the ones we used to prove the analogous statement in \cite{ADS2}, based on the  
the following claim: 

\begin{claim} For every $\bar L > 0$ there exist an $L \gg 1$ and a $\tau_0 \ll -1$ so that
\begin{equation}
\label{eq-implication}
u(\sigma,\tau) \ge \frac{L}{\sqrt{|\tau|}} \,\, \implies \,\,  \dist_{g(\tau)}(p, p_{\tau}^k) \ge \frac{\bar L}{\sqrt{R(p,\tau)}}
\end{equation}
where $p$ is any  point in our manifold, corresponding to $\sigma$,  and $p_{\tau}^k, \, k=1,2$ is any of the two tip  points 
where the  scalar curvature becomes maximal, corresponding to $\sigma_{\pm}(\tau)$.  
\end{claim} 

\begin{proof}[Proof of Claim] 

To show above claim we argue by contradiction.  Assume the claim is not correct, meaning there exist an $\bar L > 0$ and sequences $L_j\to \infty$, $\tau_j\to -\infty$ and $\sigma_j$, so that
\begin{equation}
\label{eq-close-tip}
u(\sigma_j,\tau_j)\ge \frac{L_j}{\sqrt{|\tau_j|}} \qquad \mbox{but} \qquad  \dist_{g(\tau_j)}(p_j, p_j^1) \le \frac{\bar L}{\sqrt{R(p_j,\tau_j)}}
\end{equation}
for say $k =1$,  where $p_j \in M$ corresponds to $\sigma_j$ and $p_j^k$ is a brief notation for $p_{\tau_j}^k$. Note that the distance between points $p_j$ and $p_j^1$ is measured with respect to metric $g(\cdot,\tau_j)$. Rescale the flow around $(p_j,\tau_j)$ by $\lambda_j := R(p_j,\tau_j)$, that is set $\tilde {g}_j(\cdot,\tau) = \lambda_j \, g(\cdot, \tau_j + \lambda_j^{-1} \tau)$. Then $R_{\tilde g_j}(p_j,0) = 1$ and by Perelman's compactness theorem for $\kappa$-solutions (see section 11 in \cite{Pe1}) we can extract a convergent subsequence $(M, g_j(\cdot,\tau), p_j)$ that converges to a complete $\kappa$-solution. Since the limit is complete, noncompact, by \cite{Br} we know it is either a shrinking round cylinder or a Bryant soliton. By \eqref{eq-close-tip} we have that the rescaled distance from $p_j$ to $p_j^1$ at time zero is $\widetilde{\dist}_j(p_j,p_j^1) \le \bar L$. Moreover at the tip point the scaling invariant quantity  ${\ds ({K_1}/{K_0})(p_j^1,\tau_j) = 1}$, 
yielding  that  the limiting metric is actually the Bryant soliton. 

By results in \cite{ADS3} we have $R(p_j^1,\tau_j) \sim |\tau_j|$. By \eqref{eq-close-tip} and Perelman's compactness theorem we have that $ 1\le R_{\tilde g_j}(p_j^1,0) = \frac{R(p_j^1,\tau_j)}{R(p_j,\tau_j)} \le C$, for all $j \ge j_0$ and a uniform constant $C$. Hence, for $j \ge j_0$ we have that $R(p_j,\tau_j) \sim |\tau_j|$ as well. Furthermore, if $\sigma_j$ and $\sigma_j^1$ are corresponding to points $p_j$ and $p_j^1$, respectively, we have,
\begin{equation}
\label{eq-sigma1}
|\sigma_j - \sigma_j^1| \le \dist_{g(\tau_j)}(p_j,p_j^1) \le \frac{\bar L}{\sqrt{R(p_j,\tau_j)}} \le \frac{C \bar L}{\sqrt{|\tau_j|}}.
\end{equation}
On the other hand, using $|u_{\sigma}| \le 1$ we get
\begin{equation}
\label{eq-sigma2}
|\sigma_j - \sigma_j^1| = \int_0^{u(\sigma_j,\tau_j)} \frac{du}{|u_{\sigma}|} \ge u(\sigma_j,\tau_j) \ge \frac{L_j}{\sqrt{|\tau_j|}}.
\end{equation}
Combining \eqref{eq-sigma1} and \eqref{eq-sigma2}, if we let $j\to\infty$, yield contradiction. This concludes the proof of 
the claim.
\end{proof}
\sk

To prove \eqref{eqn-cyl-estimate},  due to \eqref{eq-implication} it is enough to prove the following statement: {\em for every $\eta > 0$ there exist $L \gg 1$ and $\tau_0 \ll -1$ such that for $\tau \le \tau_0$ we have
\[ \min \big ( \dist_{g(\tau)}(p,p_{\tau}^1), \dist_{g(\tau)}(p,p_{\tau}^2) )  \ge \frac{L}{\sqrt{R(p,\tau)}} \,\, \implies \,\, \frac{K_1}{K_0}(p,\tau) \le \eta.\]} 

To prove this  we argue again by contradiction. Assume there exist $\eta > 0$ and sequences $L_j\to \infty$, $\tau_j \to -\infty$ and $p_j$ so that
\[\frac{K_1}{K_0}(p_j,\tau_j) \ge \eta  \quad \mbox{but} \quad \min \big ( \dist_{\tau_j}(p_j,p_j^1), \dist_{\tau_j}(p_j,p_j^2) \big ) \ge \frac{L_j}{\sqrt{R(p_j,\tau_j)}}.\]
Rescale the metric around $(p_j,\tau_j)$ by $\lambda_j := R(p_j,\tau_j)$, that is  consider a sequence of rescaled metrics $\tilde{g}_j(\cdot,\tau) = \lambda_j \, g(\cdot, \tau_j + \tau \lambda_j^{-1})$.  Then the rescaled distance satisfies $\widetilde{\dist}_j(p_j,p_j^k) \ge L_j$,
 for both $k =1$ and $k = 2$. Using Perelman's compactness theorem for $\kappa$-solutions, the fact that ${K_1}/{K_0}$ is scaling invariant quantity and Brendle's classification result of complete noncompact $\kappa$-solutions (see \cite{Br}), after passing to a subsequence we conclude that the  sequence of rescaled solutions $(M, \tilde g_j(\cdot,\tau), p_j)$ subconverges to a Bryant soliton. Since the $\lim_{j\to \infty} \widetilde{\dist}_j(p_j,p_j^k) = \infty$ for both $k = 1$ and $k =2$ and since $K_0, K_1 \ge 0$ for our rotationally symmetric solution,  the splitting theorem implies that  the limit has to split off a line, implying that the limit has to be a cylinder, which contradicts the above fact the limit has to be the Bryant soliton at the same time. 
 
 \sk  Finally, the bound \eqref{eq-der-3} follows by similar arguments using again that $u^2\, u_{\sigma\sigma\sigma}$ is
 a scaling invariant quantity.  This finishes the proof of the proposition. 
\end{proof}

\sk 

Based on Proposition \ref{claim-2} we next show the following crucial for our purposes sharp bound. This bound 
is  extensively used in Section \ref{sec-tip}. 

\begin{lemma}\label{lemma-crucial-estimate} Fix $\eta >0$ small. There exists 
$\theta, L$   and $\tau_0 \ll -1$ such  that 
\be\label{eqn-good}
| 1 + \frac{ \sigma \, u u_\sigma}2| < \eta 
\ee holds on $\collar_{L,\theta}$, for $\tau \leq \tau_0 \ll -1$. 

\end{lemma} 
\begin{proof}

Having Proposition \ref{claim-2} and unique asymptotics that we proved in \cite{ADS3}, the proof of the first estimate  is identical to the proof of Corollary 4.7 in \cite{ADS2}. 

\end{proof}


\subsection{Energy estimate for the linear equation} In this final section we prove 
 the following standard energy estimate adopted to  ancient solutions.    

\begin{lemma}
\label{lemma-energy}
Let $w$ be a compactly supported, ancient solution to
\[w_{\tau} = \cL[w] + g, \qquad \mbox{on} \,\, \,\, \R \times (-\infty, \tau_0]. \]
Then,  there exists a uniform constant $C$ so that
\bee
\begin{split}
\sup_{\tau \le \tau_0} \int w_{\sigma}^2\, d\mu &+ \sup_{\tau \le \tau_0} \int_{\tau-1}^{\tau} \int w_{\sigma\sigma}^2\,d\mu \, d\tau' \\
&\le C\, 
\Big( \sup_{\tau \le \tau_0}\int_{\tau-1}^{\tau} \int w^2\, d\mu \, d\tau' + \sup_{\tau \le \tau_0} \int_{\tau-1}^{\tau}\int g^2\, d\mu \, d\tau' \Big).
\end{split}
\eee 
\end{lemma}

\begin{proof}
If we multiply the linear equation by $w$ and integrate it by parts, we obtain
\[\frac12\frac{d}{d\tau} \int w^2\, d\mu = -\int w_{\sigma}^2\, d\mu + \int w^2\, d\mu + \int gw \, d\mu.\]
For any number $\tau\in(-\infty,\tau_0)$, set $\eta(\tau') = \tau'-\tau + 2$ so that $0 \le \eta(\tau') \le 2$ for $\tau' \in [\tau-2,\tau]$. Then, for any $\tau'\in [\tau-2,\tau]$, we have
\bee
\begin{split}
\frac 12 \frac{d}{d\tau'}\Big(\eta(\tau') \int w^2\, d\mu\Big) = &-\eta(\tau') \int w_{\sigma}^2\, d\mu \\
&+( \eta(\tau') + \frac 12)\,\int w^2\, d\mu + \eta(\tau')\int gw \, d\mu.
\end{split}
\eee 
If we integrate it from $\tau-2$ to $\tau'\in [\tau-1,\tau]$, for every $\tau \le \tau_0$ and Cauchy-Schwarz inequality for the last term on the right hand side, we get
\begin{equation}
\label{eq-energy-first-der}
\begin{split}
\sup_{\tau\le \tau_0} \int w^2\, d\mu &+ \sup_{\tau \le \tau_0} \int_{\tau-1}^{\tau} \int w_{\sigma}^2 \, d\mu d\tau' \\
& \le C \,  \Big( \sup_{\tau \le \tau_0} \int_{\tau-1}^{\tau}\int w^2\, d\mu d\tau' + \sup_{\tau \le \tau_0}\int_{\tau-1}^{\tau} \int g^2\, d\mu d\tau' \Big)
\end{split}
\end{equation}
for a uniform constant $C$.

\sk

Next, lets multiply the equation for $w$ by $w_{\sigma\sigma}$ and integrate by parts to get
\begin{equation*}
\begin{split}
\frac12 \frac{d}{d\tau} \int w_{\sigma}^2\, d\mu &+ \int (w_{\sigma\sigma}^2 - \sigma w_{\sigma} w_{\sigma\sigma} + \frac14 \sigma^2 w_{\sigma}^2)\, d\mu \\&=  \int (\frac12 \sigma w w_{\sigma} + \frac12 \sigma w_{\sigma} g  -  w w_{\sigma\sigma} - g w_{\sigma\sigma})\, d\mu.
\end{split} 
\end{equation*}
If we multiply the previous equation by the function $0 \le \eta(\tau) \le 2$  which is defined as above, while using 
\[\int \sigma w_{\sigma} w_{\sigma\sigma}\, d\mu = -\frac12 \int w_{\sigma}^2\, d\mu + \frac14 \int \sigma^2 w_{\sigma}^2\, d\mu\]
and Cauchy-Schwarz,  we obtain for e
\begin{equation*}
\begin{split}
\frac{d}{d\tau'}\Big(\eta (\tau') \int w_{\sigma}^2\, d\mu\Big) &+ \eta(\tau')  \int w_{\sigma\sigma}^2\, d\mu + \frac12\, \eta(\tau') 
 \int w_{\sigma}^2\, d\mu \\
&\le \epsilon \int \sigma^2 w_{\sigma}^2\, d\mu + C_{\epsilon} \int w^2\, d\mu + C_{\epsilon}\int g^2\, d\mu + C\int w_{\sigma}^2\, d\mu. 
\end{split} 
\end{equation*}
Choose $\epsilon >0$ sufficiently small. Using that the operator $f \to \sigma f$ is bounded from $\cD$ to $\cH$, by choosing  small enough, if we integrate the 
previous estimate for  $\tau' \in [\tau-2, \tau]$, for every $\tau \le \tau_0$, similarly as above we get
\bee
\begin{split}
\sup_{\tau \le \tau_0} \int w_{\sigma}^2\, d\mu &+ \sup_{\tau \le \tau_0} \int_{\tau-1}^{\tau}\int w_{\sigma\sigma}^2\, d\mu \\&\le C\sup_{\tau\le \tau_0} \int_{\tau-1}^{\tau} \int (w^2 + w_{\sigma}^2)\, d\mu + \sup_{\tau\le\tau_0}\int_{\tau-1}^{\tau}\int g^2\, d\mu.
\end{split}
\eee
Combining this estimate with \eqref{eq-energy-first-der} concludes the proof of the Lemma.
\end{proof}


\end{document}